\documentclass[final]{siamltex}
\usepackage{amssymb,amsmath,version}
\usepackage{amssymb,version,graphicx,fancybox,mathrsfs,multirow}

\usepackage{url,hyperref}
\usepackage[notcite,notref]{showkeys}
\usepackage[ruled,linesnumbered]{algorithm2e}
\usepackage{verbatim}
\usepackage{color}
\usepackage{cases}
\usepackage{mathtools}

\usepackage{graphicx}
\usepackage{epstopdf}
\usepackage{subcaption}

\usepackage[notref,notcite]{showkeys} % To see crossreferences.
\newtheorem{remark}{Remark}[section]
  \newtheorem{assumption}{Assumption}[section]

\renewcommand{\d}{\mathrm{d}}

\newcommand{\R}{\mathbb{R}}

\newcommand{\fy}{\varphi}

\def\II{(\Omega)}

\def\al{\alpha}

\def\Dal{\partial_t^\alpha}
\def\bDal{\bar \partial_\tau}

\def\T{\mathcal{T}}

\def\Q{\mathcal{Q}}
\def\DK{\mathcal{D}(K)}

\def\Ih{\mathcal{I}_h}

\allowdisplaybreaks[1]
\begin{document}

\title{Identification of potential in diffusion equations from terminal observation: analysis and discrete approximation}
%\thanks{The research of the first and third author is supported by Hong Kong RGC grant No... }}

\author{Zhengqi Zhang\thanks{Department of Applied Mathematics, The Hong Kong Polytechnic University, Kowloon, Hong Kong.
Email address: \texttt{zhengqi.zhang@connect.polyu.hk}}
\and Zhidong Zhang\thanks{School of Mathematics (Zhuhai), Sun Yat-sen University, Zhuhai 519082, Guangdong, China.
Email address: \texttt{zhangzhidong@mail.sysu.edu.cn}}
\and Zhi Zhou\thanks{Department of Applied Mathematics, The Hong Kong Polytechnic University, Kowloon, Hong Kong.
Email address: \texttt{zhizhou@polyu.edu.hk}}}
\date{\today}

\maketitle
\begin{abstract}
The aim of this paper is to study  the recovery of a spatially dependent potential in a (sub)diffusion equation from overposed final time data. We construct a monotone operator one of whose fixed points is the unknown potential. The uniqueness of the identification is theoretically verified by using the monotonicity of the operator and a fixed point argument. Moreover, we show a  conditional stability in Hilbert spaces under some suitable conditions on the problem data. Next, a completely discrete scheme is developed, by using Galerkin finite element method in space and finite difference method in time, and then a fixed point iteration is applied to reconstruct the potential. We prove the linear convergence of the iterative algorithm by the contraction mapping theorem, and present a thorough error analysis for the reconstructed potential. Our derived \textsl{a priori} error estimate provides a guideline to choose discretization parameters according to the noise level. The analysis relies heavily on some suitable nonstandard error estimates for the direct problem as well as the aforementioned conditional stability. Numerical experiments are provided to illustrate and complement our theoretical analysis.
\\


%{\bf AMS subject classifications 2010:}

\end{abstract}

\setlength\abovedisplayskip{3.5pt}
\setlength\belowdisplayskip{3.5pt}

\section{Introduction}
This work is concerned with an inverse potential problem for the diffusion model with a space-dependent
potential and its rigorous numerical analysis. Let $\Omega\subset\mathbb{R}^d$ ($d=1,2,3$) be a convex
polyhedral domain with a boundary $\partial\Omega$. Fixing $T>0$ as the final time, we  consider the following  initial-boundary  value problem for the  diffusion model with $\alpha\in(0,1]$:
 \begin{equation}\label{eqn:pde}
 \begin{cases}
  \begin{aligned}
     \Dal u(x,t) - \Delta u(x,t) +q(x) u(x,t)&=f(x), &&(x,t)\in \Omega\times(0,T],\\
    u(x,t)&=b(x),&&(x,t)\in \partial\Omega\times(0,T],\\
    u(x,0)&=v(x),&&x\in\Omega,
  \end{aligned}
  \end{cases}
 \end{equation}
where $v$ denotes the initial condition, $b$ and $f$ are space-dependent boundary data and source term, respectively .
The function $q$ refers to the radiativity or
reaction coefficient or potential in the standard parabolic case ($\alpha=1$), dependent of the specific applications.
Throughout, we assume that the potential $q$ is space-dependent.

The notation $\Dal u$ denotes the conventional first-order derivative when $\alpha=1$, and the Djrbashian-Caputo fractional derivative in time $t$ for
$\alpha\in(0,1)$ \cite[p. 92]{KilbasSrivastavaTrujillo:2006}, namely, % and \cite[Section 2.3]{Jin:2021book})
\begin{equation*}
  \Dal u (t) =
 \begin{cases}
  \begin{aligned}
  &\partial_t u(t),&&\text{for}~~\alpha=1;\\
  &\frac{1}{\Gamma(1-\alpha)}\int_0^t (t-s)^{-\alpha}  u'(s)\ {\rm d}s,&&\text{for}~~\alpha\in(0,1);
  \end{aligned}
  \end{cases}
\end{equation*}
where $\Gamma(z)=\int_0^\infty s^{z-1}e^{-s}\d s$ (for $\Re(z)>0$) denotes Euler's Gamma function.
The fractional derivative $\partial_t^\alpha u$ recovers the usual first order derivative $u'$
as the order $\alpha\to1^-$ for a sufficiently smooth function $u$.
The model \eqref{eqn:pde} with $\alpha\in(0,1)$ has been drawing increasing
attention over the past several decades,
due to the extraordinary capability of the model for describing anomalously slow diffusion processes,
also known as subdiffusion. At a microscopical level, it
can be described by continuous time random walk, where the waiting time distribution between consecutive jumps is heavy
tailed with a divergent mean, in a manner similar to Brownian motion for the standard diffusion equation ($\alpha=1$).
The model \eqref{eqn:pde} can be viewed as the governing equation for the probability density function of the particle
appearing at certain time instance $t$ and space location $x$.  It has found many applications in physics, biology and finance etc.
One may consult the review \cite{MetzlerKlafter:2000} for physical motivation and an extensive list of applications.

In this work, we study the following \textbf{inverse potential problem} for the (sub)diffusion model \eqref{eqn:pde}:
setting appropriate problem data $v, f, b$ and measuring the final time data $g(x):=u(x,T;q^\dag)$,
then we aim to recover the unknown potential term $q^\dag(x)\in L^\infty(\Omega)$ such that
\begin{equation*}%\label{eqn:ip}
   u(x,T; q^\dag)=g(x)\quad \mbox{in }\Omega.
\end{equation*}
Here we denote the solution corresponding to the potential $q$ by $u(x,t;q)$.
We also consider the numerical reconstruction from a noisy data
\begin{equation*}
    g_\delta(x) = u(x,T; q^\dag) + \xi(x) \quad \mbox{in }\Omega,
\end{equation*}
and $\xi$ denotes the measurement noise.
The accuracy of the observational data $g_\delta$ is measured by the noise level $\|g_\delta - g\|_{C(\overline\Omega)} = \delta$.
This inverse potential problem arises in many practical applications, where
$q^\dag$ represents the radiativity coefficient in heat conduction \cite{YZ:2001} and perfusion coefficient in Pennes' bio-heat
equation in human physiology \cite{Pennes:1948}.

The theoretical analysis of inverse potential problem in diffusion equation from final time observational data has a long history, see e.g, \cite{Isakov:1991,
ChoulliYamamoto:1996,ChoulliYamamoto:1997,ChenJiangZou:2020,KlibanovLiZhang:2020} and the references therein.
In \cite{Isakov:1991} Isakov showed the uniqueness and (conditional) existence of the inverse potential problem for parabolic equations,
by developing a unique continuation principle and a constructive fixed point iteration. A similar strategy was then adopted in \cite{ZhangZhou:2017}
by Zhang and Zhou for a one-dimensional time-fractional subdiffusion model. Using the spectrum perturbation argument
(\cite[Lemma 2.2]{ZhangZhou:2017} and \cite{Trubowitz:1987}) they proved that the fixed point iteration
is a contraction, from which the uniqueness and existence followed immediately.
Choulli and Yamamoto  proved a generic
well-posedness result in a H\"older space \cite{ChoulliYamamoto:1996}, and then proved a conditional stability result in a Hilbert space setting \cite{ChoulliYamamoto:1997}
for sufficiently small $T$. By using refined
properties of two-parameter Mittag--Leffler functions, e.g., complete monotonicity and asymptotics, a similar result was proved in \cite{JinZhou:IP2021-a}
for the case that $\alpha\in(0,1)$.
Kaltenbacher and Rundell \cite{KaltenbacherRundell:2019} proved the invertibility
of the linearized map (of the direct problem) from the space $L^2(\Omega)$ to $H^2(\Omega)$ under the condition $u_0>0$ in $\Omega$
and $q\in L^\infty(\Omega)$ using a Paley-Wiener type result and a type of strong maximum principle.
In \cite{KR:2020}, they studied the recovery of several parameters simultaneously from overposed data consisting of $u(T)$.
Chen et al \cite{ChenJiangZou:2020} considered the observational data in $[T_0, T_1]\times\Omega$ for the parabolic equation,
and proved conditional stability of the inverse problem in negative Sobolev spaces.
Most recently, Jin et al \cite{JLQZ:2021} used the same observational data
and showed a weighted $L^2$ stability which  leads  to a H\"older type stability
in the standard $L^2$ norm  under a positivity condition.
We also refer interested readers to \cite{KianYamamoto:2019,MillerYamamoto:2013,KR:2020-b}
and references therein for the inverse potential problem for (sub)diffusion models from different types of observational data.

In this work, we construct an operator $K$ from the PDE \eqref{eqn:pde} as follows:
\begin{equation*}
 K\psi(x)=\frac{f(x)-\Dal u(x,T;\psi)+\Delta g(x)}{g(x)}.
\end{equation*}
From the observational data $g(x):=u(x,T;q)$, we see that the exact potential $q^\dag$ is one of the
fixed points of $K$.
%Then we can solve this inverse potential problem by investigating the fixed points of operator $K$.
We show the monotonicity of $K$ and use it to construct a decreasing sequence converging to one fixed point.
With this monotone sequence, we prove that there is at most one fixed point,
which immediately leads to the uniqueness result of the inverse problem (Theorem \ref{thm:uniqueness2}).
Besides, this argument also deduces a simple reconstruction algorithm.
Noting that such the operator $K$ has been considered in \cite{Isakov:1991, ZhangZhou:2017},
but the argument is substantially different.
For instance, in \cite{Isakov:1991}, the proof of uniqueness relied on a unique continuation
result of the solution $u$, while the proof in \cite{ZhangZhou:2017} used some inverse spectral estimates,
which are only valid in the one-dimensional case (cf. \cite[Lemma 2.2]{ZhangZhou:2017}).
In this work, our analysis mainly relies on the monotonicity of the operator $K$,
which works for convex polyhedral domains in higher dimensions.
This novel argument also provides the feasibility of applying the approach in other PDE models.
Moreover, under some conditions on problem data, we show a Lipschitz-type stability  in Hilbert spaces (Theorem \ref{thm:cond-stab})
$$ \| q_1 - q_2  \|_{L^2\II} \le C \| u(T;q_1) - u(T;q_2) \|_{H^2\II},\qquad \text{for all}~~ q_1, q_2 \in \Q. $$
The proof relies heavily on the smoothing properties and asymptotics of solution operators.
This conditional stability plays an essential role in the numerical analysis of our reconstruction algorithm with fully discretization in space and time.

The ill-posed nature of inverse potential problems usually poses big
challenges to construct accurate and stable numerical approximations.
Regularization, especially Tikhonov regularization, is
designed to overcome the ill-posed nature \cite{EnglKunischNeubauer:1989,YZ:2001,DengYuYang:2008,YangYuDeng:2008}.
In practical computation, one still needs to discretize the continuous regularized formulation
and hence introduces the discretization error.
See \cite{YZ:2001} for the convergence of
the discrete approximations in the parabolic case. However, the convergence rates of discrete approximations
are generally very challenging to obtain, due to the strong nonconvexity of the regularized functional, which
itself stems from the high degree nonlinearity of the parameter-to-state map.
So far there have been only very few error bounds on discrete approximations, even though
an optimal \textsl{a priori} estimate provides a useful guideline to choose suitable discretization parameters according to the noise level.
See \cite{JLQZ:2021} for an $L^2$ estimate under a positivity condition,
where the observational data is required to be known in $[T-\sigma, T]\times\Omega$ for some positive parameter $\sigma$.
Moreover, in case that $\alpha\in(0,1)$, due to the presence of the nonlocal fractional differential operator,
the subdiffusion model \eqref{eqn:pde} differs considerably from the normal diffusion problem. For example,
many powerful tools, e.g. energy argument and integration by parts formula, are not directly applicable,
and the solution has only limited spatial and temporal regularity, even for smooth problem data.
Both of them often result in additional difficulties to the mathematical and numerical analysis for both direct and inverse problems.
See a related inverse conductivity problem in \cite{Wang2010ErrorEO} and \cite{JinZhou:SICON} respectively for normal diffusion and subdiffusion model,
where the error estimate requires the observational data in $(0,T]\times\Omega$.

In this work, we discretize the continuous problem \eqref{eqn:pde} by using Galerkin finite element method
with conforming piecewise bilinear finite elements in space and backward Euler method in time for $\alpha=1$.
In case that $\alpha\in(0,1)$, we apply the convolution quadrature generated by backward Euler method for the time discretization.
To numerically reconstruct the potential from the noisy observation, we develop a constructive iteration
and show that it generates a sequence
linearly converging to a fixed point $q^*$, provided that $T$ is relatively large.
Besides, we show the following \textsl{a priori} error estimate for any parameter $\epsilon \in (0, \min(1,2-\frac{d}{2}))$ (Theorem \ref{thm:err-fully})
\begin{align*}
\| q^\dag - q^* \|_{L^2\II} &\le   \frac{c}{1- c T^{{-(1-\epsilon)\alpha}}} \Big(\frac{\delta}{h^2} + h + \tau \Big) \le c\Big(\frac{\delta}{h^2} + h+ \tau\Big)
\end{align*}
if $c T^{-(1-\epsilon)\alpha}\le c_0<1$ for some constant $c_0$. Here $h$ and $\tau$  denote the space mesh size and time step size respectively.
This \textsl{a priori} error estimate provides guidelines to choose discretization parameters $h$ and $\tau$ according to the noise level $\delta$.
For example, the choice $\tau = h = O(\delta^{\frac13})$
leads to a best convergence rate $O(\delta^{\frac13})$. This is fully supported by our numerical results in Section \ref{sec:numerics}.
Note that at the continuous level with exact data, the iteration converges without any requirement on the terminal time $T$ (Theorem \ref{thm:uniqueness2}).
However, at the discrete level with noisy data, our theory indicates that the accuracy of the numerical reconstruction requires that $T$ cannot be too small.
The necessity of this requirement on $T$ is supported by our numerical experiments.
In Figure \ref{fig:1D:err-ite}, we observe that for a small $T$, the iteration might converge to a limit far away from the exact potential.
Our analysis relies heavily on some nonstandard error estimates (in terms of data regularity) for the
direct problem as well as the aforementioned conditional stability.
The argument works for both normal diffusion equations ($\alpha=1$) and the subdiffusion equations ($0<\alpha<1$).

The rest of the paper is organized as follows. In Section \ref{sec:uni}, we provide some preliminary results
and show the uniqueness of the inverse potential problem by constructing a monotone fixed point iteration.
Then in Section \ref{sec:stability}, we prove a conditional stability of the inverse problem in Hilbert spaces by
using the smoothing properties and asymptotics of solution operators.
The numerical reconstruction with fully discretization is developed and analyzed in Section \ref{sec:fully},
where we show the linear convergence of the iterative algorithm and
establish \textsl{a priori} error estimates (in terms of discretization parameters and noise level) for the reconstructed potential.
Finally, in Section \ref{sec:numerics}, we present illustrative one- and two-dimensional numerical results to complement the analysis.

Now we conclude with some useful notations.
For any $k\geq 0$ and $p\geq1$, the
space $W^{k,p}(\Omega)$ denotes the standard Sobolev spaces of the $k$th order, and we write $H^k(\Omega)$,
when $p=2$. The notation $(\cdot,\cdot)$ denotes the $L^2(\Omega)$ inner product.
We use the Bochner spaces $W^{k,p}(0,T;B)$ etc, with $B$ being a
Banach space. Throughout, the notations $c$ and $C$, with or without a subscript, denote generic constants
which may change at each occurrence, but they are  always independent of space mesh size $h$, time step size $\tau$ and noise level $\delta$.

\section{Unique identification by the monotone iteration}\label{sec:uni}
The aim of this section is to investigate the uniqueness of the inverse potential problem.
Our approach is to propose a monotone operator which generates a pointwise
decreasing sequence converging to the exact potential.

To begin with, we collect some preliminary setting for the controllable
conditions $v ,b, f$, and the (unknown) exact potential $q^\dag$. Throughout,
we assume that the exact potential  $q^\dag$ belongs to the admissible set
\begin{equation}\label{admissible_q}
 q^\dag \in \mathcal Q:=\{\psi\in C(\overline \Omega): 0\le  \psi\le M_1\}.
\end{equation}

Now we recall the maximum principle for the diffusion model \eqref{eqn:pde}.
See \cite{Friedman:1958} for the normal diffusion, \cite{LuchkoYamamoto:2017} and \cite[Section 6.5]{Jin:2021book} for the subdiffusion.

\begin{lemma}\label{positive}
Let $q\in \mathcal{Q}$, $v, f \in L^2\II$ and $b\in H^{\frac32}(\partial\Omega)$ with $v, f, b \ge 0$ a.e. in $\Omega$.
Then the solution $u$ to  equation  \eqref{eqn:pde} satisfies $u \ge 0$ a.e. in $(0,T)\times \Omega$.
Moreover, if $v,b > 0$, then   $u > 0$ in $(0,T)\times \Omega$.
\end{lemma}

Now we present the solution representation of the initial-boundary value problem \eqref{eqn:pde}.
For the simplicity of notations, we let $I$ be the identity operator, and  $A(q)$ be the realization of $ -\Delta + q I$   with the homogeneous Dirichlet boundary condition with the domain
$ \text{Dom}(A(q)) = \{ \psi \in H_0^1\II:\, A(q)  \psi \in L^2\II \}=H_0^1\II\cap H^2\II $.
If $q\in\mathcal Q$, for any $ \psi \in H_0^1\II\cap H^2\II$, the full elliptic regularity implies
(see e.g.  \cite[Lemma 2.1]{LiSun:2017} and \cite[Theorems 3.3 and 3.4]{GruterWidman:1982})
\begin{equation}\label{eqn:equiv-n}
c_1\|  \psi  \|_{H^2\II} \le \| A(q)  \psi  \|_{L^2\II} + \|  \psi  \|_{L^2\II} \le c_2\|  \psi  \|_{H^2\II}
\end{equation}
with constants $c_1$ and $c_2$ independent of $q$.

Let $D(q)$ be the Dirichlet map by $\phi=D(q)\psi$
with $\phi$ satisfying
$$ -\Delta \phi + q \phi = 0 ~~\text{in } \Omega ~~\text{and}~~ \phi = \psi ~~\text{in } \partial\Omega. $$
%It is known that $ \| D(q)\psi \|_{H^{\frac12+s}\II} \le C \| \psi \|_{H^s(\partial\Omega)}$ for any $s\in[0,\frac32]$,
%with a range $D(A(q)^{\frac14-s})$ for any small $s\in(0,\frac14)$\
%\cite[Proposition 2.12]{Acquistapace:1991}.
In particular, for any $q\in \Q$, there exists a constant $c$ independent of $q$ such that
\begin{equation}\label{eqn:Dqb2}
\| D(q) \psi \|_{H^2\II} \le C \|  \psi \|_{H^{\frac{3}{2}}(\partial\Omega)}\qquad \text{for all}~~ \psi\in H^{\frac{3}{2}}(\partial\Omega).
\end{equation}
This is a direct result of the regularity of the Dirichlet operator $D(0)$ \cite[(1.2.2)]{Lasiecka:1980} and a simple shift argument.

Then the solution $u$ of problem \eqref{eqn:pde} could be represented by \cite[eq. (2.2)]{Lasiecka:1980}
\begin{equation}\label{eqn:sol-rep}
\begin{aligned}
u(t) &= F(t;q)v + A(q) \int_0^t E(s;q)D(q)b \d s +   \int_0^t E(s;q)f  \d s \\
&= F(t;q)v + (I-F(t;q))D(q)b +(I-F(t;q)) A(q)^{-1}f,
\end{aligned}
\end{equation}
where the operators  $F(t;q)$ and $E(t;q)$ are defined by  \cite[eq. (6.25) and (6.26)]{Jin:2021book}
\begin{equation}\label{eqn:sol-op-1}
 \begin{aligned}
 F(t;q) = \frac{1}{2\pi\mathrm{i}} \int_{\Gamma_{\theta,\kappa}} e^{zt} z^{\alpha-1} (z^\alpha
 + A(q))^{-1}  \,\d z~~\text{and}~~
E(t;q) = \frac{1}{2\pi\mathrm{i}} \int_{\Gamma_{\theta,\kappa}} e^{zt} (z^\alpha +A(q))^{-1}\,\d z,
  \end{aligned}
\end{equation}
respectively.
Here  $\Gamma_{\theta,\kappa}$ denotes the integral contour in the complex plane $\mathbb{C}$ oriented counterclockwise, defined by
$
\Gamma_{\theta,\kappa} =\{z \in \mathbb{C}: |z| =  \kappa , |\arg z|\le \theta\}
\cup \{ z\in \mathbb{C}: z = \kappa e^{\pm i\theta},\rho\ge \kappa\},
$
with $ \kappa\ge 0$ and $\theta \in (\frac{\pi}{2}, \pi)$.
Throughout, we fix $\theta \in(\frac{\pi}{2},\pi)$ so that $z^{\al} \in \Sigma_{\al\theta}
\subset \Sigma_{\theta}:=\{0\neq z\in\mathbb{C}: {\rm arg}(z)\leq\theta\},$ for all $z\in\Sigma_{\theta}$.
Note that $E(t;q)=-A(q)\frac{d}{dt}F(t;q)$, and in case that $\alpha=1$ there holds $F(t;q)=E(t;q)$.

The next lemma gives smoothing properties and asymptotics of  $F(t;q)$ and $E(t;q)$. The proof follows from
the resolvent estimate (for any $q\in \Q$) \cite[Example 3.7.5 and Theorem 3.7.11]{ArendtBattyHieber:2011}:
\begin{equation} \label{eqn:resol}
  \| (z+A(q))^{-1} \|\le c_\phi (|z|^{-1},\lambda^{-1})   \quad \forall z \in \Sigma_{\phi},
  \,\,\,\forall\,\phi\in(0,\pi) ,
\end{equation}
where $\|\cdot\|$ denotes the operator norm from $L^2(\Omega)$ to $L^2(\Omega)$,
and $\lambda$ denotes the smallest eigenvalue of $-\Delta$ with homogeneous Dirichlet boundary condition.
In case that $q\in\Q$, the constant $c_\phi$ can be chosen independent of $q$.
The full proof of the following lemma has been given in \cite[Theorems 6.4 and 3.2]{Jin:2021book}.
%See the proof in \cite[Theorem 6.4]{Jin:2021book}.
\begin{lemma}\label{lem:op}
Let $\lambda$ be the smallest eigenvalue of $-\Delta$ with homogeneous boundary condition.
Let $F(t;q)$ and $E(t;q)$ be the solution operators defined in \eqref{eqn:sol-op-1} with potential coefficient $q\in\mathcal{Q}$.
Then they satisfy the following properties:
\begin{itemize}
\item[$\rm(i)$] $\|A(q) F (t;q)v\|_{L^2\II} + t^{\alpha-1} \| A(q) E (t;q)v  \| \le c  t^{-\alpha} \|v\|_{L^2\II},\quad\forall\,t\in(0,T]$;
\item[$\rm(ii)$] $\|F(t;q)v\|_{L^2\II}+ t^{1-\alpha}\|E(t;q)v\|_{L^2\II} \le c \min(1, \lambda^{-1} t^{-\alpha}) \|v\|_{L^2\II}  , \quad\forall\,t\in(0,T]$,
%\item[$\rm(iii)$] $t^\alpha \|AF(t)\|  + t^{1-\beta\alpha}\|A^{-\beta}F'(t)\|  \leq c ,\quad\forall\,t\in(0,T], \beta\in[0,1]$,
\end{itemize}
where the constants are independent of $q$ and $t$.
\end{lemma}\vskip5pt

%\begin{proof}
%By the representation \eqref{eqn:sol-op-1}, we have the relation
%\begin{equation*}
%A(q)^s (\partial_t)^\ell F(t;q)v = \frac{1}{2\pi\mathrm{i}} \int_{\Gamma_{\theta,\kappa}} e^{zt} z^{\alpha-1+\ell} A(q)^s(z^\alpha I
% + A(q))^{-1} v \,\d z.
%\end{equation*}
%Then  the resolvent estimate \eqref{eqn:resol} implies
%$$ \|A(q)^s(z^\al I+A(q))^{-1}\| \le \lambda^{s-1},\ s\in[0,1],$$
%As a result, direct computation leads to ($\kappa = t^{-1}$) \red{details}
%\begin{equation*}
%\|A(q)^s (\partial_t)^\ell F(t;q)v \|_{L^2(\Omega)}\le c \lambda^{s-1}t^{-\al-\ell}\|v\|_{L^2\II}.
%\end{equation*}
%For $s =0$ and $\ell = 0$, we derive $\|F(t;q)v\|_{L^2\II} \le c  \lambda^{-1}t^{-\al} \|v\|_{L^2\II}$.
%Besides, Lemma \ref{lem:op} (i) leads to $ \|F(t;q)v\|_{L^2\II} \le \|v\|_{L^2\II}$.
%The proof for the estimate for $E(t;q)$ is similar and hence omitted.
%\end{proof}

%This implies the following estimate \ZZ{prove it}
%\begin{equation}\label{eqn:H2}
%\| u(t) \|_{H^2\II} \le c \Big(\| v  \|_{H^2\II} + \| b \|_{H^\frac{3}{2}\II} + \| f \|_{L^2\II} \Big)
%\end{equation}

We also need the following assumption on the problem data.
\begin{assumption}\label{ass:cond-1}
Let the initial data $v$, the boundary data $b$ and the source term $f$ satisfy the following conditions:
\begin{itemize}
\item[(i)] $v \in H^2\II$, $v \ge M_2 >0$ in $\Omega$,  $v(x) = b(x)$ for all $x \in \partial\Omega$;
\item[(ii)] $b \in H^{2}(\partial\Omega)$, $b \ge M_2 >0$ in $\partial\Omega$;
\item[(iii)] $f\in W^{1,p}\II \subset C( \overline \Omega)$ (with some  $p>\max(d,2)$), $f \ge 0$ and $f  +\Delta v  - M_1 v  \ge 0$ in $\Omega$.
\end{itemize}
\end{assumption}\vskip5pt

Under Assumption \ref{ass:cond-1}, we have the following results about the solution regularity and behaviours for the direct problem \eqref{eqn:pde}.
\begin{lemma}\label{lem:u-reg}
Let $q\in \mathcal{Q}$ and Assumption \ref{ass:cond-1} be valid. Then
the solution $u(t)$ to problem \eqref{eqn:pde} with potential $q$ satisfies the following properties:
\begin{itemize}
\item[(i)] $u(t) \in H^2 \II$ for all $t>0$, and there exists a constant $C$ independent of $q$ such that
$\max_{t\in[0,T]} \| u(t)  \|_{L^\infty\II} \le C$;
\item[(ii)] $\Dal u(t)\in H^2\II \cap H_0^1\II $, $\Delta u(t) \in C(\overline \Omega)$, and $\Dal u(x,t) \ge 0$, $u(x,t)\ge M_2$ for all $(x,t)\in\overline\Omega\times[0,T]$;
%and there exists a constant $C$ independent of $q$ such that
%$$ \| \Dal u(t)  \|_{L^\infty\II} \le C \quad \text{for all} ~~ t\in(0,T].$$
\item[(iii)] $ f(x)+ \Delta u(x,t) \ge q(x)M_2$ for all $t>0$ and $x\in \overline \Omega$.
\end{itemize}
\end{lemma}

\begin{proof}
By the smoothing property in Lemma \ref{lem:op}, we observe that
\begin{equation*}
A(q)[F(t;q)v   -F(t;q)D(q)b-F(t;q) A(q)^{-1}f] \in L^2\II.
\end{equation*}
Then the elliptic regularity (see \cite[Lemma 2.1]{LiSun:2017} and \cite[Theorems 3.3 and 3.4]{GruterWidman:1982})  implies that
$F(t;q)v   -F(t;q)D(q)b  -F(t;q) A(q)^{-1}f \in H^2\II $.
Besides, we observe that
$D(q)b$ and $A(q)^{-1}f$ belong to $H^2\II $ (see e.g. \cite[Proposition 2.12]{Acquistapace:1991}
and \cite[Theorem B.54]{ern-guermond}).
These together with \eqref{eqn:sol-rep}
imply that $u(t)\in H^2\II$. Finally, we define an auxiliary function $\phi(x,t)$ satisfying
\begin{equation}\label{PDE_q_0}
 \begin{cases}
  \begin{aligned}
     \Dal \phi(x,t)-\Delta \phi(x,t)  &=f(x), &&(x,t)\in \Omega\times(0,T],\\
    \phi(x,t)&=b(x),&&(x,t)\in\partial\Omega\times(0,T],\\
    \phi(x,0) &=v(x)  ,&&x\in\Omega.
  \end{aligned}
  \end{cases}
 \end{equation}
By  Assumption \ref{ass:cond-1} and the maximal $L^p$ regularity
(see e.g. \cite[Lemma 2.1]{LiSun:2017} for parabolic equation and \cite[Theorem 6.11]{Jin:2021book} for fractional evolution equations),
we know that $\phi \in W^{\alpha, q}(0,T;L^2\II)\cap L^q(0,T; H^2\II)$ for any $q\in[2,\infty)$.
 Then by means of the Sobolev embedding and the interpolation between $W^{\alpha, q}(0,T;L^2\II)$ and $L^q(0,T;H^2\II)$,  we have $\phi \in C([0,T]\times\overline\Omega)$.
 As a result, the comparison principle implies
 $ \| u   \|_{C([0,T]\times\overline\Omega)} \le  \| \phi  \|_{C([0,T]\times\overline\Omega)} \le C $,
 where the constant $C$ is  independent of potential $q$. Then we complete  the proof of (i).

Next, we let $w=\Dal u$, which is the solution to the following initial-boundary value problem
\begin{equation}\label{PDE_w_1}
 \begin{cases}
  \begin{aligned}
     \Dal w(x,t)-\Delta w(x,t) +q(x) w(x,t) &=0, &&(x,t)\in \Omega\times(0,T],\\
    w(x,t)&=0,&&(x,t)\in\partial\Omega\times(0,T],\\
    w(x,0)=f(x)+\Delta v (x) &- q(x)v (x),&&x\in\Omega.
  \end{aligned}
  \end{cases}
 \end{equation}
 Noting that $w(x,0)\in L^2\II$ by Assumption \ref{ass:cond-1},
 then we apply  Lemma \ref{lem:op} to arrive that
 $$ A(q)w(t) = A(q) F(t) [f +\Delta v  - q v  ]  \in L^2\II. $$
 Then the elliptic regularity
 implies $\Dal u(t) = w(t) \in H^2\II \cap H_0^1\II \subset C(\overline \Omega)$ for $t>0$.
 Recalling Assumption \ref{ass:cond-1} (i) and (iii), we have
 $f(x)+\Delta v (x)  - q(x) v (x) \ge 0$.
 This and Lemma \ref{positive} indicate the
 positivity of $\Dal u(x,t)$.
Meanwhile, the facts that $u(t), \Dal u(t), q, f \in C( \overline \Omega)$ lead to $\Delta  u(t) \in C(\overline \Omega)$.
 Besides, by means of the facts that $v(x),b(x)\ge M_2$ in Assumption \ref{ass:cond-1} and $\Dal u(x,t) \ge 0$, we derive
 $$ u(x,t) = u(x,0) +  \int_0^t \frac{(t-s)^{\alpha-1}}{\Gamma(\alpha)} \partial_s^\alpha u(x,s)\,\d s \ge u(x,0) \ge M_2$$
for all $(x,t)\in \bar \Omega \times [0,T]$.

Finally, by the positivity of $\Dal u(x,t)$ we conclude that
\begin{equation}\label{eqn:qinD}
 f(x)+ \Delta  u(x,t)=\Dal u(x,t)+q(x)u(x,t)\ge q(x)u(x,t) \ge q(x) M_2.
\end{equation}
This completes the proof of (ii) and (iii).
\end{proof}\vskip5pt

From now on, we use the notation $u(q)$ to denote the solution to \eqref{eqn:pde} with the potential $q$.
Let $q^\dag$ be the exact potential to be reconstructed.
Under Assumption \ref{ass:cond-1}, according to Lemma \ref{lem:u-reg},
the (exact) observation $g(x) = u(x,T; q^\dag)$ satisfies
\begin{equation}\label{ass:g}
 g \in C(\overline\Omega), ~~ \Delta g \in C(\overline\Omega),~~ f(x) + \Delta g(x) \ge 0, ~~\text{and}~~ g(x)\ge M_2 > 0 ~ \text{for all}~x\in \overline\Omega .
\end{equation}

To show the uniqueness of the potential, we define an operator
\begin{equation}\label{eqn:K}
 Kq(x)=\frac{f(x)-\Dal u(x,T;q)+\Delta g(x)}{g(x)} \quad \text{for}~~q \in \Q.
\end{equation}
Then under Assumption \ref{ass:cond-1}, Lemma \ref{lem:u-reg} implies that the exact potential $q^\dag$ satisfies
\begin{equation*}\label{eqn:DK}
q^\dag \in \DK=\Big\{\psi\in C(\overline \Omega):0\le\psi\le \frac{f(x)+ \Delta g(x)}{g(x)}\Big\}.
\end{equation*}

Next, we intend to show that the inverse potential problem is equivalent
to find a fixed point of the operator $K$ in the set $\mathcal{D}(K)$.
This is given by the following lemma.
\begin{lemma}\label{lem:equiv}
Let Assumption \ref{ass:cond-1}  be valid and the operator $K$ be defined by \eqref{eqn:K}.
Then we have the following equivalence.
\begin{itemize}
\item[(i)] If $q^\dag\in\mathcal{Q}$ satisfies $u(x,T;q^\dag)=g(x)$,
then $q^\dag$ is a fixed point of the operator $K$  in $\DK$.
\item[(ii)] If $q^\dag \in \DK$ is a fixed point of the operator $K$, then $q^\dag$ satisfies $u(x,T;q^\dag)=g(x)$.
\end{itemize}
\end{lemma}
\begin{proof}
It is obvious that $u(x,T;q^\dag)=g(x)$ implies that $q^\dag$ is the fixed point of $K$.
Then the relation \eqref{ass:g}  and the fact that $\partial_t^\alpha u(x,t;q^\dag) \ge 0$
(by Lemma \ref{lem:u-reg}) yield that $q^\dag \in \DK$.

Therefore, it suffices to show the reversed conclusion.
We assume that $q^\dag \in \DK$ is one fixed point of operator $K$, then we have
\begin{equation*}
 f(x)-\Dal u(x,T;q^\dag)=q^\dag(x)g(x)-\Delta g(x)=-\Delta u(x,T;  q^\dag )+q^\dag(x)u(x,T; q^\dag ).
\end{equation*}
Letting $w(x)=u(x,T;q^\dag)-g(x)$, we observe that $w$ satisfies the elliptic system
 \begin{equation*}
 \begin{cases}
  \begin{aligned}
    -\Delta w(x)+q^\dag(x)w(x)&=0, &&x\in \Omega,\\
    w(x)&=0,&&x\in\partial\Omega.
  \end{aligned}
  \end{cases}
 \end{equation*}
Then the comparison principle of elliptic equation implies $w=0$.
Hence $u(x,T;q^\dag)=g(x)$, which implies that $q^\dag$ generates the terminal measurement $g(x)$.
\end{proof}

Due to the equivalence given by Lemma \ref{lem:equiv} and the fact that $q^\dag \in \DK$,
we aim to verify that the operator
$K$ admits a unique fixed point in $\DK$.
To this end, we intend to show that $K$ generates a decreasing sequence in $\DK$
from an \textsl{\textsl{a priori}} chosen starting value.
Then the uniqueness of the fixed point follows immediately.
%With an additional condition on the operator $K_g$, the existence can be deduce.
%At last, we impose an appropriate initial condition and request a small final
%time $T$ to show the conditional stability.

\begin{lemma}[Monotonicity]\label{monotone}
Let Assumption \ref{ass:cond-1} be valid. Then
$K$ is a monotone operator, i.e., $Kq_1\le Kq_2$ for any $q_1,q_2\in  \DK $ with $q_1\le q_2$.
\end{lemma}
\begin{proof}
First of all, we recall Lemma \ref{lem:u-reg} which implies that $\Dal u(x,t;q_2)\ge 0$ in $[0,T]\times\Omega$.
 Then we define $w(x,t)=\Dal (u(x,t;q_1) - u(x,t;q_2))$, and note that $w$ satisfies
\begin{equation*}
 \begin{cases}
  \begin{aligned}
    (\Dal -\Delta + q_1(x)) w(x,t)&=(q_2-q_1)\Dal u(x,t;q_2),&&(x,t)\in \Omega\times(0,T],\\
    w(x,t)&=0,&&(x,t)\in\partial\Omega\times(0,T],\\
    w(x,0)&= (q_2 - q_1) v(x) ,&&x\in\Omega.
  \end{aligned}
  \end{cases}
 \end{equation*}
 Since $ (q_2 - q_1) v(x) , \, (q_2-q_1)\Dal u(x,t;q_2)\ge 0$, using Lemma \ref{positive} again yields that
 \begin{equation*}
  w(x,t)=\Dal u(x,t;q_1)-\Dal u(x,t;q_2)\ge 0.
 \end{equation*}
From the definition of $K$ in \eqref{eqn:K} and the fact that $g(x)\ge M_2 >0 $ in $\Omega$ by \eqref{ass:g}, we have
 \begin{equation*}
  Kq_1-Kq_2=\frac{\Dal u(x,T;q_2)-\Dal u(x,T;q_1)}{g(x)}\le 0.
 \end{equation*}
This completes the proof of the lemma.
\end{proof}

Then the monotonicity of $K$ immediately implies the following lemma.

\begin{lemma}\label{uniqueness1}
Suppose that $v , f, b$ satisfy Assumption \ref{ass:cond-1}.
If $q_1, q_2\in \mathcal{D}(K)$ are both fixed points of $K$ and  $q_1 \le q_2$, then $q_1=q_2$.
\end{lemma}
\begin{proof}
 From Lemma \ref{lem:equiv}, we have $u(x,T;q_1)=u(x,T;q_2)=g(x)$.
 Define $w(x,t)=u(x,t;q_1)-u(x,t;q_2)$, then the PDE model for $w$ is given as
 \begin{equation}\label{PDE-w}
 \begin{cases}
  \begin{aligned}
    (\Dal-\Delta+q_1(x)) w(x,t)&=(q_2-q_1)u(x,t;q_2), &&(x,t)\in \Omega\times(0,T],\\
    w(x,t)&=0,&&(x,t)\in\partial\Omega\times(0,T],\\
    w(x,0)&=0,&&x\in\Omega.
  \end{aligned}
  \end{cases}
 \end{equation}
 From Lemma \ref{positive}, we have $u(x,t;q_2)> 0$ in $\Omega\times [0,T]$,
 which leads to the non-negativity of the source $(q_2-q_1)u(x,t;q_2)$.
 This yields that
 $w(x,t)\ge 0$ in $\overline \Omega\times[0,T]$. From the proof of Lemma \ref{monotone}, we have
 $\Dal (u(x,t;q_1)- u(x,t;q_2))=\Dal w(x,t)\ge 0$.  The relation
 \begin{equation*}
  w(x,T)= w(x,0)+ \int_0^T \frac{(T-t)^{\alpha-1}}{\Gamma(\alpha)} \partial_t^\alpha w(x,t)\ dt
 \end{equation*}
 together with the results
 $$w(x,T)=u(x,T;q_1)-u(x,T;q_2)=0,~~ w(x,0)=0 ~~ \text{and}~~ \Dal w(x,t)\ge 0$$
immediately yields that $\Dal w(x,t)=0$ for $t\in (0,T)$ almost everywhere, and hence $w(x,t)\equiv0$.
This and the equation \eqref{PDE-w} imply  that $(q_2-q_1)u(x,t;q_2)=0$ on
$\Omega\times [0,T]$. This together with the strict positivity of $u(x,t;q_2)$ in $\Omega\times [0,T]$ leads to $q_1=q_2$.
\end{proof}

The above results motivate us to define the  iteration:
\begin{equation}\label{iteration}
 q_0(x) =\frac{f(x)+\Delta g(x)}{g(x)}\in\DK \quad \text{and}\quad q_n=K  q_{n-1}
 ~~ \text{for}~~n\in\mathbb N^+.
\end{equation}
Note that the initial guess $q_0$ is set to be the upper bound of the set $\mathcal{D}(K)$.
%Then the monotonicity of $K$ implies that $0\le q^\dag \le q_n \le q_0$ and hence $q_n \in \DK$ for all $n\in\mathbb N^+$.
%Next, we intend to show the uniqueness of potential by using the monotonicity of the sequence $\{q_n\}_{n=0}^\infty$.
%To this end, we need the following result.
%
%
%
%
%
Next, we shall state the main theorem in this section which shows that the
fixed point of $K$ must be the limit of the sequence $\{q_n\}_{n=0}^\infty$ generated by \eqref{iteration} and hence it is unique.
\begin{theorem}\label{thm:uniqueness2}
If there exists one fixed point $q^\dag\in\mathcal{D}(K)$ of $K$, then the sequence
$\{q_n\}_{n=0}^\infty$ generated by \eqref{iteration} is included in $\mathcal{D}(K)$ and converges decreasingly to $q^\dag$.
Therefore, the operator $K$ admits at most one fixed point in $\mathcal{D}(K)$.
\end{theorem}
\begin{proof}
 From the proof of Lemma \ref{monotone}, we conclude that $\Dal u(x,T;q_0)\ge 0$.
 This gives that
 \begin{equation*}
  q_1=K q_0=\frac{f(x)-\Dal u(x,T;q_0)+\Delta g(x)}{g(x)}
  \le \frac{f(x) + \Delta g(x)}{g(x)}=q_0(x).
 \end{equation*}
 Meanwhile, we know that $q^\dag\in \mathcal{D}(K)$ and so $q^\dag\le q_0$. This and Lemma \ref{monotone} result in
  $$0\le q^\dag=K q^\dag\le K q_0=q_1.$$
   As a result, we obtain $0\le q^\dag\le q_1\le q_0$. Using Lemma \ref{monotone} again, we have $K q^\dag\le K q_1\le Kq_0$, namely $q^\dag\le q_2\le q_1$.
Continuing this argument, we can conclude that
\begin{equation*}
 0\le q^\dag\le \cdots\le q_{n+1}\le q_n\le\cdots\le q_0.
\end{equation*}
Now we have proved $\{q_n\}_{n=0}^\infty$ is decreasing and bounded by $q^\dag$ from below and $q_0$ from above.
Therefore, this sequence is included in $\DK$.

Next, we intend to show that the sequence $\{q_n\}_{n=0}^\infty$ converges to $q^\dag$. Note that the sequence $\{q_n\}_{n=0}^\infty$ is decreasing and
it has a lower bound, therefore this sequence converges and we denote the limit by $q^*$, i.e. $q^*=\lim_{n\to \infty}K^n q_0$.
Then $q^*$ is one  fixed point of the operator $K$.
Moreover,
we have $q^\dag\le q^*$ since $q^\dag$ is the lower bound of $\{q_n\}_{n=0}^\infty$,
and  $q^\dag\le q^*\le q_0$ indicates that $q^*\in \DK$.
Finally, we apply  Lemma \ref{uniqueness1} to conclude that $q^\dag=q^*$, and hence complete the proof.
\end{proof}

%
%\subsection{Existence}
%For the existence theorem, we need to impose an extra condition,
%\begin{equation}\label{condition_existence}
% \exists \tilde\psi\in \mathcal D(K) \ \text{such that}\
% K\tilde\psi\ge \tilde\psi.
%\end{equation}
%This condition ensures the lower bound of the sequence $\{\psi_n\}_{n=0}^\infty$ created by  iteration \eqref{iteration}. Next we will state the existence theorem.
%\begin{theorem}[Existence]\label{existence}
% Under conditions \eqref{condition_g} and \eqref{condition_existence}, there exists fixed points of $K$ in $\d$.
%\end{theorem}
%\begin{proof}
% Since $\tilde\psi\in \mathcal D(K)$, $\tilde\psi\le \psi_0$. So from condition \eqref{condition_existence} and Lemma \ref{monotone}, we have
% \begin{equation*}
%  \tilde\psi\le K\tilde\psi\le K\psi_0=\psi_1.
% \end{equation*}
%Applying Lemma \ref{monotone} on the above inequality again, it holds that
%$\tilde\psi\le K\tilde\psi\le K\psi_1=\psi_2$. Continuing this procedure, we see that $\tilde\psi\le \psi_n$ for each $n\in\mathbb N^+$.
%In Lemma \ref{uniqueness2}, we have proved that $\{\psi_n\}_{n=0}^\infty$ is decreasing. Therefore, $\{\psi_n\}_{n=0}^\infty$ is a decreasing sequence with lower bound $\tilde\psi$. This leads to the convergence and we denote the limit of $\{\psi_n\}_{n=0}^\infty$ by $\psi^*$. Due to
%$$\psi^*=\lim_{n\to \infty}K^n\psi_0,$$
%$\psi^*$ is one fixed point of $K$. In addition, the estimate $\tilde\psi\le \psi^*\le \psi_0$ ensures that $\psi^*\in \d$. We have constructed one fixed point $\psi^*\in \d$ of $K$. The existence can be achieved and the proof is complete.
%\end{proof}
%
%
%

\section{Conditional stability}\label{sec:stability}
The aim of this section is to establish a stability of the inverse potential problem.
Note that \cite{ZhangZhou:2017} provides a conditional stability in a Hilbert space setting for one dimensional diffusion problem
by applying a spectrum perturbation argument
(cf. \cite[Lemma 2.2]{ZhangZhou:2017} and \cite{Trubowitz:1987}), which is not applicable in high dimensional cases.
We refer interested readers to \cite{ChoulliYamamoto:1996,ChoulliYamamoto:1997,JinZhou:IP2021-a} for some conditional stability results
for sufficiently small $T$.

Let us begin with the following \textsl{a priori} estimate for $ \partial_t^\alpha u(t;q)$.

\begin{lemma}\label{lem:Dalu}
Let $q\in \Q$ and $u(q)$ be the solution to problem \eqref{eqn:pde}. Then we have the estimate %for all $s\in[0,2]$
$$  \|  \partial_t^\alpha u(t;q) \|_{H^{s}\II} %+ t \|  \partial_t u(t) \|_{H^2\II}
\le c \min(t^{-s\alpha/2}, t^{-\alpha}) \quad \text{for all}~~s\in[0,2],$$
where $c$ %depends on $\|v \|_{L^2\II}$, $\|b\|_{L^2(\partial\Omega)}$ and $\|f\|_{L^2\II}$,
is independent of $q$ and $t$.
\end{lemma}

\begin{proof}
%The estimate for $\| \partial_t u(t) \|_{L^2\II} $ follows from the representation \eqref{eqn:sol-rep} and Lemma \ref{lem:op} (with $s=0,\ell=1$)
%\begin{equation*}
%\begin{aligned}
%\|   \partial_t u(t) \|_{L^2\II} &\le \|    F'(t;q)v \|_{L^2\II} + \|  F'(t;q) D(q)b\|_{L^2\II} + \|  F'(t;q) A(q)^{-1} f\|_{L^2\II} \\
%&\le c t^{-\alpha-1} \big(\|v \|_{L^2\II} + \|D(q) b\|_{L^2\II} + \|f\|_{L^2\II} \Big) \\
%&\le c t^{-\alpha-1} \big(\|v \|_{L^2\II} + \|  b\|_{L^2\II} + \|f\|_{L^2\II} \Big).
%\end{aligned}
%\end{equation*}
%where the constant is independent of $q$.
%Using the norm equivalence in \eqref{eqn:equiv-n}, we have
%\begin{equation*}
%\begin{aligned}
%\| \partial_t u(t) \|_{H^2\II} &\le \|  A(q)  F'(t;q) (v -D(q)b   - A(q)^{-1} f) \|_{L^2\II} + M_1 \|  \partial_t u(t) \|_{L^2\II} \\
%&\le c t^{-\alpha-1} \big(\|v \|_{L^2\II} + \|  b\|_{L^2\II} + \|f\|_{L^2\II} \Big),
%\end{aligned}
%\end{equation*}
%where the constant is independent of $q$.
According to \eqref{PDE_w_1}, we have the representation
\begin{equation}\label{eqn:Dalu-4}
\partial_t^\al u(t;q) =   F(t;q)(\Delta v - q v +  f ) \in H^2\II \cap H_0^1\II\quad\text{for all}~~ t>0.
\end{equation}
Then applying Lemma \ref{lem:op}, we obtain
\begin{equation*}
\begin{aligned}
\| \Dal u(t;q) \|_{L^2\II} &\le \|F(t;q)(\Delta v - q v +  f )\|_{L^2\II}
 \le c \min(1,t^{-\alpha}) \big(\|v\|_{H^2\II} + \|f\|_{L^2\II} \Big).
\end{aligned}
\end{equation*}
Next, by applying the norm equivalence in \eqref{eqn:equiv-n} and the estimate in Lemma \ref{lem:op}, we derive
\begin{equation*}
\begin{aligned}
\| \Dal u(t;q) \|_{H^2\II} &\le c\Big(\| F(t;q) (\Delta v - q v +  f ) \|_{L^2\II} + \|A(q) F(t;q)(\Delta v - q v +  f )\|_{L^2\II} \big) \\
&\le c  \big( \min(1,t^{-\alpha}) \|\Delta v - q v +  f  \|_{L^2\II} +   c t^{-\alpha} \|\Delta v - q v +  f  \|_{L^2\II} \\
&\le  ct^{-\alpha}\big(\|v\|_{H^2\II} + \|f\|_{L^2\II} \Big).
\end{aligned}
\end{equation*}
These together with interpolation between $L^2\II$ and $H^2\II \cap H_0^1\II$ immediately lead to the desired result.
\end{proof}

For different potentials $q_1,q_2\in \mathcal{Q}$, we denote the solution to
\eqref{eqn:pde} with potential $q_i$ by $u(q_i)$.
Then the following lemma provides  an important \textsl{a priori} estimate
which (and whose discrete analogue) plays a crucial role in our error analysis.

\begin{lemma}\label{lem:stab-0}
Let Assumption \ref{ass:cond-1} be valid and $q_1, q_2 \in \mathcal{Q}$.
Then for any $t>t_0$ and any positive parameter $\epsilon < \min(1,2-\frac{d}{2})$ there holds
\begin{equation*}
\|\Dal(u(q_1)-u(q_2))(t)\|_{H^2\II} \le c \max(t^{-\alpha}, t^{-(1-\epsilon)\alpha}) \|q_1-q_2\|_{L^2\II},
\end{equation*}
where the constant $c$ %depends on $\|v \|_{L^\infty\II}$, $\|b\|_{L^2(\partial\Omega)}$ and $\|f\|_{L^2\II}$, but it
is independent of $q_1$, $q_2$ and $t$.
\end{lemma}

\begin{proof}
Let $\phi(x,t)=\Dal(u(q_1)-u(q_2))(t)$. Then we note that $\phi(x,t)\in H_0^1\II$ satisfies
 \begin{equation}\label{PDE-phi}
    (\Dal-\Delta+q_1(x)) \phi(x,t) =(q_2-q_1)\partial_t^\alpha u(x,t;q_2)\quad \text{for}~~(x,t)\in \Omega\times(0,T]
 \end{equation}
with the initial condition
$ \phi(0) =  (q_2-q_1)v$.
We apply the solution representation \eqref{eqn:sol-rep} to derive
\begin{equation*}
\phi(t) = F(t;q_1) \phi(0) + \int_0^t E(s;q_1) (q_2-q_1) \partial_t^\alpha u(t-s;q_2) \,\d s.
\end{equation*}
Taking $L^2$ norm on the above relation, Lemma \ref{lem:op} and  Assumption \ref{ass:cond-1}  lead to for any $\epsilon\in(0,1)$
\begin{equation*}
\begin{split}
 \| \phi(t) \|_{L^2\II} &= \| F(t;q_1)\|\, \|(q_2 - q_1)v\|_{L^2\II} + \int_0^t \| E(s;q_1) \| \, \|(q_2-q_1) \partial_t^\alpha u(t-s;q_2)\|_{L^2\II} \,\d s\\
 &\le c \| q_2-q_1 \|_{L^2\II} \Big(t^{-\alpha} +  \int_0^t s^{-1+\epsilon\alpha/2} \|\partial_t^\alpha u(t-s;q_2)\|_{L^\infty\II} \,\d s\Big).
\end{split}
\end{equation*}
Here we use the estimate that $\| E(s;q_1) \| \le c s^{-1+\epsilon\alpha/2}$ which is a direct result of the second assertion of Lemma \ref{lem:op}
and the interpolation.
Then according to Lemma \ref{lem:Dalu} and the Sobolev embedding theorem, we obtain for $r>\frac{d}{2}$ and $d=1,2,3$,
\begin{equation*}
\begin{split}
 \| \phi(t) \|_{L^2\II}  &\le c \| q_2-q_1 \|_{L^2\II}
  \Big(t^{-\alpha} +  \int_0^t s^{-1+\epsilon\alpha/2} \|\partial_t^\alpha u(t-s;q_2)\|_{L^\infty\II} \,\d s\Big)\\
 &\le c \| q_2-q_1 \|_{L^2\II} \Big(t^{-\alpha} +  \int_0^t s^{-1+\epsilon\alpha/2} \|\partial_t^\alpha u(t-s;q_2)\|_{H^r\II} \,\d s\Big)\\
 &\le c \| q_2-q_1 \|_{L^2\II} \Big(t^{-\alpha} +  \int_0^t s^{-1+\epsilon\alpha/2} (t-s)^{-r\alpha/2} \,\d s\Big)\\
 &\le c \| q_2-q_1 \|_{L^2\II} \big(t^{-\alpha} + t^{\epsilon\alpha/2-r\alpha/2}\big).
\end{split}
\end{equation*}
Finally, the choice that $r=2-\epsilon$ leads to the estimate that
$$ \| \phi(t) \|_{L^2\II} \le c \| q_2-q_1 \|_{L^2\II} \big(t^{-\alpha} + t^{-\alpha(1-\epsilon)}\big) \le c \max(t^{-\alpha}, t^{-(1-\epsilon)\alpha}) \|q_1-q_2\|_{L^2\II}. $$
This completes the proof of the lemma.
\end{proof}

 Next, we state the main theorem of this section, which shows
 the conditional stability of the inverse potential problem.
\begin{theorem}\label{thm:cond-stab}
Let Assumption \ref{ass:cond-1} be valid, $q_1, q_2 \in \mathcal{Q}$, and $u(t;q_i)$ be the solution to \eqref{eqn:pde} with the potential $q_i$.
Then there exists $T_0\ge0$ such that for any $T\ge T_0$ there holds
$$ \| q_1 - q_2 \|_{L^2\II} \le C  \|  u(T;q_1) - u(T;q_2)  \|_{H^2\II},$$
where the constant $C$ %depends on $\|v \|_{L^\infty\II}$, $\|b\|_{L^2(\partial\Omega)}$ and $\|f\|_{L^2\II}$, but it
is independent of $q_1$, $q_2$ and $T$.
\end{theorem}

\begin{proof}
Recalling that, for $i=1,2$, $q_i$ could be written as
$$q_i = \frac{f-\Dal u(T;q_i) + \Delta u(T;q_i)}{u(T;q_i)}.$$
Then we split $q_1 - q_2$ into three parts:
\begin{equation*}
\begin{aligned}
 q_1 - q_2 %& =  \frac{f-\Dal u(T;q_1) + \Delta u(T;q_1)}{u(T;q_1)}- \frac{f-\Dal u(T;q_2) + \Delta u(T;q_2)}{u(T;q_2)} \\
 & = f\frac{u(T;q_2)-u(T;q_1)}{u(T;q_1)u(T;q_2)}   + \frac{ u(T;q_1)\Dal u(T;q_2) -u(T;q_2)\Dal u(T;q_1)  }{u(T;q_1)u(T;q_2)}\\
 &\quad + \frac{u(T;q_2)\Delta u(T;q_1)-u(T;q_1)\Delta u(T;q_2)}{u(T;q_1)u(T;q_2)}. \\
   \end{aligned}
\end{equation*}
Using  Assumption \ref{ass:cond-1}, we conclude that $u_i \ge M_2>0$ and hence
$$ \Big\| f\frac{u(T;q_2)-u(T;q_1)}{u(T;q_1)u(T;q_2)} \Big\|_{L^2\II} \le \frac{\| f \|_{L^\infty\II}}{M_2^2} \| u(T;q_2)-u(T;q_1) \|_{L^2\II}.$$
Besides, we use the fact that $\|u_i(T)\|_{L^\infty\II}$ and $\|\Dal u_i (T)\|_{L^\infty\II}$
are bounded uniformly in $q$ (Lemma \ref{lem:u-reg}) and Lemma \ref{lem:stab-0} to derive for any $\epsilon$ close to $0$,
\begin{equation*}
\begin{aligned}
&\quad \Big\| \frac{ u(T;q_1)\Dal u(T;q_2) -u(T;q_2)\Dal u(T;q_1)  }{u(T;q_1)u(T;q_2)} \Big\|_{L^2\II} \\
&\le c \Big(\|  u(T;q_1)\|_{L^\infty\II} \| \Dal (u(T;q_2) -  u(T;q_1)) \|_{L^2\II} + \| \Dal u(T;q_1) \|_{L^\infty\II}\|  u(T;q_1) - u(T;q_2)\|_{L^2\II}\Big)\\
&\le c \Big( \max(T^{-\alpha},T^{-(1-\epsilon)\alpha}) \| q_1 -q_2 \|_{L^2\II} + \|  u(T;q_1) - u(T;q_2)\|_{L^2\II} \Big).
   \end{aligned}
\end{equation*}
Similarly, we apply the fact that $\|u_i(T)\|_{L^\infty\II}$ and $\|\Delta u_i (T)\|_{L^\infty\II}$ are bounded uniformly in $q_i$ (Lemma \ref{lem:u-reg}) to arrive at
\begin{equation*}
\begin{aligned}
&\quad \Big\|\frac{u(T;q_2)\Delta u(T;q_1)-u(T;q_1)\Delta u(T;q_2)}{u(T;q_1)u(T;q_2)}  \Big\|_{L^2\II} \\
&\le c \Big(\|  u(T;q_1)\|_{L^\infty\II} \| \Delta (u(T;q_2) -  u(T;q_1)) \|_{L^2\II} + \| \Delta u(T;q_1) \|_{L^\infty\II}\|  u(T;q_1) - u(T;q_2)\|_{L^2\II}\Big)\\
&\le c \Big( \| \Delta (u(T;q_1)-u(T;q_1)) \|_{L^2\II} + \|  u(T;q_1) - u(T;q_2)\|_{L^2\II} \Big).
   \end{aligned}
\end{equation*}
As a result, we arrive at
$$ \| q_1 - q_2  \|_{L^2\II} \le   c_1 \|  u(T;q_1) - u(T;q_2)  \|_{H^2\II} + c_2 \max(T^{-\alpha},T^{-(1-\epsilon)\alpha}) \| q_1 -q_2 \|_{L^2\II}. $$
Then for $T_0$ such that $c_2 \max(T_0^{-\alpha},T_0^{-(1-\epsilon)\alpha}) \le c_3$ for some constant $c_3\in(0,1)$, and $T\ge T_0$, we have
$$ \| q_1 - q_2  \|_{L^2\II} \le   \frac{c_1 }{1-c_3}  \|  u(T;q_1) - u(T;q_2)  \|_{H^2\II} . $$
This completes the proof of the lemma.
\end{proof}

 \section{Completely discrete scheme}\label{sec:fully}
In this section, we shall develop a fully discrete scheme for solving the inverse potential problem.
To this end, we shall introduce the time stepping method using convolution quadrature in the first part,
then discuss the spatial discretization using finite element method. A reconstruction algorithm will be presented
to recover the potential from the noisy observational data. Finally, we establish an \textsl{a priori} error bound
showing the way to choose the (space/time)  mesh sizes according to the noise level.

\subsection{Time stepping scheme for solving the direct problem}
The literature on the numerical approximation for the nonlocal-in-time subdiffusion
equation \eqref{eqn:pde} is vast, see e.g., \cite{JinLazarovZhou:2019} for an overview of existing schemes.
Here we apply the convolution quadrature to discretize the fractional derivative  on uniform grids.
Let $\{t_n=n\tau\}_{n=0}^N$ be a uniform partition of the time interval $[0,T]$,
with a time step size $\tau=T/N$.
The convolution quadrature (CQ) was first proposed by Lubich \cite{Lubich:1986}
for discretizing Volterra integral equations. This approach provides a systematic framework
to construct high-order numerical methods to discretize fractional derivatives, and has been the
foundation of many early works. The time stepping scheme for problem \eqref{eqn:pde}
reads: given $u^0(q)=v$, find $u^n(q) \in H^1\II$ such that $\gamma_0(u^n(q)) = b$ and
\begin{align}\label{eqn:step-0}
   \bar \partial_\tau^\alpha u^n(q) -\Delta u^n(q) + q u^n(q) = f \quad \text{with}~~ n=1,2,\ldots,N,
\end{align}
where $\bar\partial_\tau^\alpha \varphi^n$ denotes the
backward Euler CQ approximation (with $\varphi^j=\varphi(t_j)$) \cite{Lubich:1986}:
\begin{equation}\label{eqn:CQ-BE}
  \bar\partial_\tau^\alpha \varphi^n = \tau^{-\alpha} \sum_{j=0}^nb_j^{(\alpha)} (\varphi^n - \varphi^0) ,\quad\mbox{ with } (1-\xi)^\alpha=\sum_{j=0}^\infty b_j^{(\alpha)}\xi^j.
\end{equation}
Note that the weights $b_j^{(\alpha)}$ are given explicitly by
$b_j^{(\alpha)} = (-1)^j\frac{\Gamma(\alpha+1)}{\Gamma(\alpha-j+1)\Gamma(j+1)}$, and thus
$b_j^{(\alpha)} = (-1)^j(j!)^{-1}\alpha(\alpha-1)\cdots(\alpha-j+1)$, for $j\geq 1$,
from which it can be verified directly that $b_0^{(\alpha)}=1$ and $b_j^{(\alpha)}<0$ for $j\geq 1$.
In particular, when $\alpha = 1$, the operator $\bar\partial_\tau^\alpha$ reduces to the standard backward difference quotient:
$$  \bar\partial_\tau^1 \varphi^n  =  \frac{\varphi^n - \varphi^{n-1}}{\tau}, $$
and the scheme \eqref{eqn:CQ-BE} reduces to the standard backward Euler scheme.

Using the superposition principle, the time stepping solution in \eqref{eqn:step-0} could be written in the operational form as \cite{ZhangZhou:2020}
\begin{equation} \label{eqn:sol-rep-step}
\begin{aligned}
u^n(q)%
&= F_\tau(n;q)(v-D(q)b) + D(q)b +\tau\sum_{j=1}^{n} E_\tau(j;q) f\\
&= F_\tau(n;q)(v-D(q)b) + D(q)b +(I-F_\tau(n;q))A(q)^{-1} f.
\end{aligned}
\end{equation}
Here the time discrete operators $F_\tau(n;q)$ and $E_\tau(n;q)$ are defined by the discrete inverse Laplace transform:
\begin{equation}\label{eqn:op-step}
\begin{aligned}
F_\tau(n;q) &= \frac{1}{2\pi\mathrm{i}}\int_{\Gamma_{\theta,\sigma}^\tau } e^{zt_n} {e^{-z\tau}} \delta_\tau(e^{-z\tau})^{\alpha-1}({ \delta_\tau(e^{-z\tau})^\alpha}+A(q))^{-1}\,\d z,\\
E_\tau(n;q) &= \frac{1}{2\pi\mathrm{i}}\int_{\Gamma_{\theta,\sigma}^\tau } e^{zt_n} e^{-z\tau}({ \delta_\tau(e^{-z\tau})^\alpha}+A(q))^{-1}\,\d z,
\end{aligned}
\end{equation}
with $\delta_\tau(\xi)=(1-\xi)/\tau$ and the contour
$\Gamma_{\theta,\sigma}^\tau :=\{ z\in \Gamma_{\theta,\sigma}:|\Im(z)|\le {\pi}/{\tau} \}$ where $\theta\in(\pi/2,\pi)$ is close to $\pi/2$
(oriented with an increasing imaginary part).
The next lemma gives elementary properties of the kernel $\delta_\tau(e^{-z\tau})$. The detailed proof has been given in
\cite[Lemma B.1]{JinLiZhou:SISC2017}.
\begin{lemma}\label{lem:delta}
For a fixed $\theta'\in(\pi/2,\pi/\alpha)$, there exists $\theta\in(\pi/2,\pi)$ and
positive constants $c,c_1,c_2$ independent of $\tau$ such that for all $z\in \Gamma_{\theta,\sigma}^\tau$,
\begin{equation*}
\begin{aligned}
& c_1|z|\leq
|\delta_\tau(e^{-z\tau})|\leq c_2|z|,
&&\delta_\tau(e^{-z\tau})\in \Sigma_{\theta'}, \\
& |\delta_\tau(e^{-z\tau})-z|\le c\tau |z|^{2},
 &&|\delta_\tau(e^{-z\tau})^\alpha-z^\alpha|\leq c\tau |z|^{1+\alpha}.
 \end{aligned}
\end{equation*}
\end{lemma}

For any $q\in \Q$, Lemma \ref{lem:delta} and resolvent estimate of elliptic operator \eqref{eqn:resol} immediately lead to
\begin{equation}\label{eqn:relso-step}
\|(\delta_\tau(e^{-z\tau})^\al+A(q))^{-1}\| \le C\min(|z^{-\al}|,\lambda^{-1}),  \quad \forall z \in \Sigma_{\phi},
  \,\,\,\forall\,\phi\in(0,\pi),
\end{equation}
for a constant $C$ independent of $q$.
Next we give some useful properties of $F_\tau(n;q)$ and $E_\tau(n;q)$.

The first lemma provides an estimate for $F_\tau(n;q)- F(t_n;q)$.
It has been proved in the earlier work \cite[Lemma 4.2, eq. (4.7)]{ZhangZhou:2020},
so we omit its proof here.
\begin{lemma}\label{lem:fully-approx}
Let $F_\tau(n;q)$ and $E_\tau(n;q)$ be defined as in \eqref{eqn:op-step}, and  $\lambda$ be the smallest  eigenvalue of $-\Delta$ with homogeneous Dirichlet boundary condition. Then  for $q\in \mathcal Q$, there holds
\begin{equation*}
\| (F_\tau(n;q)- F(t_n;q))v\|_{L^2\II}\le c \,n^{-1} \min(1,\lambda^{-1}t_n^{-\alpha})\|v\|_{L^2\II}\quad \text{for all} ~~n\ge 1,
\end{equation*}
and
\begin{equation*}
\| A(q)(F_\tau(n;q)- F(t_n;q))v\|_{L^2\II}\le c \,n^{-1} t_n^{-\alpha} \|v\|_{L^2\II}\quad \text{for all} ~~n\ge 1,
\end{equation*}
where the constants are independent of $q$, $\tau$ and $t_n$.
\end{lemma}
%\begin{proof}
%Since the potential $q$ satisfies $0\le q\le M_0$,
%we apply Lemma \ref{lem:delta} to derive the (discrete) resolvent estimate
%\begin{equation}\label{ful:relso}
%\|(\delta_\tau(e^{-z\tau})^\al+A(q))^{-1}\|\le \min(|z^{-\al}|,\lambda^{-1}),
%\end{equation}
%Then we apply the same argument in \cite[lemma 4.2]{ZhangZhou:2020} to arrive at
%\begin{equation}
%\label{eqn:Ftau:smooth}
%\|F_\tau(n;q)-F(t_n;q)\|\le c\frac{1}{1+\lambda t_n^\al}n^{-1} = c\frac{\tau t_n^{-1}}{1+\lambda t_n^\al},
%\end{equation}
%which implies the error estimate
%\end{proof}

The next lemma provides some smoothing and asymptotic properties of operators  $F_\tau(t;q)$ and $E_\tau(t;q)$.
This is a discrete analogue to Lemma \ref{lem:op}. The proof follows from the solution representation
\eqref{eqn:sol-rep-step}-\eqref{eqn:op-step},
Lemma \ref{lem:delta}, the resolvent estimate \eqref{eqn:relso-step},
and the same argument of the proof of Lemma \ref{lem:op} in
\cite[Theorem 6.4 and 3.2]{Jin:2021book}.

\begin{lemma}\label{lem:op-step}
Let $F_\tau(n;q)$ and $E_\tau(n;q)$ be defined as  \eqref{eqn:op-step}, and $\lambda$ be the smallest eigenvalue of $-\Delta$ with homogeneous boundary condition.
Then for $q\in \mathcal Q$, there holds
\begin{equation*}
\|A(q) F_\tau(n;q)v\|_{L^2\II} + t_n^{1-\alpha}\|A(q) E_\tau(n;q)v\|_{L^2\II} \le c  t_n^{-\al}\|v\|_{L^2\II}
\end{equation*}
and
\begin{equation*}
\|F_\tau(n;q)v\|_{L^2\II}+ t_n^{1-\al} \|E_\tau(n;q)v||_{L^2\II}\le c\min(1,\lambda^{-1} t_n^{-\al})\|v\|_{L^2\II},~~ n\ge 1.
\end{equation*}
Here $c$ is the generic constant independent of $\tau$, $t_n$ and $q$.
\end{lemma}
\begin{proof}
The asymptotics of $A(q)F_\tau(n;q)$ could be derived directly
from Lemmas \ref{lem:op} and \ref{lem:fully-approx}:
\begin{equation*}
\begin{aligned}
\|A(q)F_\tau(n;q)v\|_{L^2\II}& \le \|A(q)(F_\tau(n;q) - F(t_n;q))v\|_{L^2\II} +  \|A(q)F(t_n;q)v\|_{L^2\II}\\
&\le c  (n^{-1} + 1) t_n^{-\al}\|v\|_{L^2\II}  \le ct_n^{-\al}\|v\|_{L^2\II}.
\end{aligned}
\end{equation*}
Similarly, for $F_\tau(n;q)$, we apply Lemmas \ref{lem:op} and \ref{lem:fully-approx} again to derive
\begin{equation*}
\begin{aligned}
\|F_\tau(n;q)v\|_{L^2\II}& \le \|(F_\tau(n;q) - F(t_n;q))v\|_{L^2\II} +  \|F(t_n;q)v\|_{L^2\II}\\
&\le c    (n^{-1} + 1) \min(1,\lambda^{-1} t_n^{-\al})\|v\|_{L^2\II}
\le c \min(1,\lambda^{-1} t_n^{-\al}) \|v\|_{L^2\II}.
\end{aligned}
\end{equation*}

Next, we turn to the estimate of $A(q)E_\tau(n;q)$.
Using the representation \eqref{eqn:op-step}, resolvent estimate \eqref{eqn:relso-step}
and Lemma \ref{lem:delta}, we derive
\begin{equation*}
\begin{aligned}
\|A(q)E_\tau(n;q)v\|_{L^2\II} &\le  c\int_{\Gamma_{\theta,\sigma}^\tau} |e^{zt_n}| |e^{-z\tau}| \|A(q)(\delta_\tau(e^{-z\tau})^\al+A(q))^{-1}v\|_{L^2\II} |\d z|\\
 &\le  c\int_{\Gamma_{\theta,\sigma}^\tau} |e^{zt_n}|  \Big( \| v \|_{L^2\II} +  |\delta_\tau(e^{-z\tau})^\al|\| (\delta_\tau(e^{-z\tau})^\al+A(q))^{-1}v\|_{L^2\II} \Big)|\d z|\\
&\le   c \|v\|_{L^2\II} \int_{\Gamma_{\theta,\sigma}^\tau} |e^{zt_n}|  |\d z|
 \le   c\|v\|_{L^2\II} \left(\int_\sigma^\infty e^{-c\rho  t_n} d\rho + c\sigma\int_{-\theta}^\theta  \d \psi\right)\le c\sigma.
\end{aligned}
\end{equation*}
Then we let  $\sigma = t_n^{-1}$ to derive the desired estimate for $A(q)E_\tau(n;q)$.

The estimate for $E_\tau(n;q)$ could be derived using similar argument. By letting
$\sigma = t_n^{-1}$, we apply the resolvent estimate \eqref{eqn:relso-step}
and Lemma \ref{lem:delta} to deduce
\begin{equation*}
\begin{aligned}
\|E_\tau(n;q)v\|_{L^2\II} &\le c\int_{\Gamma_{\theta,\sigma}^\tau} |e^{zt_n}|  \|(\delta_\tau(e^{-z\tau})^\al+A(q))^{-1}v\|_{L^2\II} |\d z|\\
&\le c\|v\|_{L^2\II} \int_{\Gamma_{\theta,\sigma}^\tau} |e^{zt_n}| \min(|z|^{-\al},\lambda^{-1})|\d z|\\
&\le c\|v\|\min(t_n^{\al-1},\lambda^{-1}t_n^{-1}).
\end{aligned}
\end{equation*}
Then we complete the  proof of Lemma \ref{lem:op-step}.
\end{proof}

Next, we are ready to show some \textsl{a priori} estimate of the time stepping solution.
\begin{lemma}\label{lem:bDalun}
Let  Assumption \ref{ass:cond-1} be valid and $q\in \Q$. Then the solution $u^n(q)$ to the time stepping scheme \eqref{eqn:step-0} satisfies
 \begin{equation*}
 \| u^n (q) \|_{L^\infty\II}  \le c ~~\text{for all}~~n=1,2,\ldots,N.
 \end{equation*}
Moreover, there holds for all $s\in[0,2]$,
 \begin{equation*}
  \|\bDal^\alpha u^n(q)\|_{H^s\II}  \le c \min(t_n^{-s\alpha/2},t_n^{-\alpha}) ~~\text{for}~~n=1,2,\ldots,N.
 \end{equation*}
Here the generic constants are independent of $\tau$, $t_n$ and $q$.
\end{lemma}
\begin{proof}
Using the solution representation \eqref{eqn:sol-rep-step} and triangle inequality we arrive at
\begin{align*}
\| u^n(q) \|_{H^2\II} & \le \| F_\tau(n;q)(v-D(q)b) + D(q)b +(I-F_\tau(n;q))A(q)^{-1} f  \|_{H^2\II} \\
&\le \| F_\tau(n;q)(v-D(q)b)\|_{H^2\II} + \| D(q)b \|_{H^2\II} + \|(I-F_\tau(n;q))A(q)^{-1} f  \|_{H^2\II}.
\end{align*}
We use the norm equivalence \eqref{eqn:equiv-n} and Lemma \ref{lem:op-step}  to obtain
\begin{align*}
  \| F_\tau(n;q)(v-D(q)b)\|_{H^2\II} &\le c\Big(  \|  F_\tau(n;q) A(q) (v-D(q)b)\|_{L^2\II} +   \| F_\tau(n;q)(v-D(q)b)\|_{L^2\II} \Big)\\
  &\le  c\Big(  \|  A(q) (v-D(q)b)\|_{L^2\II} +   \|  v-D(q)b \|_{L^2\II} \Big) \\
  &\le c  \|   v-D(q)b \|_{H^2\II}  \le c \Big(  \|   v \|_{H^2\II} +   \|  D(q)b \|_{L^2\II} \Big).
\end{align*}
Then the estimate \eqref{eqn:Dqb2} implies
$$\| F_\tau(n;q)(v-D(q)b)\|_{H^2\II} \le c (\| v \|_{H^2\II} + \| b \|_{H^\frac{3}{2}(\partial\Omega)}).$$
 This combined with Sobolev embedding theorem yields $ \| u^n(q) \|_{L^\infty\II} \le c$
  where the constant $c$ is independent of $\tau$, $t_n$ and $q$.

Next,
we let $w^n(q) = \bar\partial_\tau^\al u^n(q)$. By a simple computation, we obtain that $w^n(q) \in H_0^1\II$ and
\begin{equation}\label{eqn:bDalu-step}
\bar\partial_\tau^\alpha w^n(q) + A(q)w^n(q) = 0~~ \text{for all}~~ 1\le n\le N\quad \text{and}\quad w^0(q)= f+ \Delta v - qv.
\end{equation}
Then the solution representation \eqref{eqn:sol-rep-step} leads to
\begin{equation} \label{eqn:sol-bDalu-step}
w^n(q) = \bar\partial_\tau^\al u^n(q) = F_\tau(n;q)(f+\Delta v - qv).
\end{equation}
Applying Lemma \ref{lem:op-step} and the condition $q\in\Q$, we obtain
\begin{equation*}
\begin{aligned}
\| \bar\partial_\tau^\al  u^n(q) \|_{L^2\II} = \|F_\tau(n;q)(f+\Delta v - qv )\|_{L^2\II}
 \le c \min(1,t^{-\alpha}) \big(\|v\|_{H^2\II} + \|f\|_{L^2\II} \big).
\end{aligned}
\end{equation*}
Next, the norm equivalence \eqref{eqn:equiv-n}  and Lemma \ref{lem:op-step} yield
\begin{align*}
 \|  \bar\partial_\tau^\al u^n(q) \|_{H^2\II}
  \le c\big(\|\bar\partial_\tau^\al u^n(q)\|_{L^2\II} + \|A(q)\bar\partial_\tau^\al u^n(q)\|_{L^2\II} \big) \le c t_n^{-\al}(\|v\|_{H^2\II}+\|f\|_{L^2\II}).
\end{align*}
Here $c$ is independent of $\tau$, $t_n$ and $q$. The case that $s\in(0,1)$ follows immediately by interpolation.
This completes the proof of the lemma.
\end{proof}

Finally, we shall provide a useful \textsl{a priori} error estimate for $\bar \partial_\tau^\alpha u^n(q) - \partial_t^\alpha u(t;q)$.

\begin{lemma}\label{lem:un-err}
Let Assumption \ref{ass:cond-1} be valid and $q\in \Q$. Let $u^n(q)$ and $u(t;q)$ be the solutions to \eqref{eqn:step-0}
and \eqref{eqn:pde}, respectively. Then there holds
$$  \| \bar \partial_\tau^\alpha  u^n(q) -  \partial_t^\alpha u(t_n;q) \|_{L^2\II} \le c \tau t_n^{-\al-1} $$
with the constant independent of $q,\tau$ and $n$.
\end{lemma}
\begin{proof}
Combining \eqref{eqn:Dalu-4} with \eqref{eqn:bDalu-step}, we obtain
\begin{equation*}
\bar\partial_\tau^\al u^n(q)- \Dal u(t_n;q) =  (F_\tau(n;q)-F(t_n;q))(\Delta v - q v +  f).
\end{equation*}
Then we apply Lemma \ref{lem:fully-approx} with $s=0$ and note that $q\in\Q$ to derive
\begin{equation*}
\|\bar\partial_\tau^\al u^n(q)- \Dal u(t_n;q)\|_{L^2\II}\le c \tau t_n^{-\al-1}\Big(\|v\|_{H^2\II} + \|f\|_{L^2\II}\Big).
\end{equation*}
This completes the proof of the lemma.
\end{proof}

 \subsection{Fully discrete scheme}
In this section, we shall discuss the completely discrete scheme to solve the inverse potential problem.
We use the convolution quadrature for the time discretization and use Galerkin
finite element method  for the space discretization. To begin with, we introduce some settings for the finite element methods.

To illustrate the main idea, we consider the square region
$\Omega = (a, b)^d \subset\R^d$, with $1\le d\le 3$ and the discussion could be extended to
general convex polyhedral domain.
For all $i=1,\dots, d$, we denote by $a=x_{0}<x_{1}<\dots<x_{M}=b$
a partition of the interval $[a, b]$ with a uniform mesh size $\displaystyle h=x_{i}-x_{i-1} = (b-a)/M$
for all $i=1,\dots,M$.
Then domain $\Omega$ is now separated into $M^d$ subrectangles by all grid points
$(x_{j_1 },  \ldots, x_{j_d})$,
with $0 \le j_i \le M$ and $i=1,\dots, d$. We denote this partition by $\T_h$,
and note that $\displaystyle h$ is the mesh size of the partition $\T_h$.

Then we apply the tensor-product Lagrange finite elements on the partition $\T_h$.
Let $Q_1$ be the space of polynomials in the variables $x_1, \ldots, x_d$, with real coefficients and
of degree at most one in each variable, i.e.,
\begin{equation*}
Q_1 = \Big\{   \sum_{0\le \beta_1,\beta_2,\ldots,\beta_d\le 1} c_{\beta_1 \beta_2\ldots \beta_d} x_1^{\beta_1} \cdots x_d^{\beta_d},
\quad \text{with}~~ c_{\beta_1 \beta_2\ldots \beta_d} \in\mathbb{R} \Big\}.
\end{equation*}
The $H^1$-conforming tensor-product finite element space, denoted by $X_h$, is defined as
\begin{equation}\label{eqn:Xh}
 X_h = \{ v\in H^1(\Omega): v|_K \in Q_1 \,\, \text{for all} \,\, K\in \T_h \}.
\end{equation}
Besides, we define
\begin{equation}\label{eqn:Xh0}
X_h^0 = X_h \cap H_0^1\II = \{ v\in H_0^1(\Omega): v|_K \in Q_1 \,\, \text{for all} \,\, K\in\T_h \}.
\end{equation}

We let $\mathcal{I}_h$ denote the Lagrange interpolation operator associated with the finite element space $X_h$. It satisfies
the following error estimates for $s=1,2$ and $1 \le p\le \infty$ with $sp>d$
\cite[Theorem 1.103]{ern-guermond}:
\begin{align}\label{eqn:int-err}
%  \|v-\mathcal{I}_hv\|_{L^2(\Omega)} + h\|v-\mathcal{I}_hv\|_{H^1(\Omega)} & \leq ch^2 \|v\|_{H^2(\Omega)}, \quad \forall v\in H^2(\Omega),\label{eqn:int-err-2}\\
  \|v-\mathcal{I}_hv\|_{L^p(\Omega)}  +  h\|v-\mathcal{I}_hv\|_{W^{1,p}(\Omega)}  \leq ch^s\|v\|_{W^{s,p}(\Omega)}, \quad \forall v\in W^{s,p}(\Omega).
\end{align}
Similarly, we let $\Ih^\partial$ denote the Lagrange interpolation operator on the boundary.

We define the orthogonal $L_2$-projection $P_h:L^2(\Omega)\to X_h^0$ and
the Ritz projection $R_h(q):H^1_0(\Omega)\to X_h^0$ by
\begin{equation*}%\label{2.ritz}
  \begin{aligned}
    (P_h \psi,\chi_h)&=(\psi,\chi_h),&& \quad\forall \chi\in X_h^0,\\
    (\nabla R_h(q) \psi,\nabla\chi_h)&=(\nabla \psi,\nabla\chi_h) + (q\psi,\chi_h),&& \quad \forall \chi\in X_h^0.
  \end{aligned}
\end{equation*}
It is well-known that the operators $P_h$ and $R_h(q)$ (with $q\in \Q$) have the following
approximation property, cf. \cite[Lemma 1.1]{Thomee:2006} or
\cite[Theorems 3.16 and 3.18]{ern-guermond}, for $s\in[1,2]$,
\begin{equation}\label{ph-bound}
  \begin{aligned}
  \|P_h \psi-\psi\|_{L^2\II}+  \|R_h(q) \psi-\psi\|_{L^2\II}& \le c h^s\| \psi\|_{H^s(\Omega)}, \quad \forall \psi\in H^s(\Omega)\cap H_0^1(\Omega).\\
 \end{aligned}
\end{equation}
Noting that $q\in\Q$,  the constant $c$ is independent of $q$.

Let $\gamma_0$ be the trace
operator \cite[Section B.3.5]{ern-guermond}, and the set
$ X_{h}^\partial = \left\{\gamma_0(\chi_h): \  \chi_h \in X_h \right\}. $
 Now we introduce a discrete operator $D_h (q) : X_{h}^\partial \rightarrow X_h$ such that
$w_h = D_h {(q)} b_h$
for $b_h \in  X_{h}^\partial$ satisfies
\begin{equation*}
  (\nabla w_h , \nabla \chi_h) + (qw_h,\chi_h)= 0\quad \text{for all}~~ \chi_h \in X_h^0,\qquad \text{and} ~~\gamma_0(w_h) = b_h.
\end{equation*}
Then for any $q\in\Q$ and $b\in H^2(\partial\Omega)$, there holds the estimate \cite[Lemma 3.28]{ern-guermond}
\begin{equation}\label{eqn:err-Dh}
\| D(q)b - D_h(q) \Ih^\partial b  \|_{L^2\II} \le c h^2 \| b \|_{H^2(\partial\Omega)}.
\end{equation}

%Besides, we  recall the following useful inverse inequality of finite element functions (see e.g., \cite[Corollary 1.141]{ern-guermond}).
%\begin{lemma}\label{lem:inv-ineq}
%Let $X_h$ and $X_h^0$ be the finite dimensional spaced defined in \eqref{eqn:Xh} and \eqref{eqn:Xh0} respectively.
%Then there hold the estimates
%\begin{equation*}
%\begin{aligned}
% \| \psi_h  \|_{L^p\II} & \le C h^{d(\frac1p-\frac1q)}\| \psi_h  \|_{L^q\II} \quad \text{for all} ~~ 1\le q\le p \le \infty ~~\text{and}~~\psi_h \in X_h.\\
% \| \Delta_h \phi_h  \|_{L^2\II} &+ h^{-1}\| \nabla \phi_h  \|  \le C h^{-2}\| \phi_h  \|_{L^2\II} \quad \text{for all} ~~ \psi_h \in X_h^0.
%\end{aligned}
%\end{equation*}
%\end{lemma}

To discretize the problem \eqref{eqn:pde}, we consider the weak formulation  to
find $u(t) \in H^1(\Omega)$ such that for all $\fy \in H^1_0(\Omega)$ and $t>0$,
\begin{equation*}
(\Dal u(t), \fy) + (\nabla u(t),\nabla \fy) + (q u(t), \fy) =(f,\fy), ~~\text{with}~~ u(\cdot,t)=b~~\text{in}~  \partial\Omega ~~\text{and}~~u(0)=v .
\end{equation*}
Then the fully discrete scheme for \eqref{eqn:pde} reads: find $u_h^n(q)\in X_h$ for $t\ge0$
such that $\gamma_0(u_h^n(q)) = \Ih^\partial b$ on $\partial\Omega$ and for all $\fy_h \in X_h^0 $ and $n=1,2,\ldots,N$,
\begin{equation}\label{eqn:fully-i}
  ( \bDal^\alpha u_{h}^n(q),\fy_h)  +  (\nabla u_{h}^n(q), \nabla \fy_h) +(q u_{h}^n(q), \fy_h) = (f,\fy_h)\quad \text{with}~~u_h^0(q)=\mathcal{I}_h v.
\end{equation}

For $q\in \Q$ we define the discrete operator $A_h(q):\, X_h^0\to X_h^0$ such that
$$(A_h(q)\xi_h,\chi_h)  = (\nabla\xi_h,\nabla\chi_h) + (q \xi_h,\chi_h)  \quad  \text{for all} ~\  \xi_h,\chi_h \in X_h^0.$$
Then by splitting the fully discrete solution to \eqref{eqn:fully-i} as $u_h^n(q) = \fy_h^n(q) + D_h(q)\Ih^\partial b$,
we observe that $\fy_h^n(q)\in X_h^0$ satisfies
\begin{equation*}
  \bDal^\alpha \fy_h^n(q)   +  A_h(q)\fy_h^n(q)  = P_h f  \qquad \text{for}~ t >0,
\end{equation*}
with $\fy_h^0(q) = \mathcal{I}_h v - D_h(q)\Ih^\partial b$.
In particular, we define $\Delta_h=-A_h(0)$.
Then analogue to \eqref{eqn:sol-rep-step}, the fully discrete solution in
\eqref{eqn:fully-i} could be written in the operational form
\begin{equation} \label{eqn:sol-rep-fully}
\begin{aligned}
u_h^n(q) &=  F_\tau^h(n;q)\big(\mathcal{I}_h  v-D_h(q)\mathcal{I}_h^\partial  b\big) + D_h(q)\mathcal{I}_h^\partial b + \tau \sum_{j=1}^n E_\tau^h(j;q)  P_h f\\
&=  F_\tau^h(n;q)\big(\mathcal{I}_h  v-D_h(q)\mathcal{I}_h^\partial  b\big) + D_h(q)\mathcal{I}_h^\partial b + (I- F_\tau^h(n;q)) A_h(q)^{-1} P_h f,
\end{aligned}
\end{equation}
where the fully discrete operators $F_\tau^h(n;q)$ and $E_\tau^h(n;q)$ are defined as
\begin{equation}\label{eqn:op-fully}
\begin{aligned}
F_\tau^h(n;q) &= \frac{1}{2\pi\mathrm{i}}\int_{\Gamma_{\theta,\sigma}^\tau } e^{zt_n} {e^{-z\tau}} \delta_\tau(e^{-z\tau})^{\alpha-1}({ \delta_\tau(e^{-z\tau})^\alpha}+A_h(q))^{-1}\,\d z,\\
E_\tau^h(n;q) &= \frac{1}{2\pi\mathrm{i}}\int_{\Gamma_{\theta,\sigma}^\tau } e^{zt_n} e^{-z\tau}({ \delta_\tau(e^{-z\tau})^\alpha}+A_h(q))^{-1}\,\d z.
\end{aligned}
\end{equation}
Let $\lambda$ be the smallest eigenvalue of $-\Delta$ with the homogeneous Dirichlet boundary condition, and $\lambda_h(q)$ be the smallest eigenvalue of discrete operator $A_h(q)$. Recalling that the finite element space $X_h^0$ is conforming in $H_0^1\II$ and $q\in\Q$ ,
 the Courant minimax principle implies the relation that
$ 0<\lambda \le  \lambda_h(0) \le \lambda_h(q)$.
Then we have the resolvent estimate for the (discrete) elliptic operator $A_h(q)$:
with fixed $\phi\in(0,\pi)$
\begin{equation*}
\|(\delta_\tau(e^{-z\tau})^\al+A_h(q))^{-1}\| \le C\min(|z^{-\al}|,\lambda_h(q)^{-1}) \le C\min(|z^{-\al}|,\lambda^{-1}),   \quad \forall z \in \Sigma_{\phi},
\end{equation*}
for a constant $C$ independent of $q$ and $h$.
This immediately indicates the following result
for the fully discrete scheme \eqref{eqn:fully-i}, similar to Lemmas \ref{lem:op} and \ref{lem:op-step}.

\begin{lemma}\label{lem:op-fully}
Let $F_\tau^h(n;q)$ and $E_\tau^h(n;q)$ be the operators defined  in \eqref{eqn:op-fully}.
Let $\lambda$ be the smallest eigenvalue of $-\Delta$ with homogeneous boundary condition.
Then for any $q\in \mathcal Q$ and $v_h\in X_h^0$,  there holds for $n\ge 1$,
\begin{align*}%\label{eqn:smooth-Ftau}
\|A_h(q) F_\tau^h(n;q)v_h\|_{L^2\II} + t_n^{1-\al} \|A_h(q) E_\tau^h(n;q)v_h\|_{L^2\II} &\le c  t_n^{-\al}\|v_h\|_{L^2\II},\\
\|F_\tau^h(n;q)v\|_{L^2\II}+ t_n^{1-\al} \|E_\tau^h(n;q)v_h||_{L^2\II}&\le c\min(1,\lambda^{-1} t_n^{-\al})\|v_h\|_{L^2\II}.
\end{align*}
Here $c$ is the generic constant independent of $\tau$, $t_n$ and $q$.
\end{lemma}

Next, we recall the following useful inverse inequality of finite element functions (see e.g., \cite[Corollary 1.141]{ern-guermond}).
\begin{lemma}\label{lem:inv-ineq}
Let $X_h$ and $X_h^0$ be the finite dimensional spaces defined in \eqref{eqn:Xh} and \eqref{eqn:Xh0} respectively.
Then we have the inverse estimates
\begin{equation*}
\begin{aligned}
 \| \psi_h  \|_{L^p\II} & \le C h^{d(\frac1p-\frac1q)}\| \psi_h  \|_{L^q\II} \quad \text{for all} ~~ 1\le q\le p \le \infty ~~\text{and}~~\psi_h \in X_h,\\
 \| \Delta_h \phi_h  \|_{L^2\II} &+ h^{-1}\| \nabla \phi_h  \|  \le C h^{-2}\| \phi_h  \|_{L^2\II} \quad \text{for all} ~~ \psi_h \in X_h^0.
\end{aligned}
\end{equation*}
\end{lemma}

Next, we intend to derive an \textsl{a priori} estimate for  $\bar \partial_\tau^\alpha u_h^n(q) -  \bar \partial_\tau^\alpha u^n(q)$.

\begin{lemma}\label{lem:uhn-err}
Let Assumption \ref{ass:cond-1} be valid and $q\in \Q$. Let $u^n(q)$ and $u_h^n(q)$ be the solutions to \eqref{eqn:step-0}
and \eqref{eqn:fully-i}, respectively. Then there holds for any $\epsilon\in(0,1)$,
$$ \|  \bar \partial_\tau^\alpha (u_h^n(q) -  u^n(q)) \|_{L^2\II} \le c h^{2-\epsilon} \max(t_n^{-\alpha},t_n^{-(1-\epsilon)\alpha}). $$
Here the constants are independent of $q,\tau$ and $n$.
\end{lemma}
\begin{proof}
First of all, we recall that $w^n(q) = \bar\partial_\tau^\al u^n(q) \in H_0^1\II$ and it satisfies \eqref{eqn:bDalu-step}.
Meanwhile, Assumption \eqref{assump:numerics} implies that the fully discrete approximation
$w_h^n(q) = \bar\partial_\tau^\al u_h^n(q)\in X_h^0$ satisfies
\begin{equation}\label{eqn:whn}
\bar\partial_\tau^\al w_h^n(q)+A_h(q)w_h^n(q) = 0 ,~~n\ge 1,~~\text{with}~~w_h^0(q) = P_h f - A_h(q)(\mathcal{I}_hv-D_h(q)\mathcal{I}_h^\partial b).
\end{equation}
To derive an estimate for $w_h^n(q) - w^n(q)$, we apply the splitting %Firstly we split
$$
w_h^n(q)-w^n(q) =  \big(w_h^n(q)- P_h w^n(q)\big)+ \big( P_h w^n(q)-w^n(q)\big)=: \theta_h^n+\rho^n.
$$
Then the bound of $\rho^n$ can be derived from  \eqref{eqn:int-err} and Lemma \ref{lem:bDalun} as
\begin{equation*}
\|\rho^n\|_{L^2\II}\le ch^2 \|\bar\partial_\tau^\al u^n(q)\|_{H^2\II}\le ch^2 t_n^{-\al}.
\end{equation*}
Next we turn to derive an estimate for $\theta_h^n \in X_h^0$, which satisfies
\begin{equation*}
\begin{aligned}
\bar\partial_\tau^\al \theta_h^n +A_h(q)\theta_h^n &= A_h(q)(R_h(q)-P_h)w^n(q) \quad \text{for all}~~ n=1,2,\ldots,N,\\
\theta_h^0 &= A_h(q)R_h(q)(v-D(q)b)-A_h(q)(\mathcal I_h v-D_h(q)\mathcal I_h^\partial b),
\end{aligned}
\end{equation*}
where we use the fact that $A_h(q)R_h(q) \psi = P_h A(q) \psi$ for $\psi\in H^2\II\cap H_0^1\II$.
By the representation \eqref{eqn:sol-rep-fully} we have
\begin{equation}\label{eqn:thetah-2}
\theta_h^n = F_\tau^h(n;q)\theta_h(0)+\tau \sum_{j=1}^{n} E_\tau^h(j;q)A_h(q)(R_h(q)-P_h)w^{n+1-j}(q) =:  I+II.
\end{equation}
From Assumption \ref{ass:cond-1}, we have
$v-D(q)b \in  H^2(\Omega)\cap H_0^1\II$. Then \eqref{eqn:int-err}, \eqref{ph-bound}, \eqref{eqn:err-Dh} and Lemma \ref{lem:op-fully} imply
\begin{equation*}
\begin{aligned}
\|I\|_{L^2\II}&\le ct_n^{-\alpha} \|R_h(q)(v-D(q)b) - (\Ih v-D_h(q)\Ih^\partial b) \|_{L^2\II} \\
&\le  ct_n^{-\alpha}\Big( \|(R_h(q)-I)(v-D(q)b) \|_{L^2\II} + \|v-\Ih v\|_{L^2\II} + \|D(q)b  -  D_h(q)\Ih^\partial b  \|_{L^2\II}\Big)\\
&\le c h^2 t_n^{-\alpha}\Big( \|v \|_{H^2\II} + \|b\|_{H^2(\partial\Omega)}\Big).
\end{aligned}
\end{equation*}
Now we turn to the estimate for the term $II$.
By Lemma \ref{lem:op-fully},
we have $$\|A_h(q)^s E_\tau^h(n;q) v_h\|_{L^2\II} \le c t_n^{(1-s)\alpha -1}\|v_h\|_{L^2\II}.$$
Meanwhile, the second inverse inequality in Lemma \ref{lem:inv-ineq} implies
$$ \|  A_h(q)^s v_h  \|_{L^2\II} \le c h^{-2s} \|v_h\|_{L^2\II}.$$
The fact $q\in\Q$ implies that the constant $c$ is independent of $q$.
Then we apply the above estimates combined with Lemma \ref{lem:bDalun} for $s=2-\epsilon$,
and obtain
\begin{equation*}
\begin{aligned}
\|II\|_{L^2\II}&\le  \tau \sum_{j=1}^{n} \|E_\tau^h(j;q)A_h(q)^{1-\epsilon/2}\| \, \|A_h(q)^{\epsilon/2}(R_h(q)-P_h)w^{n+1-j}(q)\|_{L^2\II}\\
&\le c\tau \sum_{j=1}^{n} t_{j}^{-1+\epsilon\alpha/2} \, \| (R_h(q)-P_h)w^{n+1-j}(q)\|_{L^2\II} h^{-\epsilon}\\
&\le ch^{2-\epsilon} \tau \sum_{j=1}^{n} t_{j}^{-1+\epsilon\alpha/2} \, \| w^{n+1-j}(q)\|_{H^{2-\epsilon}\II} \\
&\le ch^{2-\epsilon} \tau \sum_{j=1}^{n} t_{j}^{-1+\epsilon\alpha/2} t_{n+1-j}^{-\alpha+\epsilon\alpha/2} \le ch^{2-\epsilon} t_{n}^{-\alpha+\epsilon\alpha}.
\end{aligned}
\end{equation*}
This completes the proof of the lemma.
\end{proof}

%
%
%The next Result is a discrete analogue to Lemma \ref{lem:Dalu}.
%
%

The next result provides an \textsl{a priori} estimate for $  \bar \partial_\tau^\alpha  u_h^n(q_1) - \bar \partial_\tau^\alpha u_h^n(q_2) $,
which plays a key role in the stability analysis for the numerical solution of the inverse potential problem.
\begin{lemma}\label{lem:Dal-uhn}
Suppose that Assumption \ref{ass:cond-1} is valid and $q_1,q_2\in \Q$.
For $i=1,2$, let $u_h^n(q_i)$ be the solution to the fully discrete scheme \eqref{eqn:fully-i}, with potential $q_i$, respectively.
Then there holds for any positive parameter $\epsilon<\min(1,2-\frac{d}{2})$, %and spatial mesh size $h \le 1$
\begin{equation*}
\|  \bar \partial_\tau^\alpha (u_h^n(q_1) - u_h^n(q_2)) \|_{L^2\II} \le c \max(t_n^{-\alpha},t_n^{-(1-\epsilon)\alpha})\| q_1 - q_2 \|_{L^2\II},
\end{equation*}
where the constant $c$ is independent of $h$, $\tau$, $q_1$, $q_2$ and $t_n$.
\end{lemma}

\begin{proof}
We let $\theta_h^n = \bar\partial_\tau^\alpha (u_h^n(q_1)-u_h^n(q_2))$. Note that $\theta_h^n\in X_h^0$ and it satisfies
\begin{equation*}
\bar\partial_\tau^\al \theta_h^n +A_h(q_1)\theta_h^n =P_h[(q_2-q_1)\bar\partial_\tau^\al u_h^n(q_2)]~~\text{with}~~ \theta_h^0 = P_h[(q_2-q_1)\mathcal I_hv].
\end{equation*}
Now we apply the stability of $L^2$-projection $P_h$ to obtain
\begin{equation}\label{eqn:theta0}
\begin{split}
\| \theta_h(0) \|_{L^2\II} &\le \| (q_2-q_1) \Ih v \|_{L^2\II}
 \le \| q_2-q_1\|_{L^2\II} \| \Ih v \|_{L^\infty\II}\\
& \le \| q_2-q_1\|_{L^2\II} \| v \|_{L^\infty\II}.
 \end{split}
 \end{equation}
 Meanwhile, using the stability of $P_h$ and the inverse inequality in Lemma \ref{lem:inv-ineq} we arrive at
\begin{equation*}
\begin{aligned}
&\quad \|P_h[(q_2-q_1)\bar\partial_\tau^\alpha u^n(q_2)]\|_{L^2\II}\\
&\le c\|q_2-q_1\|_{L^2\II} \|\bar\partial_\tau^\alpha u^n(q_2)\|_{L^\infty\II}\\
&\le c\|q_2-q_1\|_{L^2\II}\left(\|\bar\partial_\tau^\al(u_h^n(q_2)-\Ih u^n(q_2)\|_{L^\infty\II}+\|\Ih \bar\partial_\tau^\al u^n(q_2)\|_{L^\infty\II}\right)\\
&\le c\|q_2-q_1\|_{L^2\II}\left(h^{-\frac d2}\|\bar\partial_\tau^\al(u_h^n(q_2)-\Ih u^n(q_2)\|_{L^2\II}+\| \bar\partial_\tau^\al u^n(q_2)\|_{L^\infty\II}\right).
\end{aligned}
\end{equation*}
Then we apply the Sobolev embedding theorem to derive that for $\epsilon<\min(1,2-\frac{d}{2})$,
\begin{equation*}
\begin{aligned}
&\quad \|P_h[(q_2-q_1)\bar\partial_\tau^\alpha u^n(q_2)]\|_{L^2\II}\\
& \le c\|q_2-q_1\|_{L^2\II}\left(h^{-\frac d2}\|\bar\partial_\tau^\al(u_h^n(q_2)-\Ih u^n(q_2)\|_{L^2\II}+\| \bar\partial_\tau^\al u^n(q_2)\|_{H^{2-\epsilon}\II}\right).
\end{aligned}
\end{equation*}
This together with Lemma \ref{lem:bDalun} leads to
\begin{equation*}
 \|P_h[(q_2-q_1)\bar\partial_\tau^\alpha u^n(q_2)]\|_{L^2\II}
 \le c\|q_2-q_1\|_{L^2\II}\left(h^{-\frac d2}\|\bar\partial_\tau^\al(u_h^n(q_2)-\Ih u^n(q_2))\|_{L^2\II}+ t_n^{-(1-\epsilon/2)\alpha} \right).
\end{equation*}
Then using  Lemmas \ref{lem:bDalun} and \ref{lem:uhn-err}, we obtain for $\epsilon<\min(1,2-\frac{d}{2})$,
\begin{equation*}
\begin{aligned}
&\quad h^{-\frac d2}\|\bar\partial_\tau^\al(u_h^n(q_2)-\Ih u^n(q_2)\|_{L^2\II} \\
&\le h^{-\frac d2}\Big(\|\bar\partial_\tau^\al(u_h^n(q_2)- u^n(q_2)\|_{L^2\II} +
\|\bar\partial_\tau^\al(\Ih u^n(q_2) - u^n(q_2))\|_{L^2\II}\Big)\\
&\le c h^{2-\frac d2-\epsilon} \Big(t_n^{-(1-\epsilon/2)\alpha} + \|\bar\partial_\tau^\al  u^n(q_2) \|_{H^{2-\epsilon}\II} \Big)
\le c h^{2-\frac d2-\epsilon} t_n^{-(1-\epsilon/2)\alpha}.
\end{aligned}
\end{equation*}
As a result, we conclude that for $\epsilon<\min(1,2-\frac{d}{2})$,
\begin{equation}\label{eqn:theta1-1}
\|P_h[(q_2-q_1)\bar\partial_\tau^\alpha u^n(q_2)]\|_{L^2\II}\le ct_n^{-(1-\epsilon/2)\alpha}\|q_2-q_1\|_{L^2\II}.
\end{equation}

Now, using the  representation  \eqref{eqn:sol-rep-fully}, we derive
\begin{equation*}
\theta_h^n =  F_\tau^h(n;q_1)\theta_h^0+\tau\sum_{j=1}^n E_\tau^h(j;q_1) P_h[(q_2-q_1)\bar\partial_\tau^\alpha u_h^{n+1-j}(q_2)].
\end{equation*}
Then  Lemma \ref{lem:op-fully} indicates that for any $\epsilon<\min(1,2-\frac{d}{2})$,
\begin{equation*}
\begin{aligned}
\|\theta_h^n\|_{L^2\II}&\le \|F_\tau^h(n;q_1)\theta_h(0)\|_{L^2\II} +\tau\sum_{j=1}^n \|E_\tau^h(j;q_1) P_h[(q_2-q_1)\bar\partial_\tau^\alpha u_h^{n+1-j}(q_2)]\|_{L^2\II}\\
&\le ct_n^{-\al}\|\theta_h^0\|_{L^2\II} + \tau\sum_{j=1}^n t_j^{-1+\epsilon\alpha/2}\|P_h[(q_2-q_1)\bar\partial_\tau^\alpha u_h^{n+1-j}(q_2)]\|_{L^2\II}.
\end{aligned}
\end{equation*}
This combined with \eqref{eqn:theta0} and \eqref{eqn:theta1-1} leads to the desired result.
\end{proof}

\subsection{The inverse potential problem: numerical reconstruction and error estimate}
In this part, we shall design a robust completely discrete scheme for the recovery of the potential.
Throughout this section, we need the following assumption.
\begin{assumption}\label{assump:numerics}
We assume that the exact potential $q^\dag$ and observational data $g_\delta$ satisfy the following conditions:
\begin{itemize}
\item[(i)] $q^\dag\in \Q\cap W^{1,p}\II$ with some $p>\max(d,2)$ and $q^\dag|_{\partial\Omega}$ is \textsl{a priori} known;
\item[(ii)]  $g_\delta(x) \in C(\overline \Omega)$ is noisy and it satisfies
$\gamma_0(g_\delta) = \gamma_0(g) = b$ and $\|  g_\delta - g  \|_{C(\overline \Omega)} = \delta$.
\end{itemize}
\end{assumption}

Based on Assumption \ref{ass:cond-1} and Assumption \ref{assump:numerics} (i), we have $f,q^\dag\in W^{1,p}\II$ for some $p>\max(d,2)$. Moreover,  Lemma \ref{lem:u-reg} indicates that $ \Dal u(T,q^\dag), u(T;q^\dag)\in H^2\II \subset W^{1,p}\II \subset L^\infty\II$ with $p\in(\max(d,2),6)$. Therefore, we conclude that
for some $p \in(\max(d,2),6)$
\begin{equation}\label{eqn:Deltau}
    \Delta g(x) = \Delta u(T;q^\dag) =  -f+ \Dal u(T,q^\dag) + q^\dag u(T;q^\dag) \in W^{1,p}\II.
\end{equation}

Besides, Assumption \ref{assump:numerics} (i) and (ii) imply
$$\gamma_0(\Delta g) = \gamma_0(qg - f) =  \gamma_0(q) b - \gamma_0(f),$$
which is \textsl{a priori} known.
Note that $\Delta g_\delta$ might not be well-defined in $L^2\II$.
Therefore, we need a numerical approximation to the unknown function $\Delta g$.
Now we define a function $\psi_h \in X_h$ such that
\begin{equation}\label{eqn:psih}
\gamma_0(\psi_h)= \Ih^\partial( \gamma_0(q) b - \gamma_0(f) ) \quad \text{and}\quad (\psi_h,\phi_h)=-(\nabla \Ih g_\delta, \nabla \phi_h)~~\text{for all}~~ \phi_h \in X_h^0.
\end{equation}
Then we have $\psi_h \approx \Delta g$
provided that $h=O(\delta^\frac13)$. %according to the noise level $\delta$.
This is given by the following lemma.

\begin{lemma}\label{lem:reg-err}
Suppose that Assumptions  \ref{ass:cond-1} and  \ref{assump:numerics} are valid.
Let $\psi_h \in X_h$ be the function defined in \eqref{eqn:psih}.
Then there holds
\begin{equation*}
\|  \psi_h - \Delta g  \|_{L^2\II} \le c \Big(\frac{\delta}{h^2} + h\Big),% \|\Delta g\|_{H^2\II}
\end{equation*}
where the constant $c$ is independent of $h$ and $\delta$.
\end{lemma}
\begin{proof}
To derive the estimate, we define the auxiliary function
$ \tilde \psi_h\in X_h$
such that
\begin{equation*}
\gamma_0( \tilde \psi_h)=\Ih^\partial( \gamma_0(q) b - \gamma_0(f) ) \quad \text{and}\quad
( \tilde \psi_h,\phi_h)=-(\nabla \Ih g, \nabla \phi_h)~~\text{for all}~~ \phi_h \in X_h^0.
\end{equation*}
Then we consider the split
\begin{align*}
 \psi_h - \Delta g
 =  (\psi_h-\tilde \psi_h ) +   (\tilde \psi_h  -   \Ih \Delta g )
 + (\Ih \Delta g - \Delta g ) .
\end{align*}
According to the definition of $\psi_h$ and $\tilde \psi_h$, we know $\psi_h-\tilde \psi_h \in X_h^0$.
Then the inverse inequality  in Lemma \ref{lem:inv-ineq} implies
\begin{align*}
 \| \psi_h-\tilde \psi_h \|_{L^2\II} &= \sup_{\phi_h \in X_h^0}\frac{(\psi_h-\tilde \psi_h, \phi_h)}{\| \phi_h \|_{L^2\II}}
 = \sup_{\phi_h \in X_h^0}\frac{(\nabla (\Ih g - \Ih g_\delta), \nabla \phi_h)}{\| \phi_h \|_{L^2\II}}\\
 &\le c h^{-2} \|  \Ih g - \Ih g_\delta   \|_{L^2\II} \le c \delta h^{-2}.
\end{align*}
Meanwhile, using the fact that $\Delta g \in W^{1,p}\II$ for some $p\in(\max(2,d),6)$ by \eqref{eqn:Deltau}, the approximation property of $\Ih$ in \eqref{eqn:int-err} implies
\begin{align*}
  \|\Ih \Delta g - \Delta g  \|_{L^2\II} \le   \|\Ih \Delta g - \Delta g  \|_{L^p\II} \le c h  \| \Delta g \|_{W^{1,p}\II}.
\end{align*}
Finally, according to the definition of $\tilde \psi_h$, we know $\tilde \psi_h  -   \Ih \Delta g \in X_h^0$, and hence
\begin{align*}
  \| \tilde \psi_h  -   \Ih \Delta g  \|_{L^2\II} &=
  \sup_{\phi_h \in X_h^0} \frac{(\tilde \psi_h  -   \Ih \Delta g, \phi_h)}{\|   \phi_h  \|_{L^2\II}}
  = \sup_{\phi_h \in X_h^0} \frac{(\tilde \psi_h  -    \Delta g, \phi_h) + ( \Delta g -   \Ih \Delta g, \phi_h)}{\|   \phi_h  \|_{L^2\II}} \\
 &=  \sup_{\phi_h \in X_h^0} \frac{(\nabla  (g-\Ih g) , \nabla \phi_h)}{\|   \phi_h  \|_{L^2\II}} + ch  \| \Delta g \|_{W^{1,p}\II} .
\end{align*}
Then the superconvergence \cite[Theorem 4.1]{LinLin:2006}
 \begin{align*}
  (\nabla  (g-\Ih g) , \nabla \phi_h) \le ch^2 \| g \|_{H^3\II} \| \phi_h \|_{H^1\II},
\end{align*}
together with the inverse inequality in Lemma \ref{lem:inv-ineq} leads to
\begin{align*}
  \| \tilde \psi_h  -   \Ih \Delta g  \|_{L^2\II}
 &\le   \sup_{\phi_h \in X_h^0} \frac{ c h^2   \| \phi_h \|_{H^1\II}}{\|   \phi_h  \|_{L^2\II}} + ch \le \sup_{\phi_h \in X_h^0} \frac{ ch   \| \phi_h \|_{L^2\II}}{\|   \phi_h  \|_{L^2\II}} + ch \le ch.
\end{align*}
This completes the proof of the lemma.
\end{proof}

Now we define the operator $K_{h,\tau}: \Q \rightarrow \Q$ such that
\begin{equation}\label{eqn:Kh}
 K_{h,\tau} q (x) :=  P_{[0,M_1]}\Big(\frac{f(x)-\bar \partial_\tau^\alpha u_h^N(x;q) + \psi_h(x)  }{g_\delta(x)} \Big),
\end{equation}
where the function $P_{[0,M_1]}:\mathbb{R} \rightarrow \mathbb{R}$ denotes a truncation function defined by
\begin{equation}\label{eqn:P0M1-fully}
P_{[0,M_1]} (a) :=   \max(\min( M_1, a ),0).
\end{equation}

The next lemma shows a contraction property of the operator $K_{h,\tau}$.

\begin{lemma}\label{lem:Dal-uhn-2}
Let $q_1,q_2 \in  \Q$. Then there holds for any positive $\epsilon < \min(1,2-\frac{d}{2}) $,
$$\| K_{h,\tau} q_1 - K_{h,\tau} q_2 \|_{L^2\II} \le c \max(T^{-\alpha},T^{-(1-\epsilon)\alpha}) \|  q_1 -  q_2 \|_{L^2\II} .$$
\end{lemma}
\begin{proof}
By the definition \eqref{eqn:Kh} and the property that $|P_{[0,M_1]} (a)  - P_{[0,M_1]} (b) | \le |a-b|$, there holds
\begin{align*}
|(K_{h,\tau} q_1 - K_{h,\tau} q_2)(x)| \le \Big|\frac{  \bar \partial_\tau^\alpha (u_h^N(x;q_2) - u_h^N(x;q_1) ) }{g_\delta(x)}\Big|
\le \frac{|  \bar \partial_\tau^\alpha (u_h^N(x;q_2) - u_h^N(x;q_1) ) |} {M_2 - \delta},
\end{align*}
where the second inequality follows from the facts that $g(x)=u(x,T)\ge M_2$ (Lemma \ref{lem:u-reg})
and $\|  g - g_\delta \|_{C(\overline\Omega)}= \delta$. Then Lemma \ref{lem:Dal-uhn} yields for any positive   $\epsilon < \min(1,2-\frac{d}{2}) $,
\begin{align*}
\|K_{h,\tau} q_1 - K_{h,\tau} q_2\|_{L^2\II} &\le c \| \bar \partial_\tau^\alpha (u_h^N(q_2) - u_h^N(q_1) ) \|_{L^2\II}\\
&\le c\max(T^{-\alpha},T^{-(1-\epsilon)\alpha}) \| q_1 - q_2 \|_{L^2\II}.
\end{align*}
This completes the proof of the lemma.
\end{proof}

Now we are ready to present the main theorem of this section.

\begin{theorem}\label{thm:err-fully}
Suppose that  Assumptions  \ref{ass:cond-1} and  \ref{assump:numerics} are valid.
Let $K_{h,\tau}$ be the operator defined in \eqref{eqn:Kh}. Then with sufficiently large $T$,
for any $q_0 \in  \Q$, the iteration
\begin{align}\label{eqn:iter-fully}
 q_{n+1} = K_{h,\tau} q_n,\qquad \forall~~ n=0,1,\ldots,
\end{align}
linearly converges to a unique fixed point $q^* \in L^\infty\II$  of $K_{h,\tau}$  with $0\le q^*\le M_1$ s.t.
\begin{align*}
\|   q^* - q_{n+1}  \|_{L^2\II}
\le  cT^{{-(1-\epsilon)\alpha}}
\| q^* - q_{n}  \|_{L^2\II}\qquad \text{for}~~ n\ge 0.  \end{align*}
 Moreover, there holds
\begin{align*} \| q^* - q^\dag \|_{L^2\II} \le  c \Big(\frac{\delta}{h^2} + h + \tau\Big), \end{align*}
\end{theorem}
where $q^\dag$ is the exact potential  and the constant $c$ is independent of $\tau$, $h$ and $\delta$.
\begin{proof}
Choosing an arbitrary initial guess $q_0 \in \Q$, the contraction mapping theorem and Lemma \ref{lem:Dal-uhn-2} (with sufficiently large terminal time $T$) imply that
the iteration \eqref{eqn:iter-fully} generates a Cauchy sequence $\{q_n\}_{n=1}^\infty$ in $L^2\II$ sense. Therefore the sequence $\{ q_n \}$ converges to a fixed point of  $K_{h,\tau}$  as $n\rightarrow\infty$, denoted by
$q^*\in L^2\II$. Then the use of the box restriction $P_{[0,M_1]}$ indicates  $0 \le q^* \le M_1$.

Next, we show the error estimate between $q^*$ and $q^\dag$. Since $q^\dag \in \Q$,  it holds that
\begin{align*}
\| q^\dag - q^* \|_{L^2\II} &\le \Big\|\frac{f-  \partial_t^\alpha u(T;q^\dag) + \Delta g}{g}
-  \frac{f- \bar \partial_\tau^\alpha u_h^N(q^*) +\psi_h  }{g_\delta}\Big\|_{L^2\II}\\
&\le \Big\|\frac{f - \partial_t^\alpha u( T;q^\dag) + \Delta g}{g}  - \frac{f - \partial_t^\alpha u(T;q^\dag) + \Delta g}{g_\delta}\Big\|_{L^2\II}\\
&\quad + \Big\| \frac{f - \partial_t^\alpha u(T;q^\dag) + \Delta g}{g_\delta} - \frac{f- \bar \partial_\tau^\alpha u_h^N(q^*) +\psi_h  }{g_\delta}\Big\|_{L^2\II}
 =: I + II.
\end{align*}
Due to the fact that $f(x),\ \partial_t^\alpha u(x, t;q^\dag), \ \Delta g \in L^2\II$, it is straightforward to see that the first term satisfies $I \le c\delta$.
So it suffices to establish a bound for $II$. First, we observe that for any positive $\epsilon < \min(1,2-\frac{d}{2}) $,
\begin{align*}
&\quad \| \partial_t^\alpha u(T;q^\dag) - \bar \partial_\tau^\alpha u_h^N(q^*) \|_{L^2\II} \\
& \le \| \partial_t^\alpha u(T;q^\dag) - \bar \partial_\tau^\alpha u^N(q^\dag) \|_{L^2\II}
+ \| \bar \partial_\tau^\alpha u^N(q^\dag) -  \bar \partial_\tau^\alpha u_h^N(q^\dag)  \|_{L^2\II}
+  \|  \bar \partial_\tau^\alpha u_h^N(q^\dag)  - \bar \partial_\tau^\alpha u_h^N(q^*) \|_{L^2\II} \\
& \le c (h^2  + \tau T^{-1}) T^{{-(1-\epsilon)\alpha}} +c T^{{-(1-\epsilon)\alpha}} \|  q^\dag - q^*  \|_{L^2\II},
\end{align*}
where for the last inequality we apply Lemmas \ref{lem:un-err}, \ref{lem:uhn-err} and \ref{lem:Dal-uhn}. This combined with Lemma \ref{lem:Dal-uhn-2} implies that
with $T$ away from $0$ there holds
\begin{align*}
 II \le  c \Big(\frac{\delta}{h^2} + h + \tau \Big) + c T^{{-(1-\epsilon)\alpha}} \|  q^\dag - q^*\|_{L^2\II}.
\end{align*}
Then we arrive at
\begin{align*}
\| q^\dag - q^* \|_{L^2\II} &\le   c_1 \Big(\frac{\delta}{h^2} + h + \tau \Big) + c_2 T^{{-(1-\epsilon)\alpha}} \|  q^\dag - q^* \|_{L^2\II}.
\end{align*}
Therefore, there exists a constant $T_0$ sufficiently large such that $ c_2 T_0^{{-(1-\epsilon)\alpha}} \le c_0 $ with some constant $c_0 \in(0,1)$ and
for any $T\ge T_0$ there holds
\begin{align*}
\| q^\dag - q^* \|_{L^2\II} &\le   \frac{c_1}{1- c_0} \Big(\frac{\delta}{h^2} + h + \tau \Big) \le c\Big(\frac{\delta}{h^2} + h+ \tau\Big) .
\end{align*}
This completes the proof of the theorem.
\end{proof}

\begin{remark}
The  error estimate in Theorem \ref{thm:err-fully} provides useful guidelines to choose discretization parameters $h$ and $\tau$ according to the  \textsl{a priori} known noise level $\delta$. For example, the choice $\tau, h = O(\delta^{\frac13})$ leads to the best convergence rate $O(\delta^{\frac13})$. This is fully supported by our numerical results in Section \ref{sec:numerics}.
\end{remark}

\section{Numerical experiments}\label{sec:numerics}
In this section, we present some one- and two-dimensional numerical results to illustrate the analysis.
The noisy data $g_\delta$ is generated by
$$g_\delta(x_i) = u(x_i,T)+\delta\zeta(x_i),$$
where $\zeta$ follows the standard Gaussian distribution and the points $x_i$ are grid points of a fine partition of $\Omega$.
Then to compute the numerical reconstruction $q^*$, we follow the idea in Section \ref{sec:fully} and design the iterative algorithm \ref{alg}.
 All the computations are carried out on a personal desktop with MATLAB 2021.
\begin{algorithm}
\SetKwInOut{Input}{input}\SetKwInOut{Output}{output}

\caption{An iterative algorithm for finding fixed point $q^*$ from $g_\delta$ }\label{alg}
 \KwData{Order $\alpha$, terminal time $T$, source term $f$, initial condition $v$, boundary data $b$,
 noisy observation $g_\delta$, upper bound constant $M_1$, discretization parameter $h$ and $\tau$;}
\KwResult{Approximate potential $q^*$.}
Compute $\psi_h$ by \eqref{eqn:psih};
set   $q_0  = P_{[0,M_1]}\Big[ \dfrac{f+\psi_h}{g_\delta}\Big]$, $k=0$ and  $e^0 = 1$\;
\While{$e^k>\text{tol}=10^{-10}$}{Compute $u_h^n(q_k)$, the fully discrete solution to \eqref{eqn:fully-i} with potential $q_k$\;
Update the potential by
$$
q_{k+1}  =  K_{h,\tau}q_k  = P_{[0,M_1]}\Big[\frac{f-\bar\partial_\tau^\al u_h^N(q_k)+\psi_h}{g_\delta} \Big];$$\\
Compute  error
$$
e^{k+1} = \|q_{k+1} - q_k\|_{L^2\II};
$$\\
$k\leftarrow k+1$\;
}
$q^*\leftarrow q_k$\;
\Output{The approximated potential $q^*$.}
\end{algorithm}

\subsection{One-dimensional examples}
To begin with, we consider the diffusion model in one dimension with $\Omega = (0,10)$. We set the problem data as
\begin{equation}\label{eqn:ion}
b(0) = b(10) = 1,\quad  v = \frac{1}{50}x(10-x)+1\quad \text{and}\quad f = 10.
\end{equation}
These data satisfy Assumption \ref{ass:cond-1}. In our experiments,
we test the following three (exact) potentials:
\begin{itemize}
\item[(1)] Smooth potential: $q^\dag_1 = 3+ \cos(0.6\pi x)$;
\item[(2)] Piecewise smooth potential: $q^\dag_2 = 4-tri(x)$, where $tri(x)$ is a triangle wave with value between $[0,1]$ and period $2$;
\item[(3)] Nonsmooth potential: $q^\dag_3 = 4-\chi_{[2,4]\cup[6,8]}$, where $\chi_D(x)$ denotes the characteristic function.
\end{itemize}
Note that $q^\dag_1$ and  $q^\dag_2$ satisfy Assumption \eqref{assump:numerics} (i), while $q_3^\dag \in H^{\frac{1}{2}-\epsilon}\II$
for any $\epsilon\in(0,1/2)$.

As we discussed in Section \ref{sec:fully},
we use the standard piecewise linear FEM with uniform mesh size $h$
for the space discretization, and the backward Euler (convolution quadrature) method with
uniform step size $\tau$ for the time discretization.
Since the closed form of exact solution is unavailable, we compute the exact observational data $g(x)=u(T;q^\dag)\approx u_h^N(q^\dag)$
by the fully discrete scheme \eqref{eqn:fully-i} with the fine meshes, i.e. $h=10^{-2}$, $\tau = 10^{-3}$.

For the \text{a priori} known noise level $\delta$, we choose the discretization parameters $h,\tau = O(\delta ^{1/3})$, and examine the relative error
\begin{equation}\label{eqn:eq}
e_q = \|q^\dag - q^*\|_{L^2\II}/\|q^\dag\|_{L^2\II},
\end{equation}
where $q^\dag$ is the exact potential and $q^*$ is the numerical reconstruction by Algorithm \ref{alg}.
Theorem \ref{thm:err-fully} concludes that Algorithm \ref{alg} produces a sequence $\{q_k\}$
linearly converging to a fixed point $q^*$, and the error satisfies $e_q = O(\delta ^{1/3})$.
In Figure \ref{fig:1D:sol-err} (a), (b) and (c), we present the profiles of exact potentials and reconstructed potentials with noise level $\delta=0.001$ and terminal time $T=1$.
We observe that our reconstructions  agree with the exact potentials.
In Figure \ref{fig:1D:sol-err} (d), (e) and (f), we plot the relative error $e_q$ defined by \eqref{eqn:eq} versus $\delta$, with different $\alpha$.
Our numerical results show that for the
smooth potentials $q^\dag_1$ and $q^\dag_2$,
the convergence rate is $O(\delta ^{1/3})$ which fully supports our theoretical results.
However, if the potential is discontinuous, the convergence rate is clearly less than order $1/3$ (cf. Figure \ref{fig:1D:sol-err} (f)).
 This illustrates the necessity of the Assumption on the smoothness of exact potential.
 Meanwhile, the experiments indicate that the error is robust with respect to the order $\alpha$.

\begin{figure}[htbp]
\centering
\begin{subfigure}{.32\textwidth}
\centering
\includegraphics[scale=0.34]{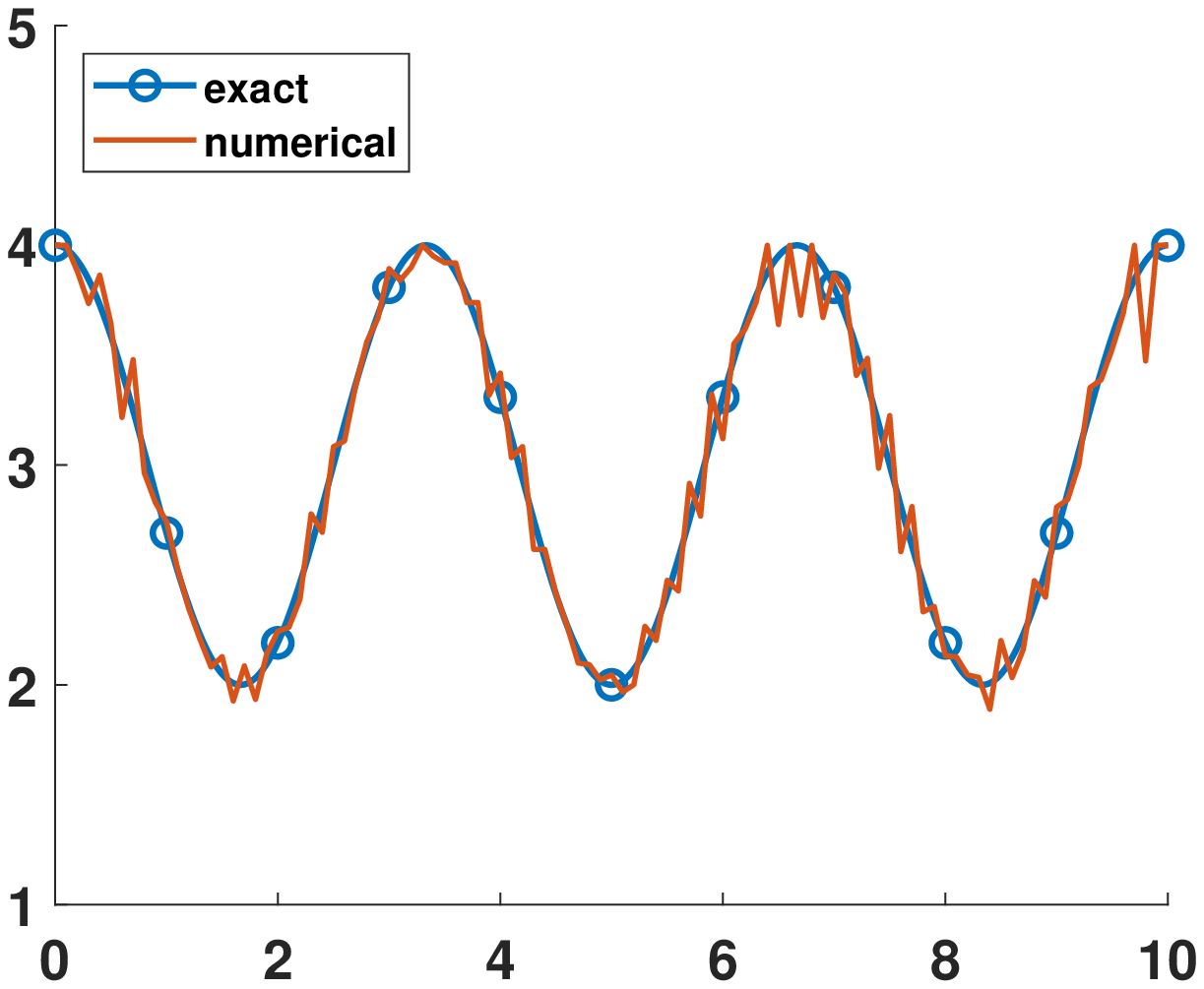}
\caption{Numerical reconstruction of $q^\dag_1$}
\end{subfigure}\hskip5pt
\begin{subfigure}{.32\textwidth}
\centering
\includegraphics[scale=0.34]{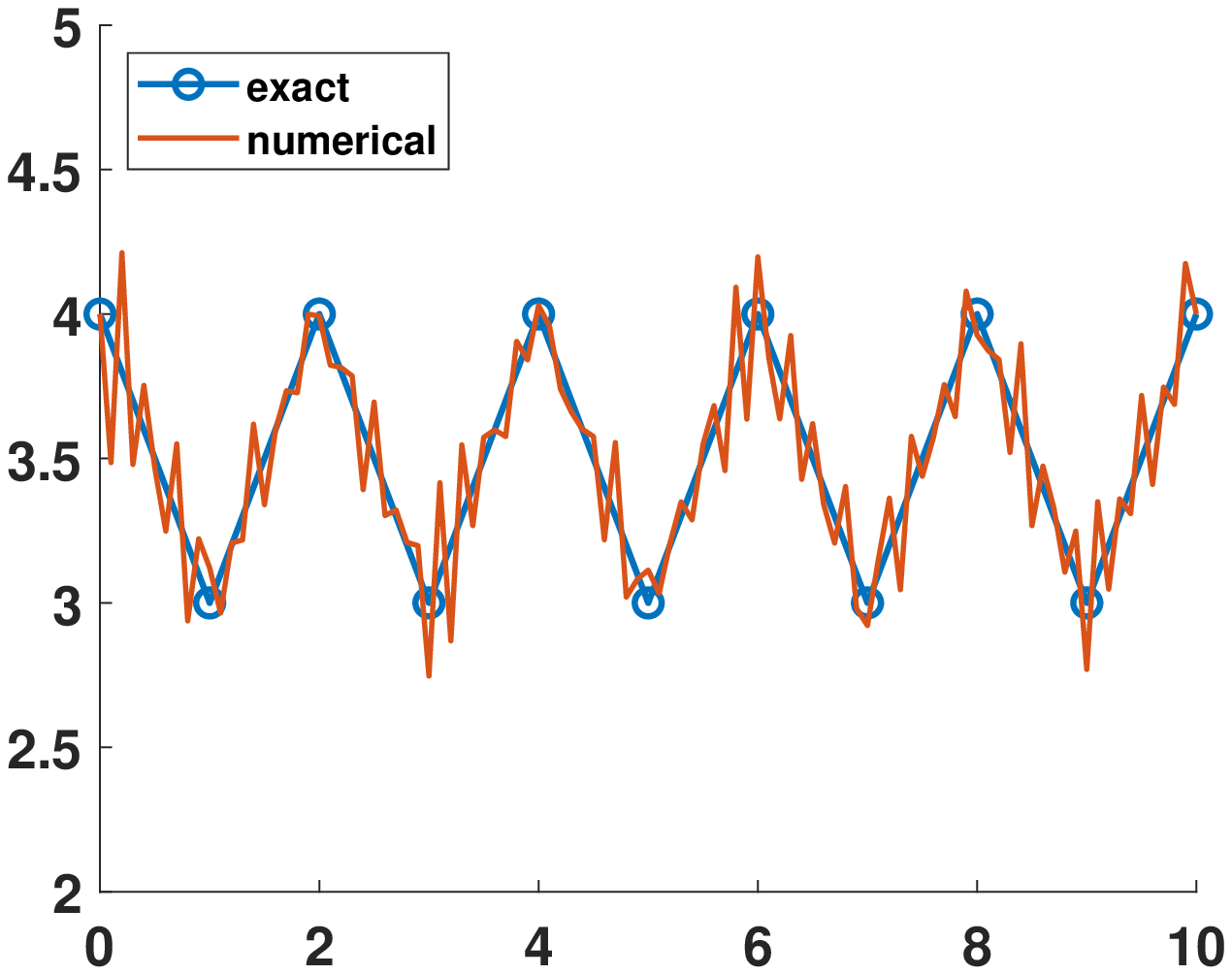}
\caption{Numerical reconstruction of $q^\dag_2$}
\end{subfigure}\hskip5pt
\begin{subfigure}{.32\textwidth}
\centering
\includegraphics[scale=0.34]{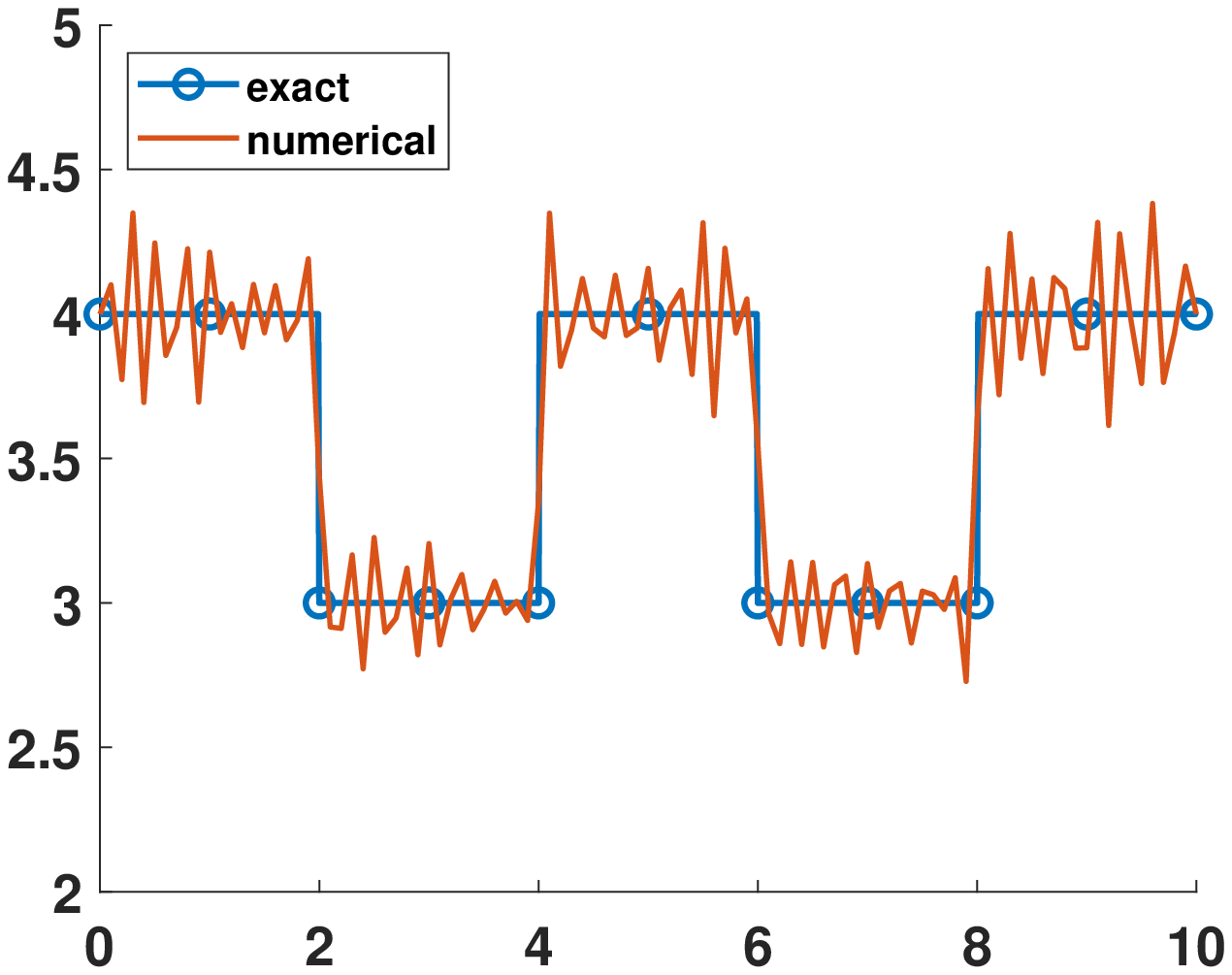}
\caption{Numerical reconstruction of $q^\dag_3$}
\end{subfigure}

\begin{subfigure}{.32\textwidth}
\centering
\includegraphics[scale=0.34]{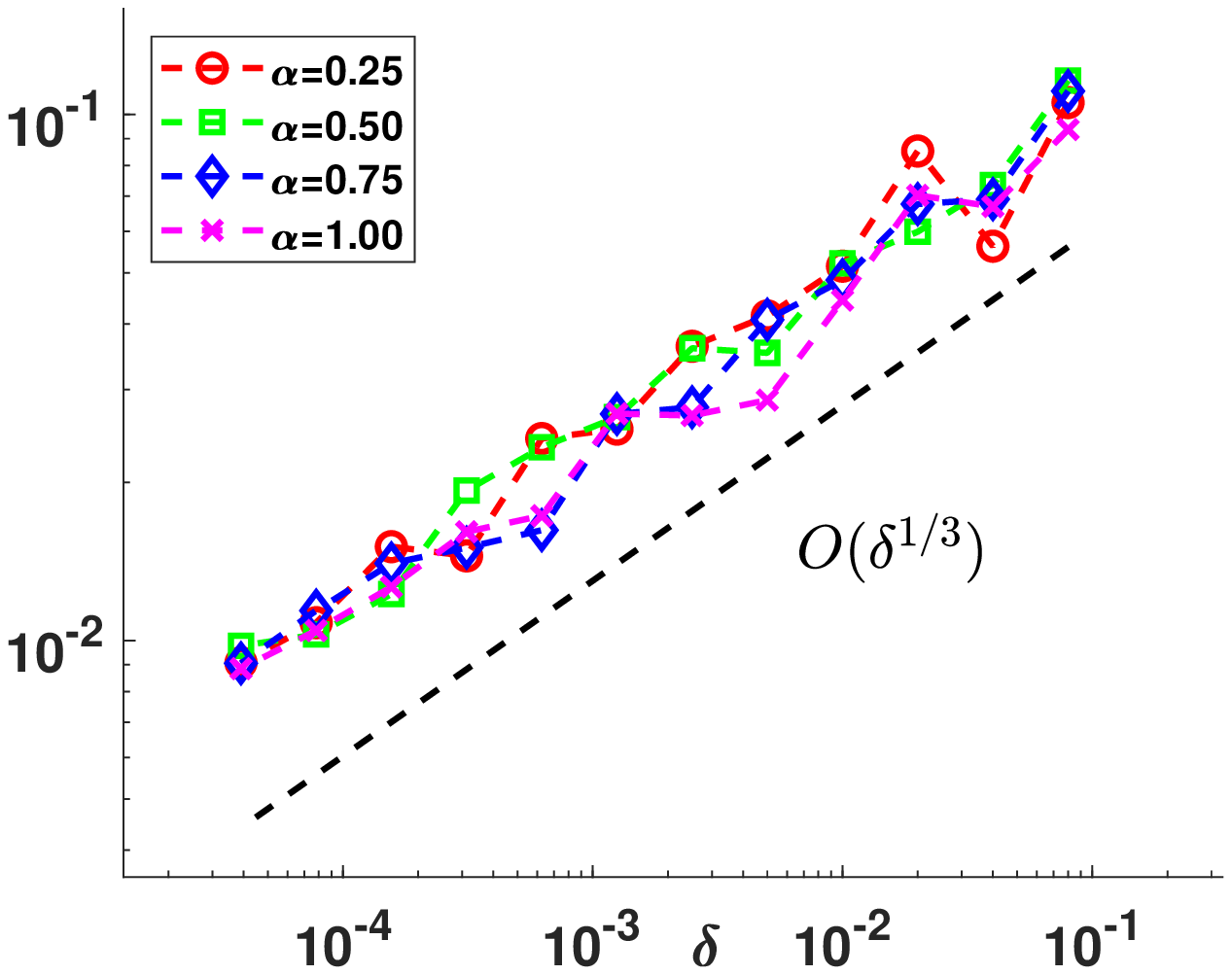}
\caption{Reconstruction error for $q^\dag_1$}
\end{subfigure}\hskip5pt
\begin{subfigure}{.32\textwidth}
\centering
\includegraphics[scale=0.34]{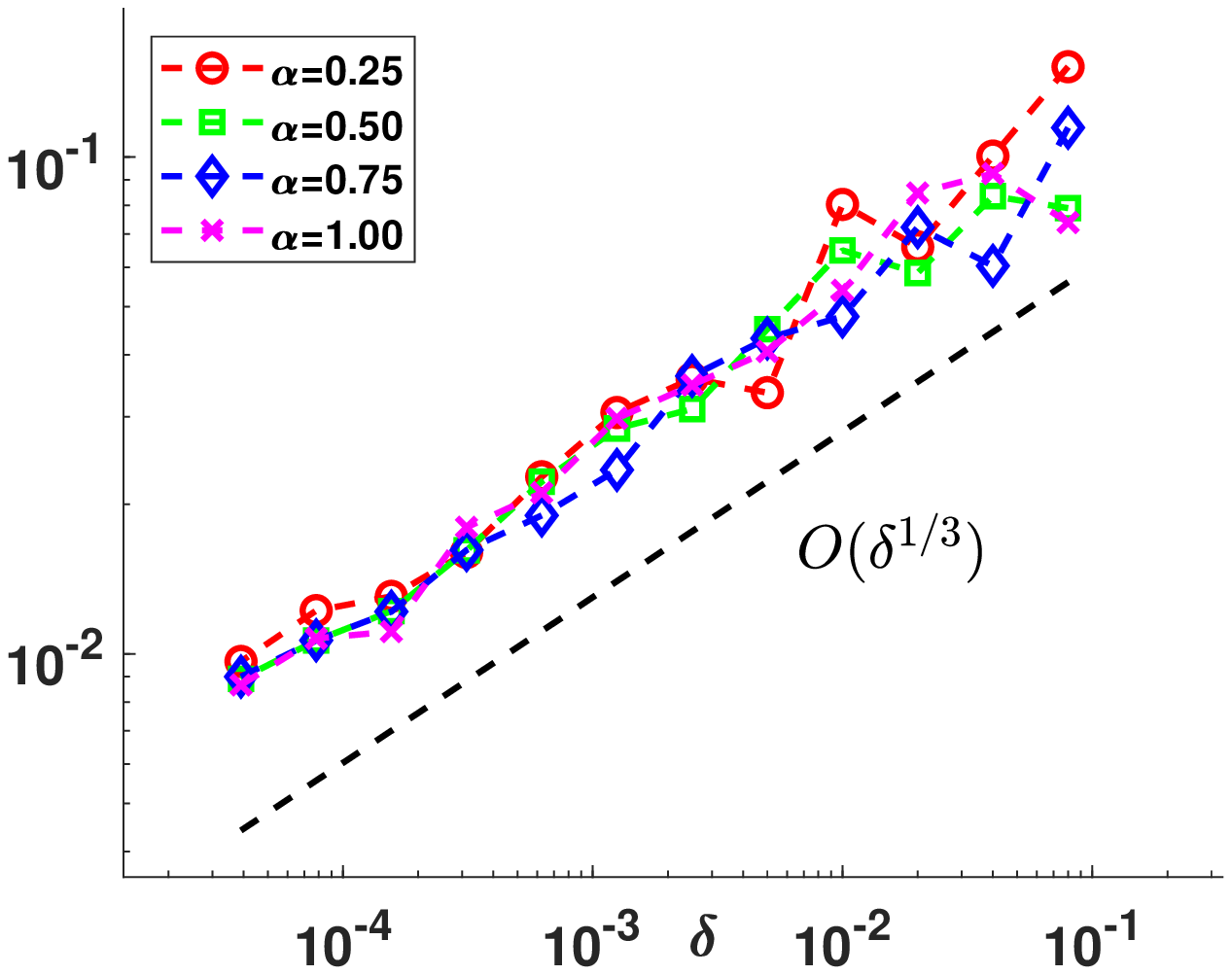}
\caption{Reconstruction error for  $q^\dag_2$}
\end{subfigure}\hskip5pt
\begin{subfigure}{.32\textwidth}
\centering
\includegraphics[scale=0.34]{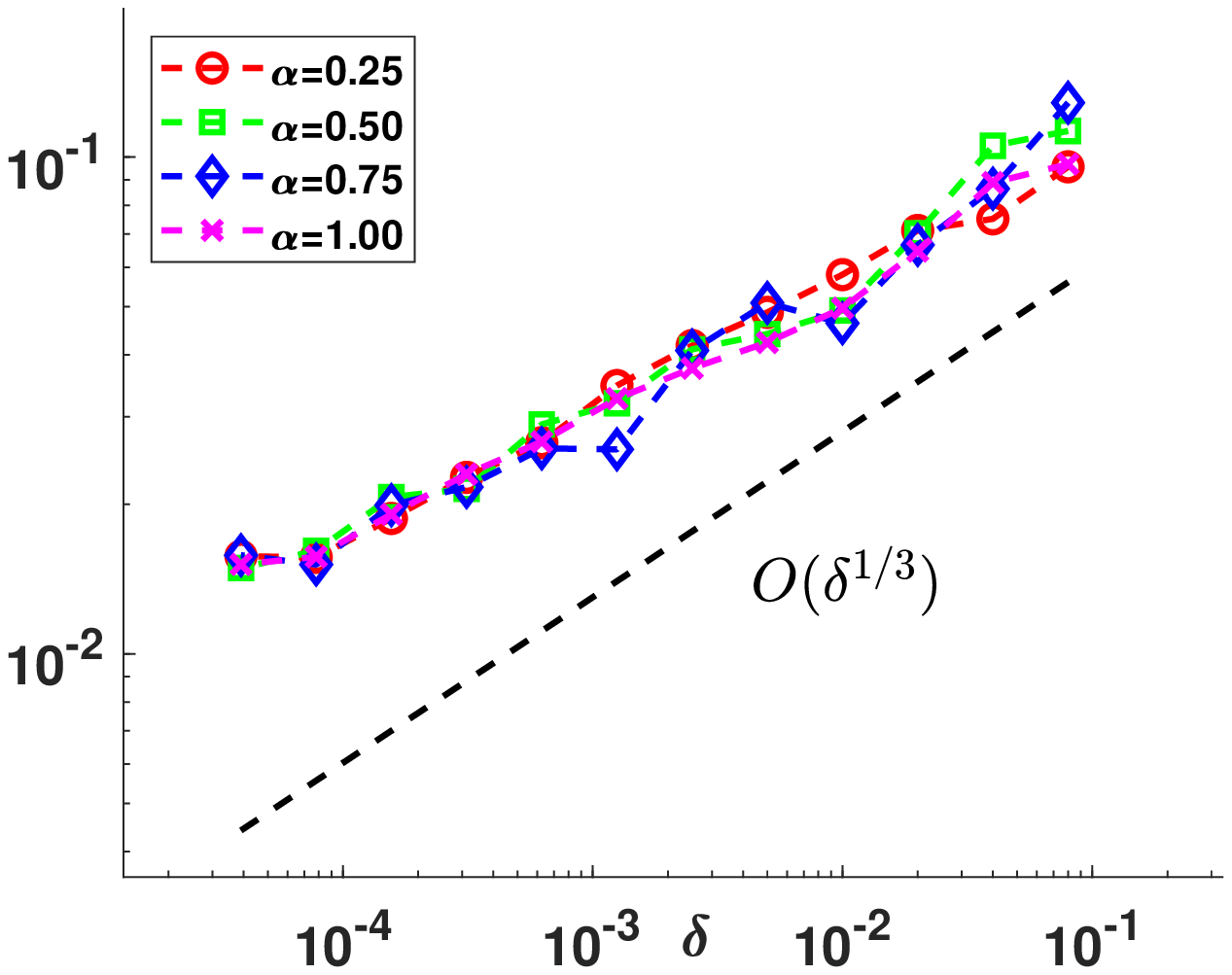}
\caption{Reconstruction error for  $q^\dag_3$}
\end{subfigure}%
\caption{(a), (b) and (c): Numerical reconstruction of potential with $\al = 0.5$, $\delta = 10^{-3}$, $h=0.1$, $\tau = 0.01$.
(d), (e) and (f): relative error $e_q$  versus noise level  $\delta$ with $h= \delta^{1/3}$, $\tau = \delta^{1/3}/10$ and $\al = 0.25,\,0.5,\,0.75,\,1$.}
\label{fig:1D:sol-err}
\end{figure}

Next, we test the convergence of the iteration in Algorithm \ref{alg}, with different $\alpha$ and $T$.
In the experiments, we use the problem data \eqref{eqn:ion} and the exact potential $q^\dag = q^\dag_2$.
Meanwhile, we fix $\delta = 10^{-6}$, $h = 10^{-2}$, $\tau = T/100$ and $q_0=4+x(1-x)/5$.
We let $q_k$ be the numerical solution obtained by $k$-th  iteration in Algorithm \ref{alg}, and compute the error at each iteration:
$$  e_k = \|q_k - q^\dag\|_{L^2\II}\qquad \text{for all}~~k\ge 0.$$
In Figure \ref{fig:1D:err-ite} (a) and (b), we report the convergence histories for $T=0.1$ and $T=2$ with different $\alpha$.
We clearly observe that the iteration converges linearly, and the convergence factor decreases as $T$ becomes larger.
Besides, the convergence appears to be robust to different fractional orders $\alpha$.
The errors stop at $8\times10^{-3}\approx \delta^{1/3}$, which agrees well with our theory.
Moreover, in Figure \ref{fig:1D:err-ite} (c), we test the convergence behavior for both large $T$ and small $T$.
Our experiments show that for small $T$, e.g. $T=10^{-2}$ and $10^{-3}$, the iteration does not converge to a reasonable
approximation to the exact potential.
In  Figure  \ref{fig:counter_eg1} (a) and (b), we plot the numerical reconstructions for $T=10^{-4}$ and $T=1$ respectively,
where we set $\delta=10^{-3}$, $h=0.1$ and $\tau = T/100$.
The numerical results show that Algorithm \ref{alg} produces an excellent reconstruction for $T=1$,
while the numerical reconstruction is inaccurate when $T$ is small.
This observation shows the necessity of the assumption in Theorems \ref{thm:cond-stab} and \ref{thm:err-fully}
that the terminal time $T$ should be sufficiently large.

\begin{figure}[htbp]
\centering
\begin{subfigure}{.33\textwidth}
\centering
\includegraphics[scale=0.33]{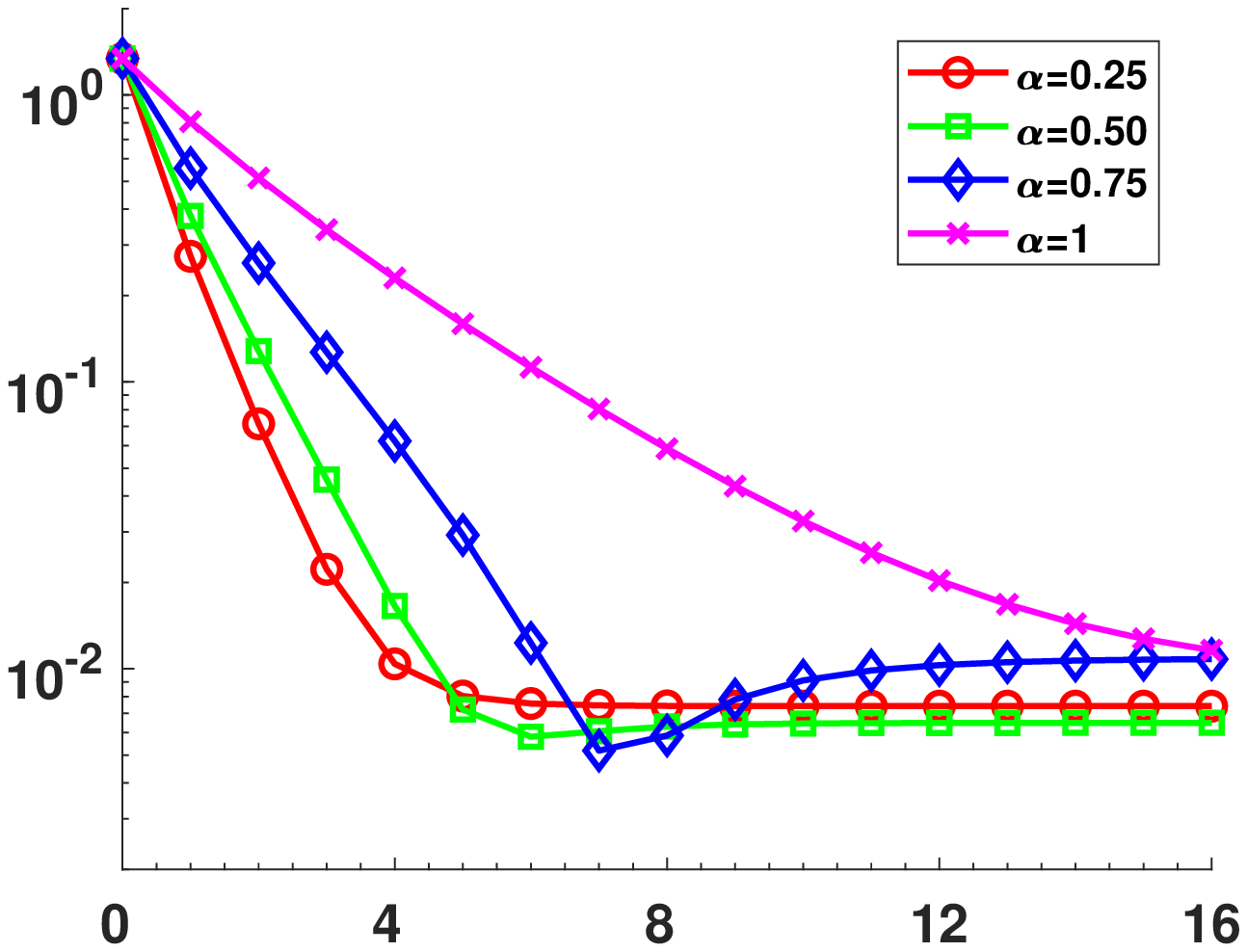}
\caption{$T=0.1$}
\end{subfigure}%
\begin{subfigure}{.33\textwidth}
\centering
\includegraphics[scale=0.33]{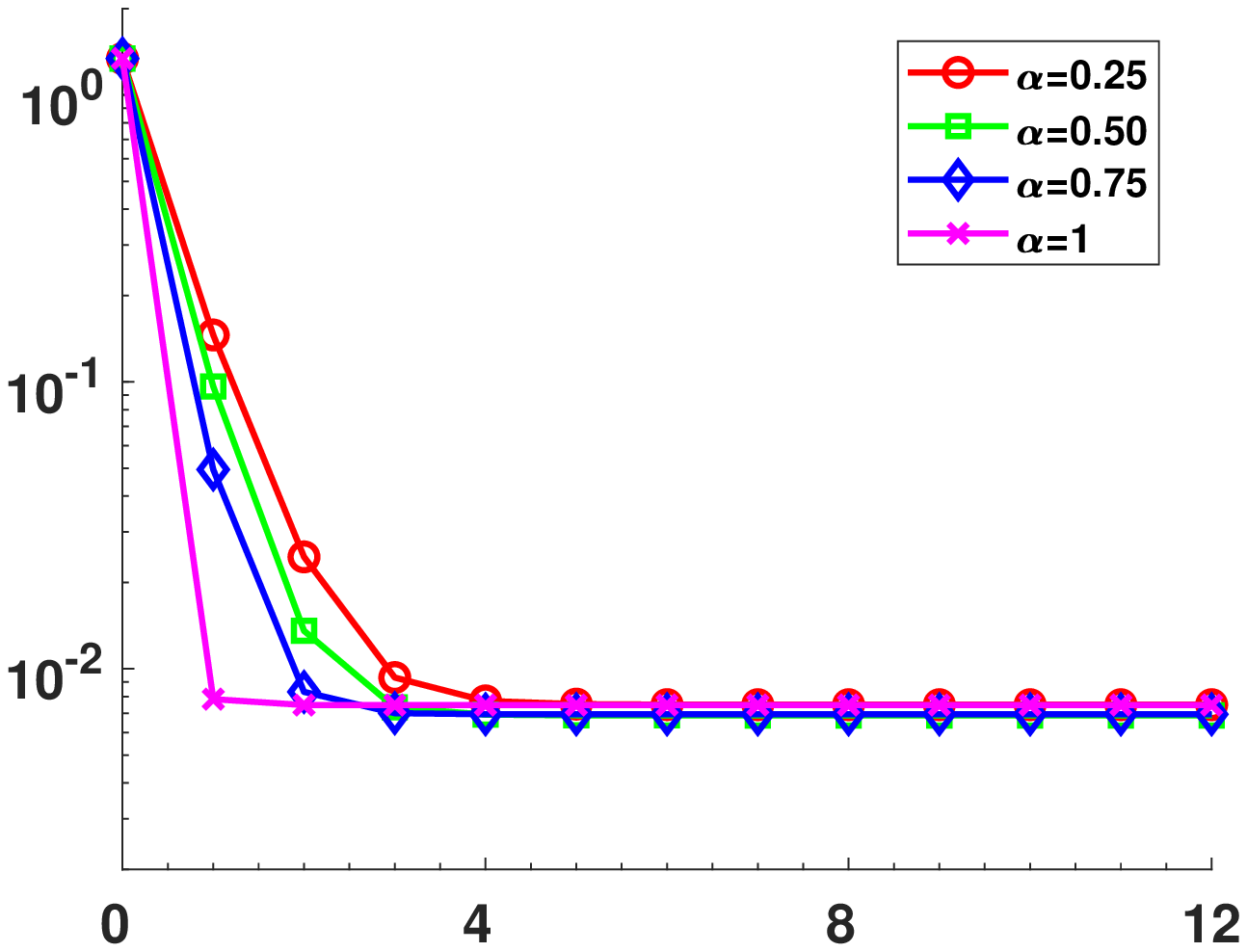}
\caption{$T=2$}
\end{subfigure}%
\begin{subfigure}{.33\textwidth}
\centering
\includegraphics[scale=0.33]{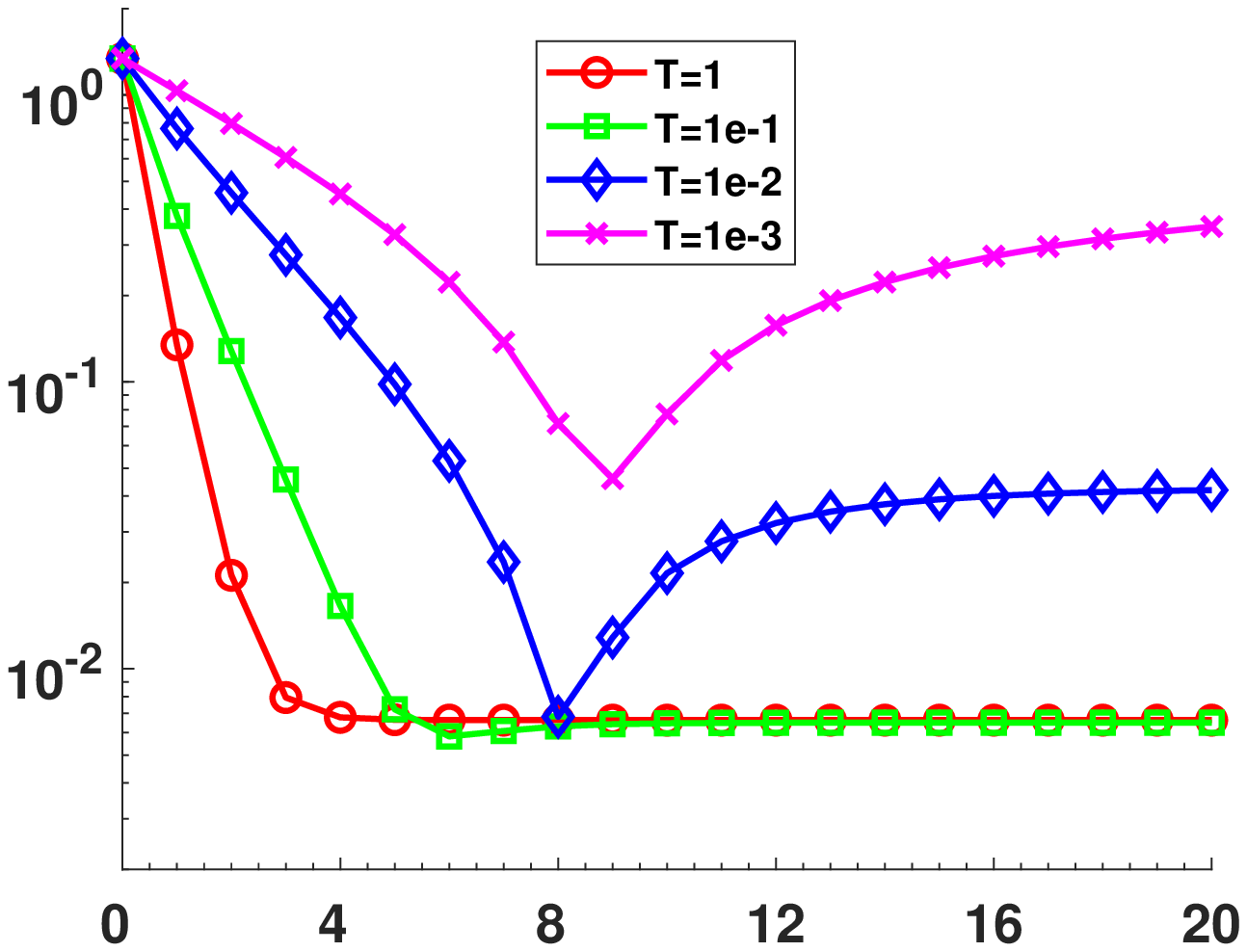}
\caption{$\alpha =0.5$}
\end{subfigure}%
\caption{Convergence histories of Algorithm \ref{alg} with different $T$ and $\alpha$.}\label{fig:1D:err-ite}
\end{figure}

\begin{figure}[htbp]
\centering
\begin{subfigure}{.45\textwidth}
\centering
\includegraphics[scale=0.45]{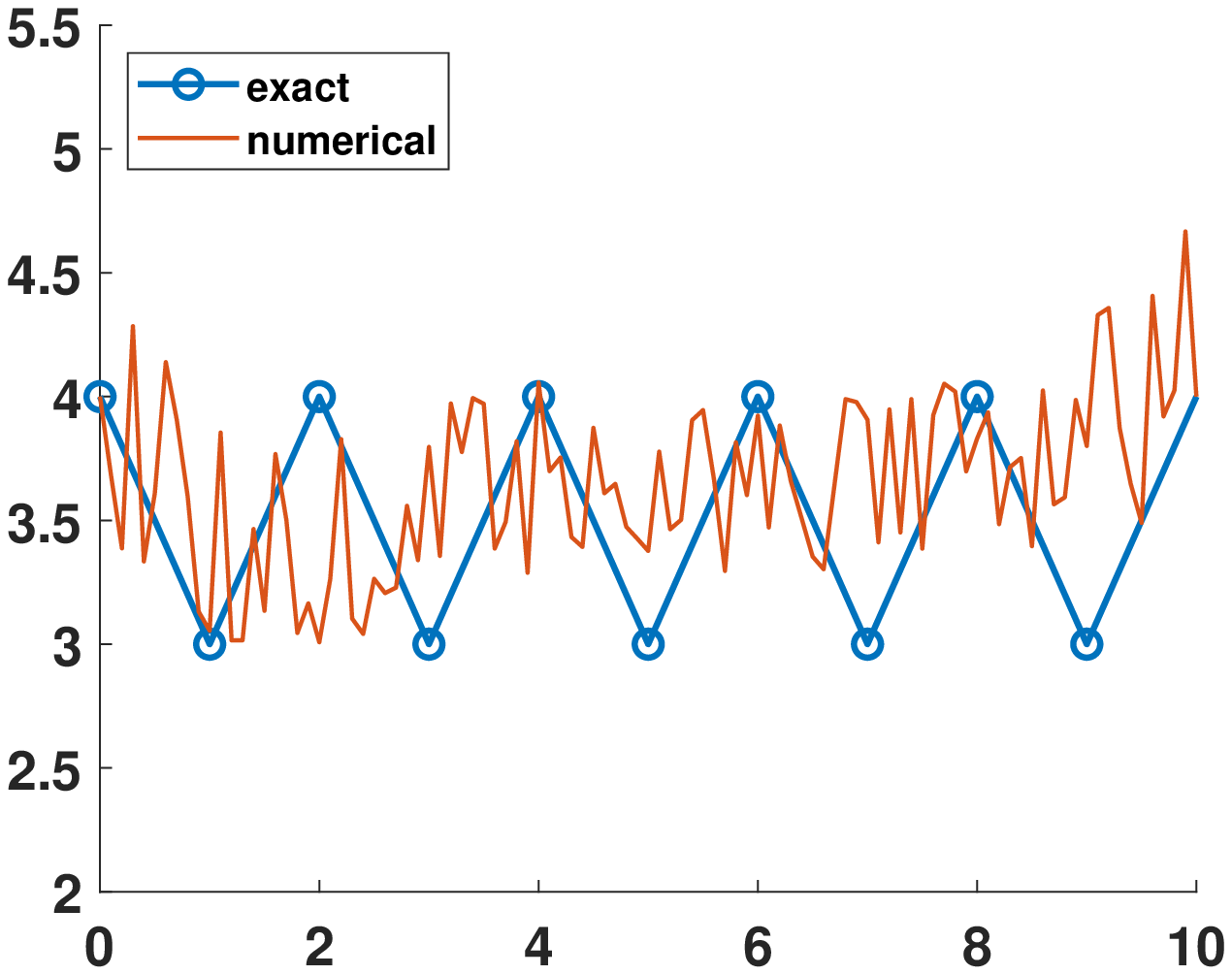}
\caption{$T= 10^{-4}$}
\end{subfigure}%
\begin{subfigure}{.45\textwidth}
\centering
\includegraphics[scale=0.45]{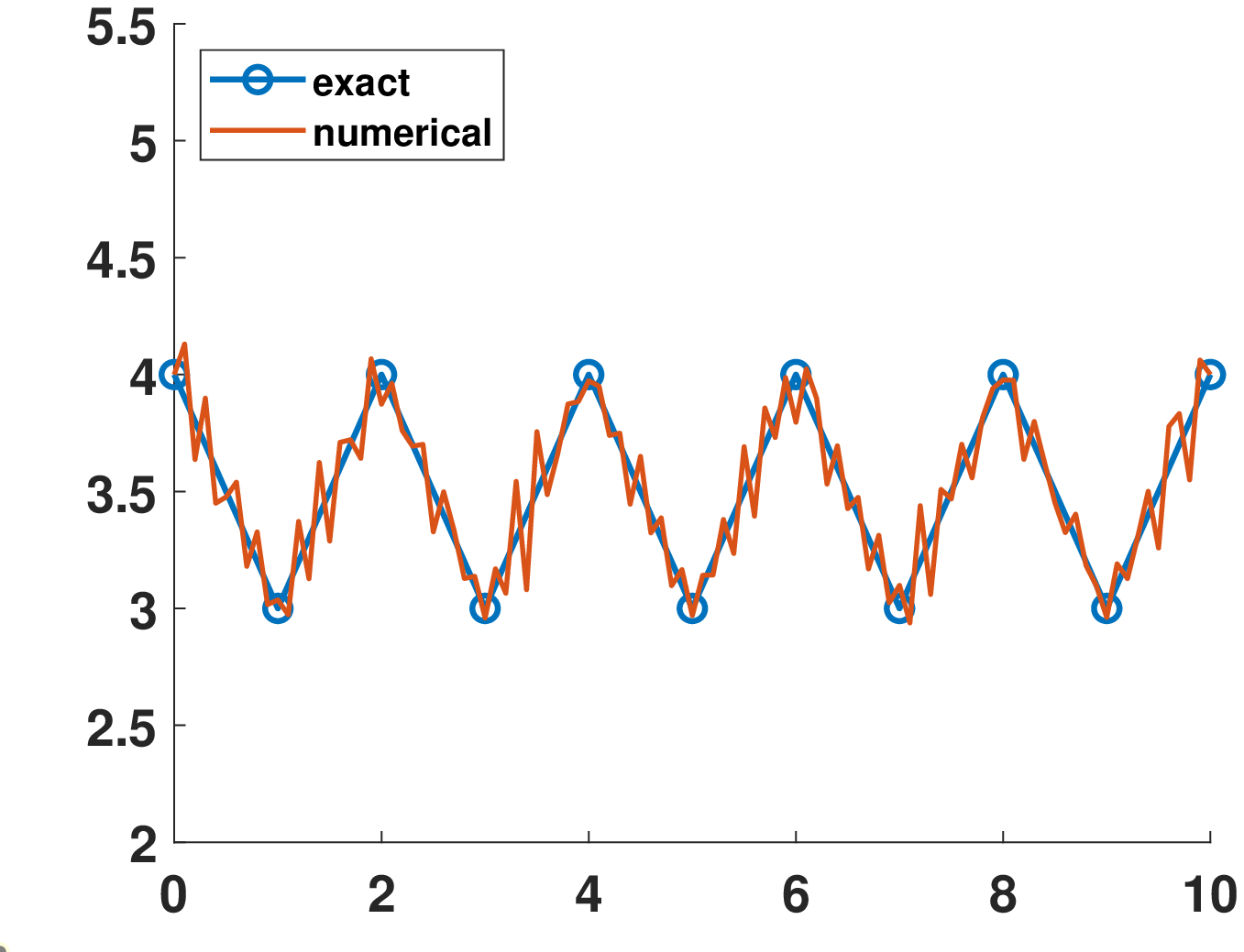}
\caption{$T=1$}
\end{subfigure}%
\caption{Plot of numerical reconstruction $q^*$. Left: $T= 10^{-4}$, $29819$ iterations and $\| q^{29819} - q^{29818} \|_{L^2\II} \le 10^{-10}$.
Right: $T=1$, $11$ iterations and $\| q^{11} - q^{10} \|_{L^2\II} \le 10^{-10}$. }
\label{fig:counter_eg1}
\end{figure}

\subsection{Examples in 2D}
Next, we present numerical experiments for a two-dimensional problem with the domain $(x,y)\in\Omega=(0,3)^2$
and the problem  data
\begin{equation*}
f(x,y) = 10,  ~~ b(x,y) =\frac{x(3-x)}{4}+1,~~ v(x,y) = x(3-x)\Big(\frac 14+\frac{y(3-y)}{10}\Big)+1.
\end{equation*}
Note that those data satisfy Assumption \ref{ass:cond-1}.
In the example, we test a smooth potential function
$$
q^\dag(x,y) = 3- \cos(\pi x)\cos(\pi y),\quad \text(x,y)\in\Omega,
$$
and set $M_1 = 5$.
The  exact observational data $g(x) = u(x,T)\approx U_h^N(x)$ is computed by the fully discrete scheme \eqref{eqn:CQ-BE}
with the spatial mesh size $h=0.01$ and time step size $\tau = T/10^4$.
In Figure \ref{fig:2D:err}, we report the reconstruction error \eqref{eqn:eq}  versus  noise level $\delta$,
where we set $h=\delta^{\frac13}$ and $\tau=\delta^{\frac{1}{3}}\times T/10$.
For $T=1$ and $T=5$, we clearly observe the convergence with rate $O(\delta^\frac13)$, cf. Figure \ref{fig:2D:err} (a) and (b).
Moreover, in case that $T$ is very small, our numerical results show that Algorithm \ref{alg} does not provide a good reconstruction $q^*$ for $\alpha = 0.50,\,0.75,\,1.0$, due to the loss of the stability (Theorem \ref{thm:cond-stab}), cf. Figure \ref{fig:2D:err} (c).
Interestingly, when $T=10^{-4}$, we still observe the convergence of optimal order $O(\delta^\frac13)$ for $\alpha=0.25$. This might be due to the faster decay of $\partial_t^\alpha u(t)$ for small $\alpha$ when $t$ is close to zero.
The exact reason still awaits further theoretical investigation.
See also Figure \ref{fig:2D:sol}
for an illustration of the reconstructions at different noise levels.
\begin{figure}[htbp]
\centering
\begin{subfigure}{.33\textwidth}
\centering
\includegraphics[scale=0.35]{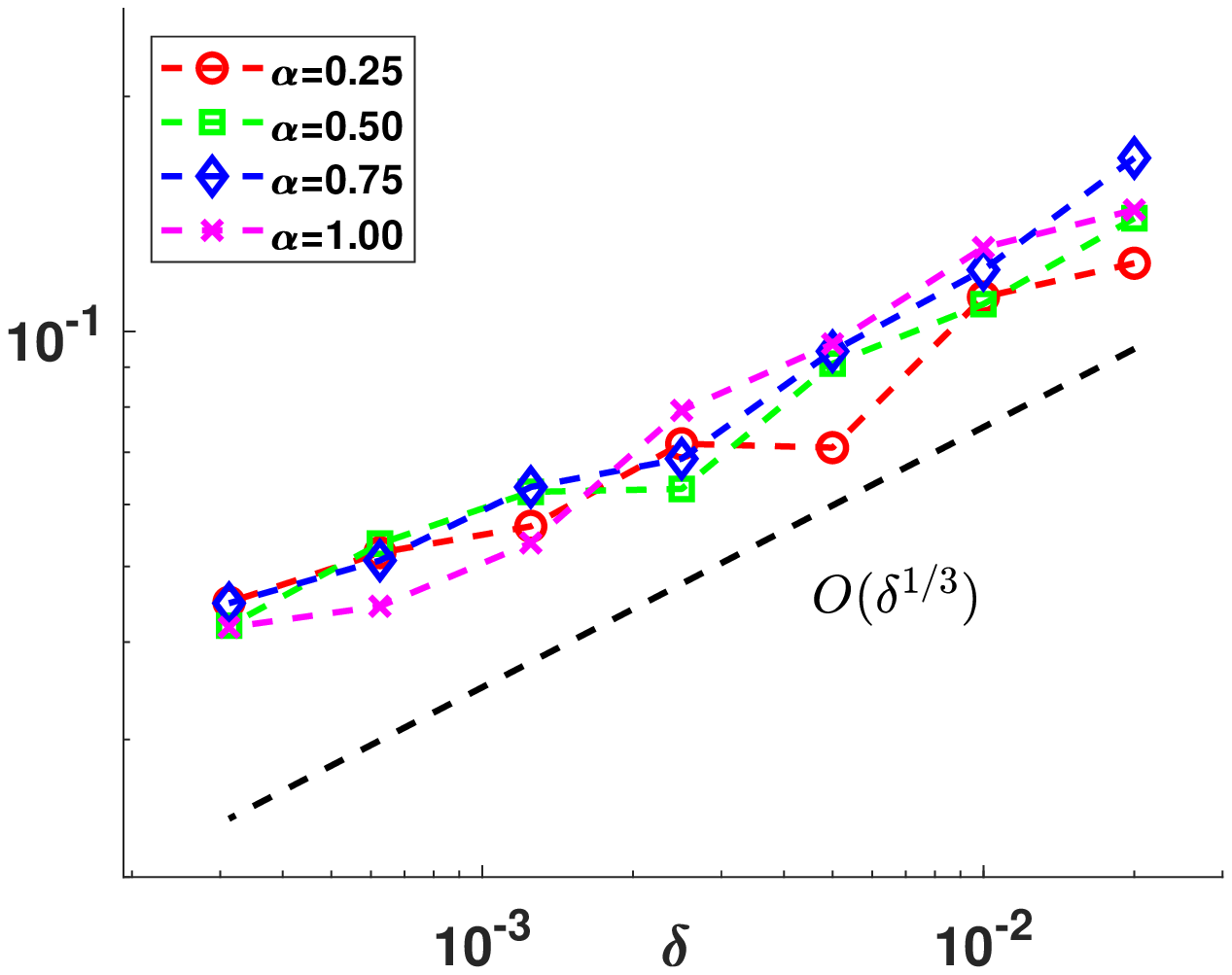}
\caption{$T=1$}
\end{subfigure}%
\begin{subfigure}{.33\textwidth}
\centering
\includegraphics[scale=0.35]{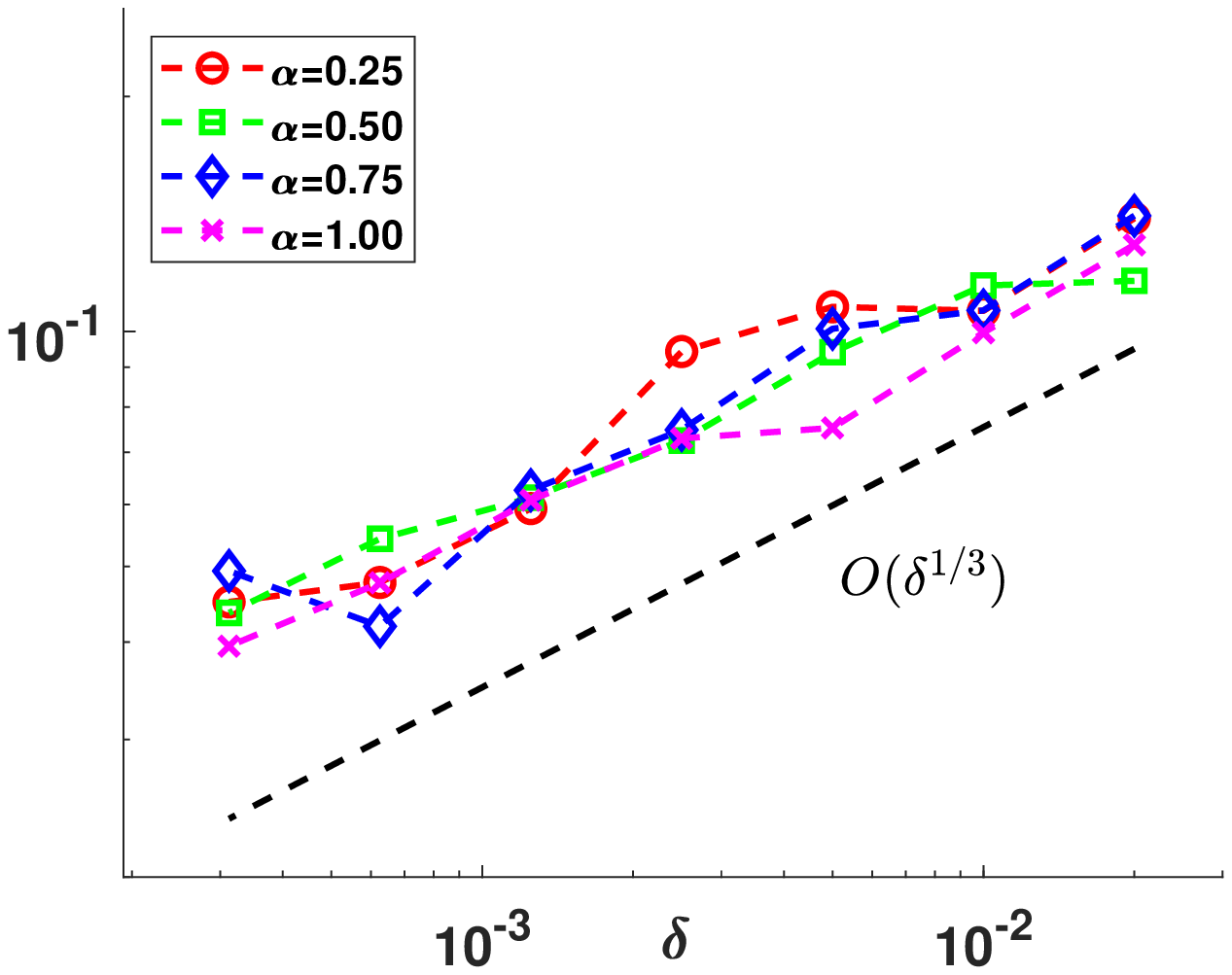}
\caption{$T=5$}
\end{subfigure}%
\begin{subfigure}{.33\textwidth}
\centering
\includegraphics[scale=0.35]{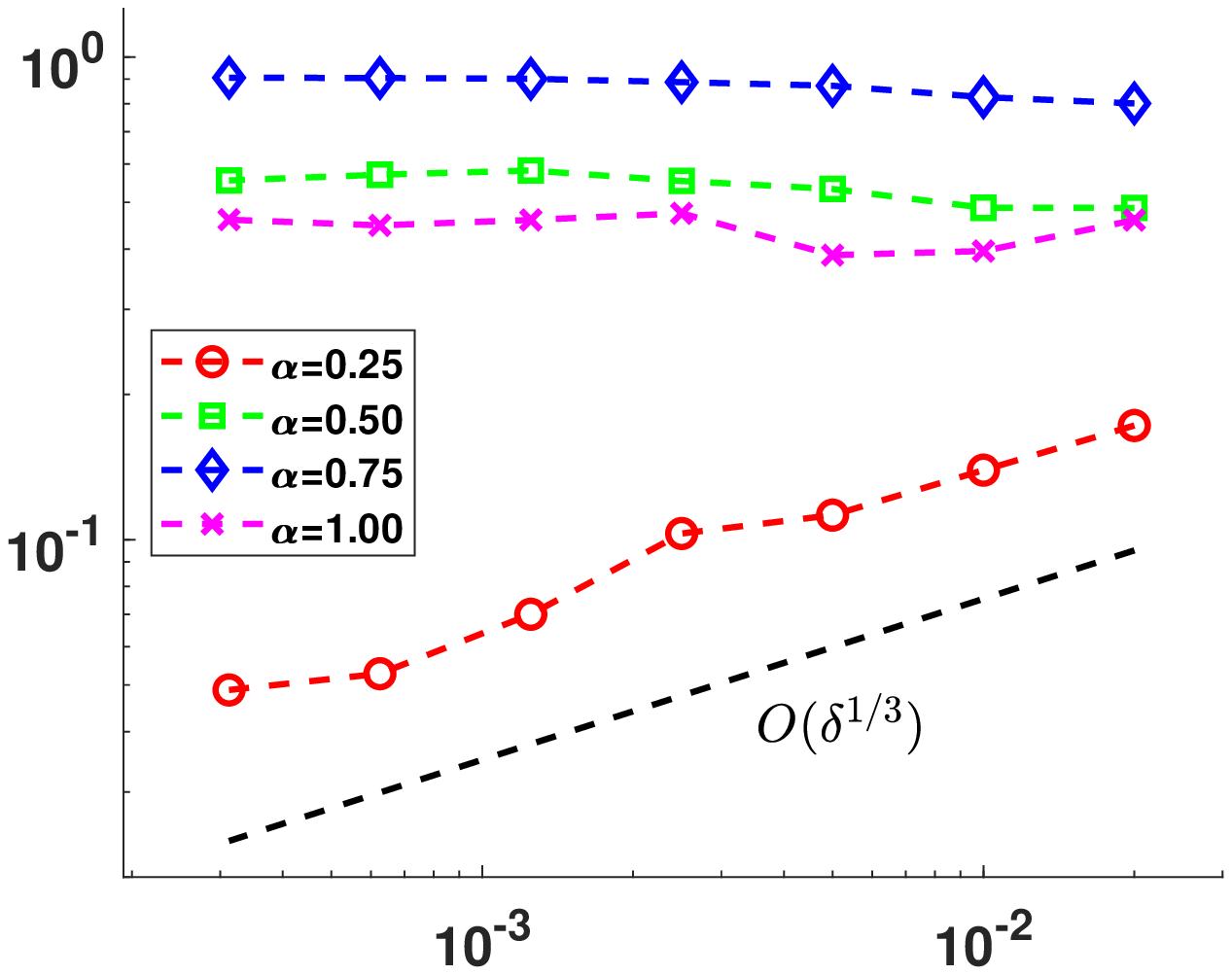}
\caption{$T = 10^{-4}$}
\end{subfigure}%
\caption{Relative error $e_q$ versus noise level $\delta$, where $h= \delta^{1/3}$,
$\tau = \delta^{1/3}/10$ and $\alpha=0.25, \, 0.5, \,0.75, \,1$.}\label{fig:2D:err}
\end{figure}

\begin{figure}[htbp]
\begin{subfigure}{.24\textwidth}
\centering
\includegraphics[scale=0.25]{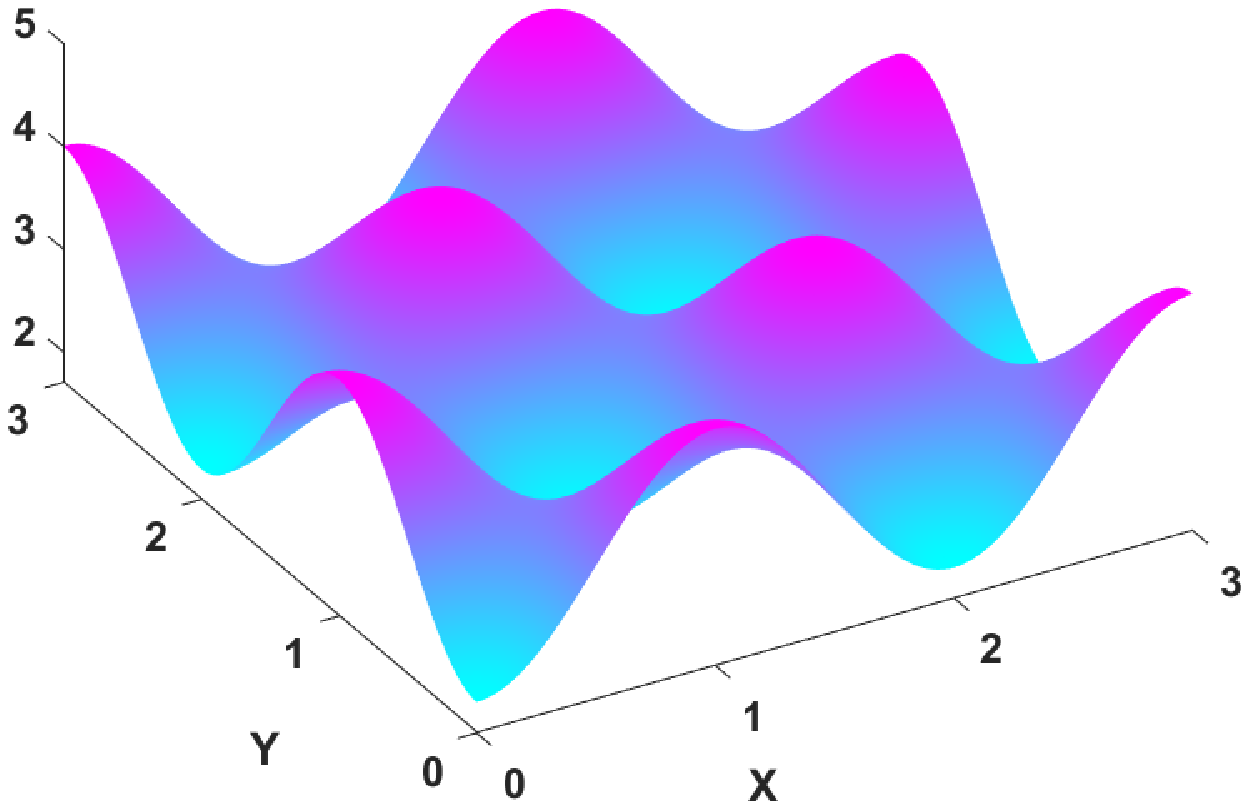}
\end{subfigure}%
\begin{subfigure}{.24\textwidth}
\centering
\includegraphics[scale=0.25]{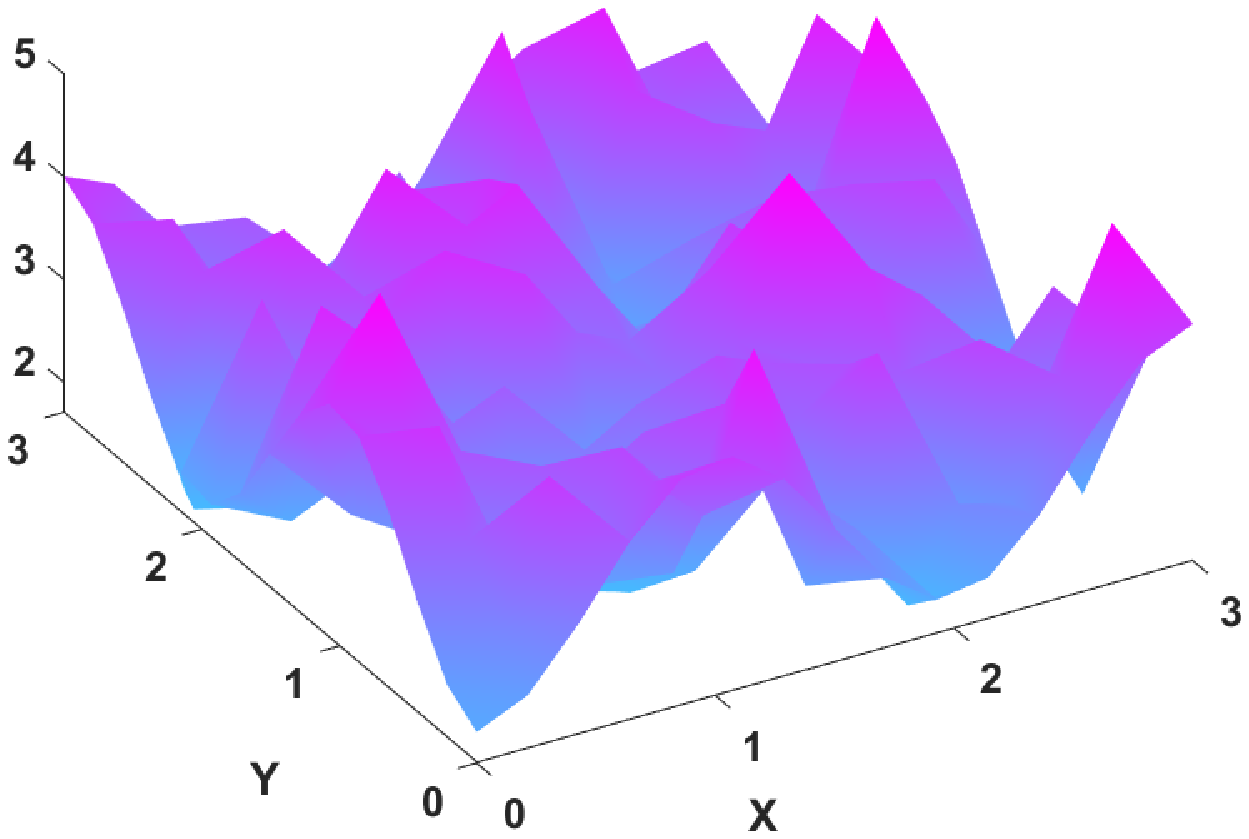}
\end{subfigure}%
\begin{subfigure}{.24\textwidth}
\centering
\includegraphics[scale=0.25]{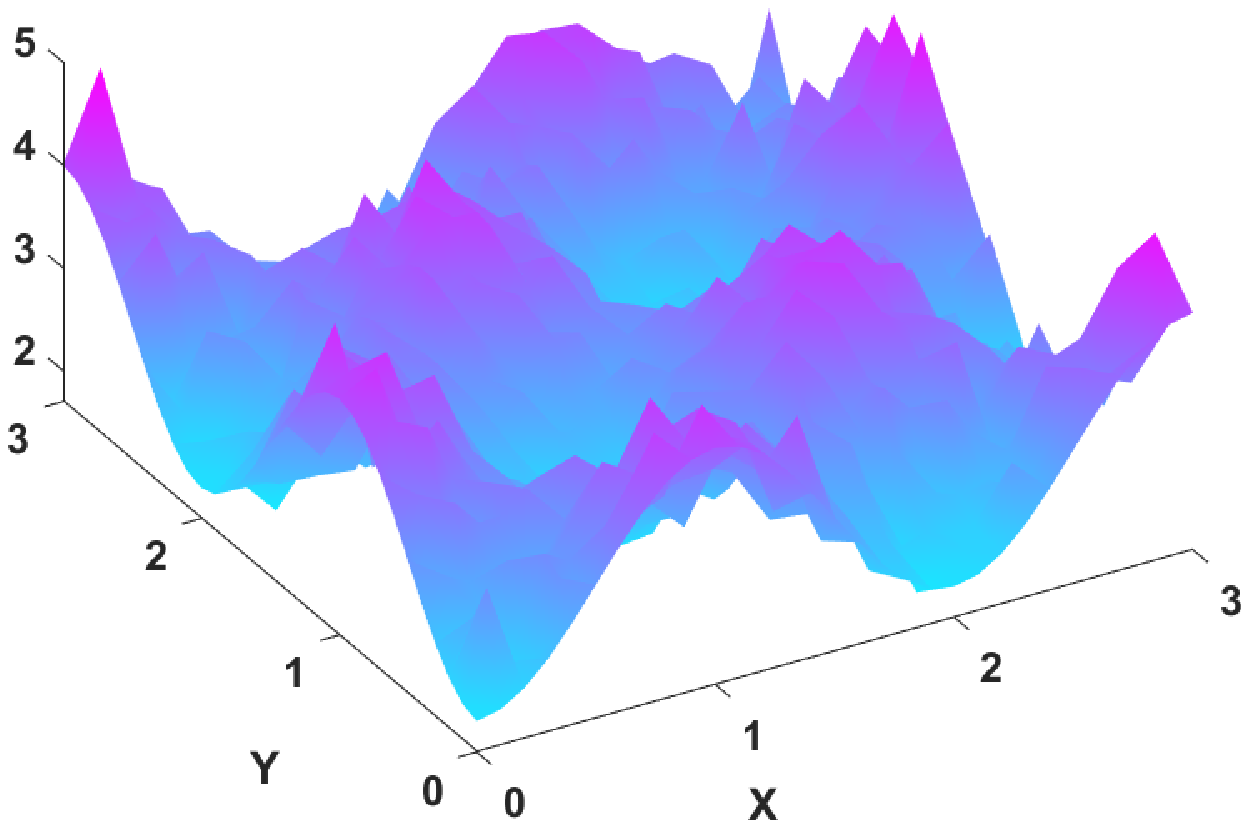}
\end{subfigure}
\begin{subfigure}{.24\textwidth}
\centering
\includegraphics[scale=0.25]{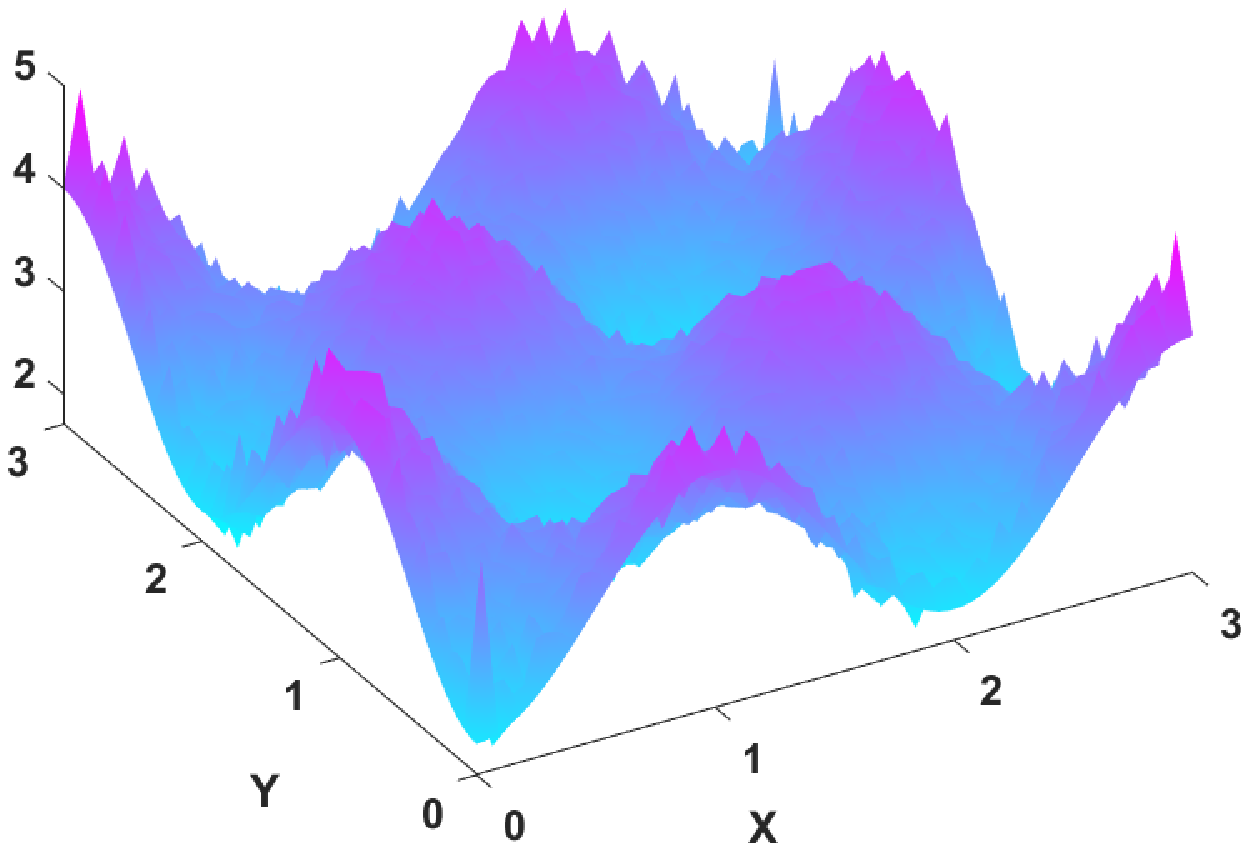}
x
\end{subfigure}
\newline
\raggedleft
\begin{subfigure}{.24\textwidth}
\centering
\includegraphics[scale=0.25]{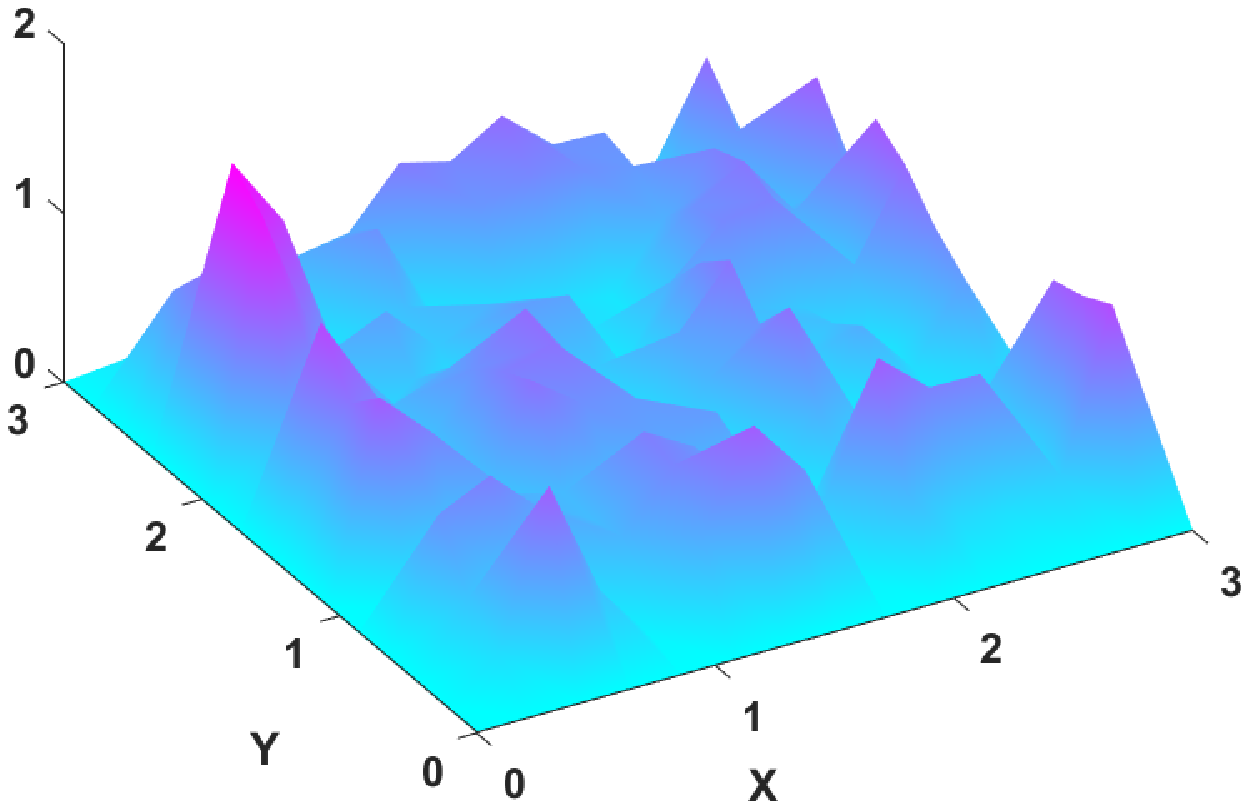}
\caption{$\delta =10^{-2}$}
\end{subfigure}%
\begin{subfigure}{.24\textwidth}
\centering
\includegraphics[scale=0.25]{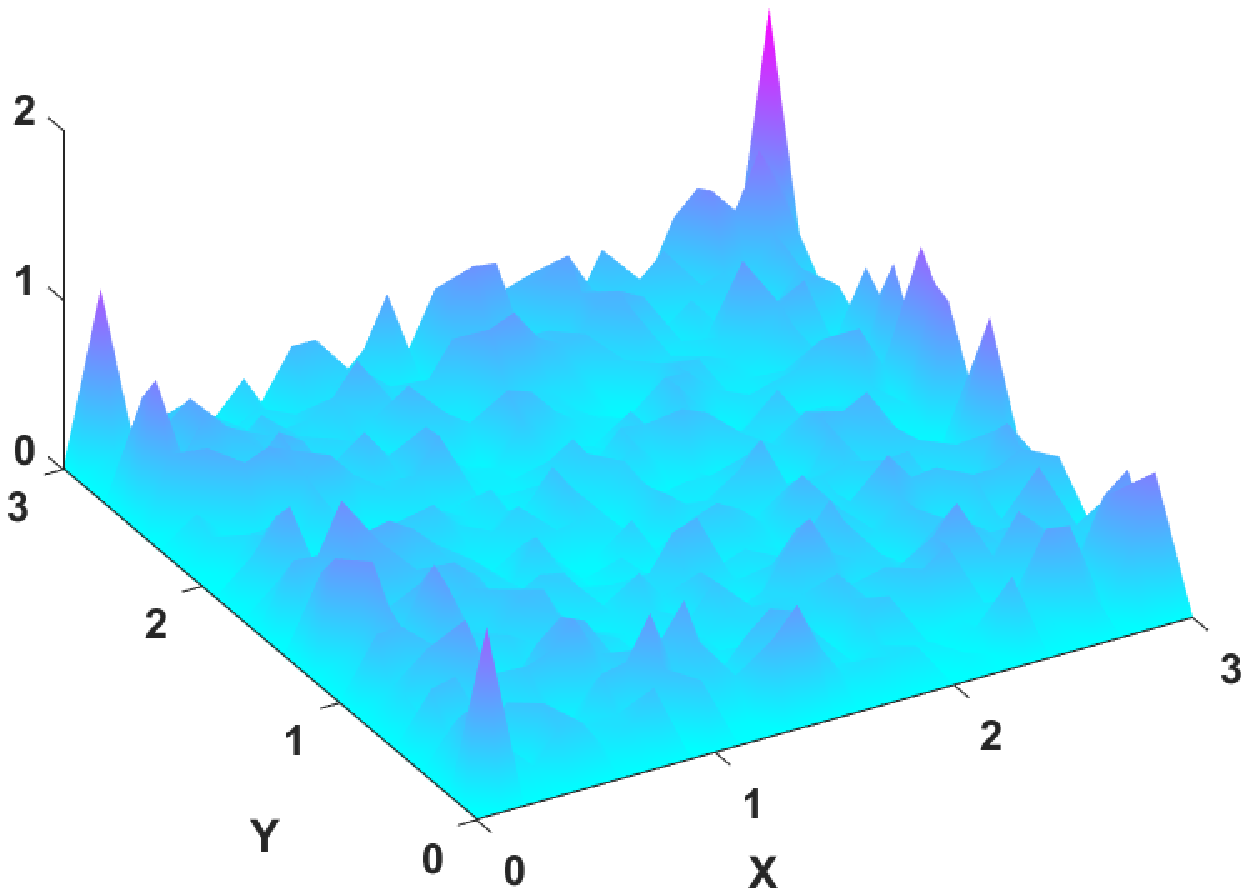}
\caption{$\delta =10^{-3}$}
\end{subfigure}
\begin{subfigure}{.24\textwidth}
\centering
\includegraphics[scale=0.25]{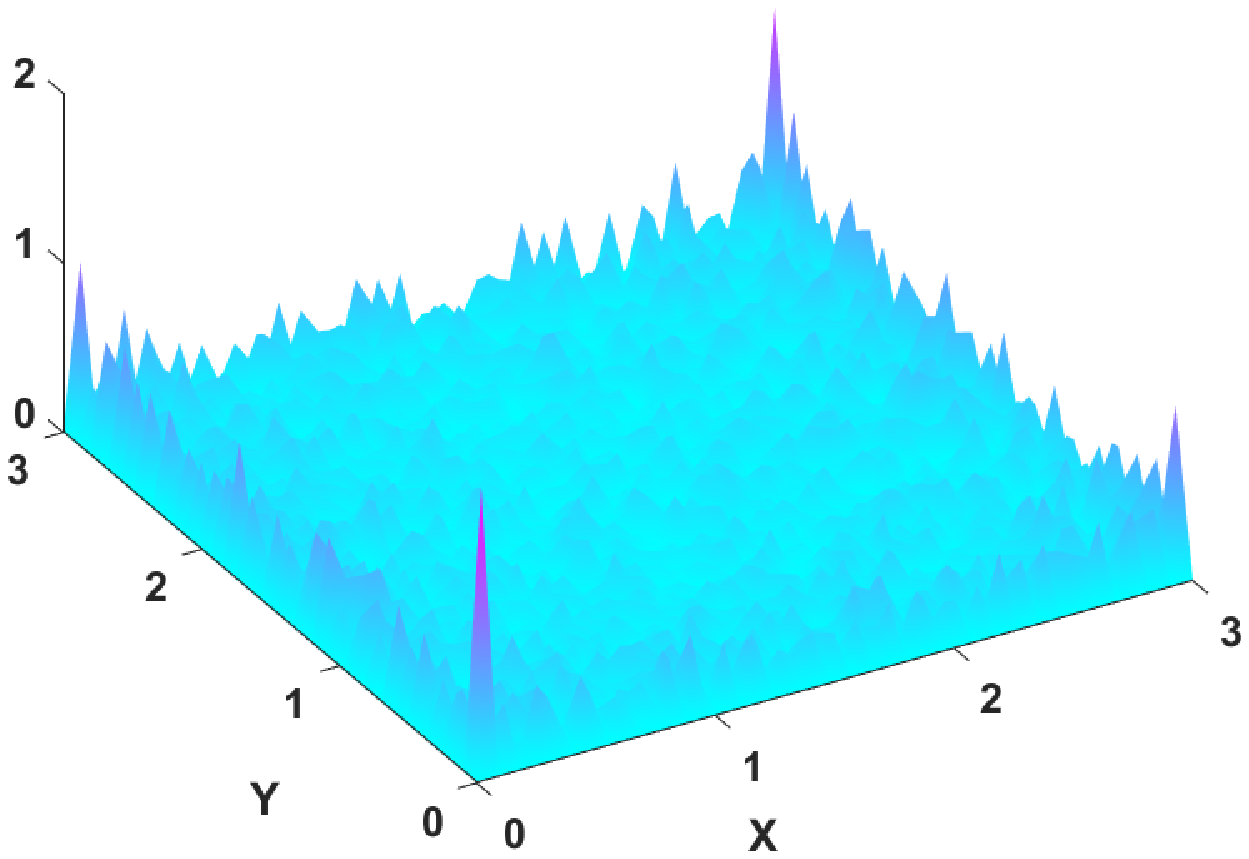}
\caption{$\delta= 10^{-4}$}
\end{subfigure}
\caption{Top left: Exact potential $q^\dag$. The other three columns are profiles of numerical reconstructions $q^*$ and  corresponding pointwise error $e=|q^*-q^\dag|$, with $T=1$, $\al=0.5$, $h=\delta^\frac 13$ and $\tau = \delta^\frac 13/10$.}
\label{fig:2D:sol}
\end{figure}

\bibliographystyle{abbrv}

\end{document}